\documentclass[11pt,a4paper,english,reqno]{amsart}
\usepackage{verbatim}
\usepackage{amsthm}
\usepackage{amscd}
\usepackage{amstext}
\usepackage{amssymb}
\usepackage[all]{xy}

\usepackage[dvips]{graphicx}
\usepackage[colorlinks=true,linkcolor=red,citecolor=blue]{hyperref}
\usepackage{tabu}
\usepackage{tikz}
\usepackage{mathdots}

\usepackage{scalerel}

\makeatletter
\@namedef{subjclassname@2010}{%
	\textup{2010} Mathematics Subject Classification}
\makeatother

\numberwithin{equation}{section}
\numberwithin{figure}{section}
\theoremstyle{plain}%%%%%%%%%%%%%%%%%%

\newtheorem{thm-nonum}{Theorem}

\theoremstyle{definition}%%%%%%%%%%%%%%

\theoremstyle{remark}%%%%%%%%%%%%%%%

\usepackage{amscd, amssymb,}
\xyoption{2cell}
\UseTwocells

\newcommand{\xyR}[1]{% 
	\xydef@\xymatrixrowsep@{#1}}

\newcommand{\xyC}[1]{% 
	\xydef@\xymatrixcolsep@{#1}}

\newcommand{\rt}{\rightarrow}
\newcommand{\lrt}{\longrightarrow}

\newcommand{\st}{\stackrel}

\newcommand{\la}{\lambda}
\newcommand{\La}{\Lambda}

\newcommand{\CA}{\mathcal{A} }
\newcommand{\CC}{\mathcal{C} }

\newcommand{\CF}{\mathcal{F} }

\newcommand{\CN}{\mathcal{N} }
\newcommand{\CP}{\mathcal{P} }
\newcommand{\CQ}{\mathcal{Q} }

\newcommand{\CS}{\mathcal{S} }

\newcommand{\CV}{\mathcal{V}}
\newcommand{\CU}{\mathcal{U}}

\newcommand{\CB}{\mathcal{B} }

\newcommand{\Mod}{{\rm{Mod\mbox{-}}}}

\newcommand{\mmod}{{\rm{{mod\mbox{-}}}}}

\newcommand{\Ker}{{\rm{Ker}}}

\newcommand{\Hom}{{\rm{Hom}}}

\theoremstyle{plain}
\newtheorem{theorem}{Theorem}[section]
\newtheorem{corollary}[theorem]{Corollary}
\newtheorem{lemma}[theorem]{Lemma}

\newtheorem{proposition}[theorem]{Proposition}

\theoremstyle{definition}
\newtheorem{definition}[theorem]{Definition}

\newtheorem{construction}[theorem]{Construction}

\newtheorem{remark}[theorem]{Remark}

\def\repr[#1;#2;#3;#4;#5]{
	\left(
	\begin{matrix}#1\\#2\end{matrix}
	#3
	\begin{matrix}#4\\#5\end{matrix}
	\right)}
\usepackage[all]{xy}
\numberwithin{equation}{section}
\numberwithin{figure}{section}

\newcommand\reallywidehat[1]{\arraycolsep=0pt\relax%
	\begin{array}{c}
		\stretchto{
			\scaleto{
				\scalerel*[\widthof{\ensuremath{#1}}]{\kern-.5pt\bigwedge\kern-.5pt}
				{\rule[-\textheight/2]{1ex}{\textheight}} %WIDTH-LIMITED BIG WEDGE
			}{\textheight} % 
		}{0.7ex}\\           % THIS SQUEEZES THE WEDGE TO 0.5ex HEIGHT
		#1\\                 % THIS STACKS THE WEDGE ATOP THE ARGUMENT
		\rule{-1ex}{0ex}
	\end{array}}

\begin{document}

\title[  Covering theory, (mono)morphism categories and  stable Auslander algebras]{Covering theory, (mono)morphism categories and  stable Auslander algebras}

\author[Rasool Hafezi and Elham Mahdavi  ]{Rasool Hafezi and Elham Mahdavi}
\dedicatory{Dedicated to the memory of  Andrzej Skowro{\'n}ski}
\address{School of Mathematics, Institute for Research in Fundamental Sciences (IPM), P.O.Box: 19395-5746, Tehran, Iran}
\email{hafezira@gmail.com, hafez@ipm.ir}
\address{ Department of Mathematics, University of Isfahan, P.O.Box:81746-73441, Isfahan, Iran}
\email{e.mahdavi@sci.ui.ac.ir}



%\thanks{This research was in part supported by a grant from IPM (No. 93130216)}

\begin{abstract}
Let $\CA$ be a locally bounded $k$-category  and $G$
a torsion-free  group of $k$-linear automorphisms of $\CA$ acting freely on the objects of $\CA,$ and  $F:\CA\rt \CB$ is a Galois functor. We extend  naturally the push-down  functor $F_{\la}$ to the functor $\rm{H}\rm{F}_{\la}:\rm{H}(\mmod \CA)\rt \rm{H}(\mmod \CB)$, resp. $\CS \rm{F}_{\la}:\CS(\mmod \CA)\rt \CS(\mmod \CB)$, between the corresponding morphism categories, resp. monomorphism categories,  of $\mmod \CA$ and $\mmod \CB$. Under some additional conditions, we show that $\rm{H}(\mmod \CA)$, resp. $\CS(\mmod \CA)$, is locally bounded if and only if $\rm{H}(\mmod \CB)$, resp. $\CS(\mmod \CB)$, is of finite representation type. As an application, we show that the stable Auslander algebra of a representation-finite  selfinjective algebra $\La$ is again representation-finite if and only if $\La$ is of Dynkin type  $\mathbb{A}_{n}$ with $n\leqslant 4$.
\end{abstract}

\maketitle
\section{Introduction}\label{Introduction}
Assume that $\mathcal{A}$ is an abelian category. The {\em morphism category} of $\mathcal{A}$ is an abelian  category $\rm{H}(\mathcal{A})$ defined by the following data.
The objects of $\rm{H}(\mathcal{A})$ are all   morphisms $f:X\rightarrow Y$ in $\mathcal{A}$. The morphisms from object $(X\st{f}\rightarrow Y)$ to $(X'\st{f'}\rightarrow Y')$ are pairs $(a,b)$, where $a:X\rightarrow X'$ and $b:Y\rightarrow Y'$ such that $b\circ f=f'\circ a$. The composition of morphisms is component-wise. We denote by $\CS(\mathcal{A})$ the full subcategory of $\rm{H}(\mathcal{A})$ consisting of all monomorphisms in $\mathcal{A}$, which is called the {\em monomorphism category of}  $\mathcal{A}$. 	Some authors use the terminology  submodule category of  $\mathcal{A}$, see e.g., \cite{RS2}.\\
Here, we focus on  $\CA:=\mmod \La$, the category of finitely generated (right) $\La$-module over finite dimensional algebra $\La$.
Let $T_2(\La):= \tiny {\left[\begin{array}{ll} \La & 0 \\ \La & \La \end{array} \right]}$ be the lower triangular $2\times 2$ matrix algebra over $\La$, and
recall that the category $\rm{H}(\mmod \La)$ is equivalent to the category of finitely generated right modules over the path algebra $\La A_2 \simeq T_2(\La)$, where $A_2:\bullet\rt \bullet$.
Up to this equivalence,  objects of $\rm{H}(\mmod \La)$ are the  finitely generated  right $T_2(\La)$-modules.

The categories $\rm{H}(\mmod \La)$ and $\CS(\mmod \La)$ are much more complicated than the underlying module categories $\mmod \La$. It is very far to be expected that $\rm{H}(\mmod \La)$  or  $\CS(\mmod \La)$ to be of finite representation type (having only finitely many indecomposable objects up to isomorphism), even in the simple case where $\La$ is a uniserial algebra. For instance, the authors in \cite{RS1} were shown that   for $n \leqslant 5$ the monomorphism category $\CS(\mmod k[x]/(x^n))$ is of  finite representation (see Corollary \ref{Corollary 7.8} for a generalization to selfinjective Nakayama algebras), and   while for $n\geqslant 7$ it is of  wild representation type,  where one cannot expect any classification of indecomposable objects. For the remaining case $n=6$ the authors gave a complete classification by using different methods in the representation theory, including Auslander-Reiten theory and covering theory developed by Gabriel and his school. Accordingly, such a  classification result demonstrates the complexity of $\rm{H}(\mmod \La)$. \\

The importance of the representation theory of $T_2(\La)$  for the representation theory of $\La$ was pointed out by M. Auslander \cite{Au4} and has since drawn the attention of several people, e.g. \cite{Ar2, LSk,  LS} and references therein. Such attention is mostly to determine when $T_2(\La)$ is of (tame, domestic and polynomial growth) finite representation type by using various methods, in particular covering techniques. Assume $\La$ is representation-finite. Let $\{M_1\cdots M_n\}$ be a complete  set of representatives of the isomorphism  classes of indecomposable $\La$-modules. Let $M=M_1\oplus\cdots \oplus M_n$ and $A=\rm{End}_{\La}(M)$. Then $A$ is called the {\it Auslander algebra} of $\La$. As discussed in \cite{Ar2} (see also Remark \ref{remark 6.3}), the algebra $A$ is representation-finite if and only if so is $T_2(\La)$. Therefore, the representation theory of $T_2(\La)$ and $A$ are very close to each other. A necessary and sufficient condition for $\La$ such that its Auslander algebra $A$ is of finite representation type is given in \cite{IPTZ} in terms of the universal covering of the Auslander-Reiten quiver of $\La.$\\
The study of the monomorphism category has been recently attracted more attention and studied in a systematic and deep  work by Ringel and Schmidmeier \cite{RS1, RS2}. See also \cite{LZ, XZZ, XZ} for a generalization of such works. The study of the submodule category goes back to Birkhoff \cite{Bi}, where he classified  the indecomposable objects of $\CS(\mmod \mathbb{Z}/(p^n))$ with $n\geqslant 2$ and $p$ a prime number, and see also \cite{Si} for the case that the ring is of the form $k[x]/(x^n)$, where $k$ is a field. It is worth noting that D. Kussin, H. Lenzing and H. Meltzer \cite{KLM} established a surprising link between the stable submodule category with the singularity theory via weighted projective lines of type $(2,3, p)$. \\
Ringel and Zhang \cite{RZ} presented a nice connection   between the submodule category $\CS(\mmod k[x]/(x^n))$  and the module category  of the preprojective algebra $\Pi_{n-1}$ of type $\mathbb{A}_{n-1}$ with the help of the  functor was constructed by Auslander and Reiten quite a long time ago, and  recently another one by Li and Zhang (see \cite{AHS} for a higher version of theses functors). Such a nice connection was generalized recently by the first named author \cite{H} between the $\CS(\mmod \La)$ and the module category over the stable Auslander algebra of $\La$, assuming $\La$ is representation-finite (see also \cite{E} for a generalization in the setting of selfinjective algebras). Recall the {\it stable Auslander algebra} of $\La$ is the quotient algebra $A/\CP$, where $A$ is the Auslander algebra of $\La$, defined in the above, and $\CP$ is the two-sided ideal of $A$ formed by all endomorphisms $M\rt M$ factor through a projective $\La$-module. Note that  the preprojective algebra $\Pi_{n-1}$ is the stable Auslander algebra of $k[x]/(x^n)$ (see \cite{DL},Theorem 3, Theorem 4 and Chapter 7).  In sum, analogue to the morphism categories, the study of the monomorphism categories is related to the stable Auslander algebras. In \cite{H}, a description of almost split sequences in $\CS(\mmod \La)$ via the ones of $\mmod A$ is given. A similar consideration also will be given in  Section \ref{Section 5} for the morphism category case. Such  descriptions  are vital for us in Section \ref{Section 5} to prove our main result in concern of the stability of almost split sequences under the covering functors.\\
Covering technique has been introduced into representation theory of algebras and developed in a series of papers by K. Bongartz, P. Gabriel , C. Riedtmann \cite{BG,G, Ri3} and, E.
Green, A. de la Pe\~a, Mart\'inez-Villa, et. al., \cite{Gr,MD}, and more recently by H. Asashiba and R. Bautista and S. Liu \cite{A,BL}. This technique reduces a problem  for modules over an algebra $\CA$ to that of a category $\CA$, often much simpler, with an action of  a group $G$ such that $\CA$ is equivalent to the orbit category $\CA/G$. \\
The aim of this paper is to formulate simultaneously the covering theory in general  for  morphism and monomorphism categories.

 	Let $\CA$ be a locally bounded $k$-catgery  and $G$
 a torsion-free  group of $k$-linear automorphisms of $\CA$ acting freely on the objects of $\CA,$ and  $F:\CA\rt \CB$ is a Galois functor, see Section \ref{Perliminary} for notion and notation relating to the covering theory. One of the most important result in the covering theory  is Gabriel's theorem \cite{G} which asserts:  $\CA$ is locally representation-finite (or  equivalently locally bounded) if and only if $\CB$ is so. Recall that  a locally finite dimensional category  $\CC$ is called  {locally bounded} if, for every $x \in \CC,$ there are only finitely many $y \in \CA$ such that $\CC(x, y)\neq 0$ or $\CC(y, x)\neq 0.$   In a natural way, we  extend the push-down functor $F_{\la}:\mmod \CA\rt \mmod \CB$ to  the functors $\rm{H}\rm{F}_{\la}:\rm{H}(\mmod \CA)\rt \rm{H}(\mmod \CB)$ and $\CS \rm{F}_{\la}:\CS(\mmod \CA)\rt \CS(\mmod \CB)$. We show that these functors play the role of the push-down functor of $F_{\la}$ for the relevant morphism and monomorphism categories. The constructions enable us  to prove the following theorem, that is a version of Gabriel's theorem in our desired setting.  

\begin{theorem}(see Theorem \ref{Theorem 6.5})\label{Theorem 1.1}
	Assume  further the action of $G$ on $\CA$ have only finitely many $G$-orbits and $\mmod \CB$ is of finite representation type and  connected.
	\begin{itemize}
		\item [$(i)$] The following conditions are equivalent.
		\begin{itemize}		
			\item [$(1)$] $\rm{H}(\mmod \CA)$ is locally bounded.		
			\item [$(2)$] $\rm{H}(\mmod \CB)$ is of finite representation type.			
		\end{itemize}
		\item [$(ii)$] The following conditions are equivalent.
		\begin{itemize}
			\item [$(1)$] $\CS(\mmod \CA)$ is locally bounded.			
			\item [$(2)$] $\CS(\mmod \CB)$ is of finite representation type.			
		\end{itemize}
	\end{itemize}
\end{theorem}
As an important application of the above theorem we prove that the stable Auslander algebra of a representation-finite  selfinjective algebra $\La$ is of finite representation type if and only if  $\La$ is of Dynkin type   $\mathbb{A}_{n}$ with $n\leqslant 4$. For more details,  see Section \ref{Section 7}.\\
We should remark that Z. Leszczy\'nski and A. Skowro{\'n}ski in \cite{LSk} showed when the triangular matrix $T_2(\La)$ is of  polynomial growth, then  there is a  Galois functor $F^2_{\la}:\mmod T_2(\hat{\La})\rt \mmod T_2(\La)$ which is induced by a Galois covering $F:\hat{\La}\rt \La$. Next, they applied the  Galois functor  $F^2_{\la}$ to classify for which the triangular matrix $T_2(\La)$ is of tame representation type.
What we follows in our work is to reformulate a covering theory for triangular matrix algebras in terms of functor categories and morphism categories. Then  we apply parallel such an approach to establish a covering theory for monomorphism categories. Although such a covering technique  for monomorphism categories have been used  for instance in \cite{RS1,Sc}  for some special cases. Our work provides a unified way of using covering methods for monomorphism categories.
 
The paper is organized as follows. Section \ref{Perliminary} contains an overview of the necessary background  on Galois
coverings of locally bounded $k$-categories and the associated pull-up and push-down functors. Section \ref{Subsection 3.1} is devoted to prove that the functors $\rm{H}\rm{F}_{\la}$ and $\CS \rm{F}_{\la}$ are $G$-precovering, in the sense of \cite{BL}. We begin Section \ref{Subsection 3.2} by recalling the definition of the functor $\Phi:\CF(\mmod\CA )\rt \CF(\mmod \CB)$ introduced  in \cite{P}.  Then the relationship between our functors  $\rm{H}\rm{F}_{\la}$ and $\CS \rm{F}_{\la}$ and $\Phi$ will be  investigated. Section \ref{Section 5} is the core of our paper where we will prove that the aforementioned functors preserve the almost split sequences. For the proof, the auxiliary functor $\Phi$ is vital. Sections \ref{Section 6} and \ref{Section 7} are devoted to prove Theorem \ref{Theorem 1.1} and its interesting application, respectively. We mention that Section \ref{Section 6} is indeed a direct consequence of our results in Section \ref{Section 5}, however for the convenience of the reader we will provide the necessary known backgrounds. 
\subsection*{Notations and  Conventions} To consider a map $A \st{f}\rt B$ as an object in $\rm{H}(\CA)$ we use the parentheses, and denote it
by $(A \st{f}\rt B)$. We also  sometimes  write 
$$
 \quad \begin{array}{c}
A \\
\mathrel{\mathop{\downarrow}^f}
\\
B
\end{array} $$
Especially, the first notation is used in diagrams. A morphism $(\phi_1, \phi_2): (A\st{f}\rt B)\rt (C\st{g}\rt D)$ in $\rm{H}(\CA)$ is also visualized as 
{\footnotesize  \[\xymatrix@R-2pc {  &  A\ar[dd]^{f}  & C\ar[dd]~  \\   &  _{\ \ \ \ \   }\ar[r]^{\phi_1}_{\phi_2}  \ar[r]&_{\ \  }~ {g}   \\ & B & D  }\]}
Let $\CC$ be a skeletally small (additive) category. The Hom-set between two objects $X$ and $Y$ in $\CC$ will be denoted by  $\CC(X, Y)$. By definition, a (right) $\CC$-module is a contravariant additive functor $H:\CC^{\rm{op}}\rt \CA b$, where $\CA b$ denotes the category of abelian groups. The $\CC$-modules and natural transformations between them, called $\CC$-homomorphism (we sometimes only say morphism), form an abelian category denoted by $\Mod \CC$. An $\CC$-module $H$ is called {\it finitely presented} if there exists an exact sequence $$\CC(-, C')\rt \CC(-, C)\rt H\rt 0 $$ with $C$ and $C'$ in $\CC$. Following \cite{P}, we denote by  $\CF(\CC)$ the category of all finitely presented functors over $\CC.$ The composition of two morphisms (resp. natural transformations) $f:X\rt Y$ and $g:Y\rt Z$ of  a given category (resp. between some categories)  is denoted by  $g\circ f$, or sometimes simply $gf.$

\section{Galois coverings of locally bounded $K$-categories}\label{Perliminary}

Throughout this paper $k$ denotes an algebraically closed field. The word algebra always means a basic finite dimensional $k$-algebra with identity, and all categories and functors are assumed to be $k$-category and  $k$-linear, respectively.
   A $k$-category is a category $\CA$ whose morphism-sets  $\CA(x, y)$, $x, y \in \CA$ are endowed with $k$-vector space structures such that the composition maps are $k$-bilinear.  We record all algebras as categories and work. 
  In below we shall explain how one can consider an algebra as a category. A  category $\CA$ is called {\it skeletally small} if the isoclasses of objects of $\CA$ form a set. A skeletally small category $\CA$ is called  
  {\it locally finite dimensional} if the following conditions are satisfied:
 \begin{itemize}
 	\item [$(i)$] $\CA$ is {\it basic}, i.e., distinct objects of $\CA$ are not isomorphic;
 	\item [$(ii)$] $\CA$ is {\it semiperfect}, i.e., every object of $\CA$ has a local endomorphism ring;
 	\item [$(iii)$] $\CA$ is {\it Hom-finite}, i.e., the space $\CA(x, y)$ is finite dimensional for every $x, y \in \CA.$ 
 \end{itemize}
A locally finite dimensional category  $\CA$ is called {\it finite} if it has only a finite number of objects; and $\CA$
is called {locally bounded} if, for every $x \in \CA,$ there are only finitely many $y \in \CA$ such that $\CA(x, y)\neq 0$ or $\CA(y, x)\neq 0.$ For example, let $A$ be a basic algebra (basic means that A is the direct sum of pairwise non-isomorphic indecomposable projective $A$-modules) and let $\{e_1, \cdots, e_n\}$ be a complete set of pairwise orthogonal primitive idempotents. Then $A$ can be viewed as a locally
bounded category as follows: $e_1, \cdots, e_n$ are the objects of $A$, the space of morphisms from
$e_i$ to $e_j$ is equal to $e_jAe_i$ for any $i, j$ and the composition of morphisms is induced by the
product in $A$. Notice that different choices for the primitive idempotents $e_1\cdots e_n$ give rise to
isomorphic $k$-categories.\\
 We say that a group $G$ acting on a category $\CA$, that is, there exists a homomorphism $A$ from $G$ into the group of automorphisms of $\CA.$ The $G$-action on $\CA$ is called {\it free} provided that  $A(g)(x)=x$ if and only if $g=e$, where $e$ denotes the identity in the group $G$.

 Assume that $\CA, \CB$ are locally finite dimensional $k$-categories, $F:\CA\rt \CB$ is a $k$-linear functor and $G$
a group of $k$-linear automorphisms of $\CA$ acting freely (under inclusion homomorphism $G\hookrightarrow \rm{Aut}(\CA)$) on the objects of $\CA.$  For an  automorphism $g: \CA \rt \CA$, for ease of notation  we usually write $gx:=g(x)$, $gf:=g(f)$ for all objects $x$ and morphisms $f$ in $\CA.$ \\
\begin{definition}\label{Def 2.1} $F:\CA\rt \CB$ is said to be  a  {\it Galois covering} if and only if 
\begin{itemize}
	\item [$(a)$] the functor $F:\CA\rt \CB$ induces isomorphisms 
	$$\bigoplus_{g \in G}\CA(gx, y)\rt \CB(F(x), F(y)) \ \ \  \text{and} \ \ \bigoplus_{g \in G}\CA(x, gy)\rt \CB(F(x), F(y)).$$
for all $x, y \in \CA;$
\item [$(b)$] the functor $F:\CA\rt \CB$ induces a surjective function $\rm{ob}(\CA)\rt \rm{ob}(\CB);$
\item [$(c)$] $F\circ g=F,$ for any $g \in G;$
\item [$(d)$]for any $x, y \in \rm{ob}(\CA)$ such that $F(x)=F(y)$ there is $g \in G$ such that $gx=y.$
\end{itemize}
\end{definition}
We recall that (see \cite{BG, G}) a functor $F:\CA \rt \CB$ satisfies the above conditions if and only if $F$ induces an  isomorphism $\CB \simeq \CA/G$ where $\CA/G$ is the {\it orbit category}. Given a category $\CA$ and a $G$-action on $\CA$, the orbit category $\CA/G$ is defined as follows: the objects of $\CA/G$ are all classes $Gx$ in $\CA,$ where $Gx:=\{gx: g \in G\}$,  and the set of morphisms between two objects $Gx$ and $Gy$ are defined by 

$$\CA/G(Gx,Gy):=\Delta(Gx, Gy)^G$$
where $\Delta(Gx, Gy):=\{f=(f_{b, a}) \in \prod_{(a, b) \in G \times G} \CA(ax, by)\mid \text{f is row finite and column finite} \}$. Here $(-)^G$ stands for the set of $G$-invariant elements, namely
$$\Delta(Gx, Gy)^G:= \{f=(f_{b, a}) \in \Delta(Gx, Gy) \mid \forall  g \in G, f_{gb, ga}=g(f_{b, a})\}. $$
The canonical   functor $F:\CA\rt \CA/G$  which sends an object $x \in \CA$ to the associated orbit $Gx$   in  $\CA/G,$  and for any morphism $f \in \CA(x, y)$, $F(f
)$ is defined by $F(f):=(\delta_{(a, b)}af)_{(a, b) \in G\times G}$ is a Galois covering functor, where  the notation $\delta_{(a, b)}$ stands for the Kronecker delta, namely it has the value 1 if $a=b$,
and the value $0$ otherwise, and with respecting this convention, that is, $1f=f$ and $0f=0$ for any morphism $f$ in $\CA$. 
 
 \subsection*{Setup}  In the rest of the paper,  we fix $\CA$ is a locally bounded catgery  and $G$
  a group of $k$-linear automorphisms of $\CA$ acting freely on the objects of $\CA,$ and  $F:\CA\rt \CB$ is a Galois covering  functor. In the beginning of each section  we may add some more additional assumptions on our setting. However, for the convenience of the reader and avoiding any confusion, we recall the required hypothesis to state our main results. Further, we freely use the following introduced notations concerning $F$ throughout the paper. \\

 A {\it right $\CA$-module} is a $k$-linear contravariant functor from $\CA$ to the category $\Mod k$ of $k$-vector spaces. Note that since $\CA$ is a $k$-category, the notion of right  $\CA$-module defined here is equal to the one defined in Introduction.
  An $\CA$-module $M:\CA^{\rm{op}}\rt \Mod k$ is called {\it finite dimensional} if and only if $\rm{dim} \ M=\Sigma_{x \in \rm{ob}(\CA)} \rm{dim}_k \ M(x) < \infty.$ The category of all $\CA$-modules and  all finite dimensional $\CA$-modules are denoted by $\Mod \CA$ and $\mmod \CA$, respectively.\\
 The {\it pull-up} functor $F_{\bullet}:\Mod \CB\rt \Mod \CA$ associated with $F$ is the functor $(-)\circ F^{\rm{op}}.$ The pull-up functor has the left adjoint functor $F_{\la}:\Mod \CA\rt \Mod \CB$ and the right adjoint $F_{\rho}:\Mod \CA \rt \Mod \CB$ which are called the {\it push-down} functors. Because of the importance of the push-down functor $F_{\la} $ in below we describe it explicitly. \\
Assume that $M:\CA^{\rm{op}}\rt \Mod k$ is an $\CA$-module. We define the $\CB$-module $F_{\la}(M):\CB^{\rm{op}}\rt \Mod k$ as follows. Since $F$ is a surjection on objects, we can fix for any $a \in \rm{ob}(\CB)$, $x_a \in \rm{ob}(\CA)$ such that $F(x_a)=a.$ Set  $F_{\la}(M)(a):=\oplus_{g \in G}M(gx_a)$. Assume that $\alpha \in \CB(a, b)$.  Hence by the isomorphism $\bigoplus_{g \in G}\CA(gx_a, x_b)\simeq  \CB(F(x_a), F(x_b))=\CB(a, b)$ of Definition \ref{Def 2.1},   there is a unique representation $(\alpha_g)_{g \in G}$, where $\alpha_g:gx_a \rt x_b$ for any $g \in G,$ such that $\alpha=\Sigma_{g \in G}F(\alpha_g).$ Then $F_{\la}(M)(\alpha):F_{\la}(M)=\oplus_{g \in G}M(gx_a)\rt \oplus_{g \in G}M(gx_b)=F_{\la}(M)(b)$ is defined by the matrix presentation $[M(h\alpha_{h^{-1}g})]_{(g, h) \in G}$. Assume that $M, N$ are in $\mmod \CA$ and $f=(f_x)_{x\in \rm{ob}(\CA)}$, $f_x:M(x)\rt N(x).$ Then for  any $a \in \rm{ob}(\CB)$, $F_{\la}(f)(a):F_{\la}(M)(a)\rt F_{\la}(N)(a)$ is defined as diagonal matrix  $[\delta_{(g, h)}f_{gx_a}]_{(g, h) \in G \times G}$. Let $M \in \mmod \CA$ and $X \in \mmod \CB$ be given. The assignment $f=(f_a)_{a\in  \rm{ob}(\CB)}$ in $\mmod \CB(F_{\la}(M), X)$, where $f_a=[f^g_{x_a}]_{g \in G}$, $f^g_{x_a}:M(gx_a)\rt X(a)$, to the $\CA$-homomorphism $h=(h_y)_{y \in \text{ob}(\CA)}$ in $\Mod \CA(M, F_{\bullet}X)$, where $h_y:M(y)\rt X(F(y))$, $h_y=f^g_{x_{F(y)}}$ here $g$ is a unique element in $G$ with $y=gx_{F(y)}$, gives the adjoint isomorphism in $M$ and $X$, denoted by $\eta_{M, X}$. The inverse map $\eta^{-1}_{M, X}: \Mod \CA(M, F_{\bullet}X) \rt \mmod \CB(F_{\la}(M), X)$ is defined by sending $h=(h_{y})_{y\in \text{ob}(\CA)}$ to $f=(f_a)_{a \in \text{ob}(\CB)}$, where $f_a=\oplus_{g \in G} M(gx_a)\rt X(a)$, $f_a=[h_{gx_a}]_{g \in G}$. Note that here by using the Galois covering's properties (Definition \ref{Def 2.1}(c)), we have  $F_{\bullet}X(gx_a)=X(F(gx_a))=X(F(x_a))=X(a).$ \\

Observe that by definition if an $\CA$-module $M$ is finite dimensional, then so is  $F_{\la}(M)$. Hence the functor $F_{\la}$ restricts to a functor $\mmod \CA\rt \mmod \CB.$ This functor is still denoted by $F_{\la}$.  \\
Let $\CA$-module $M$ be given. We denote by ${}^gM$ the $\CA$-module $M\circ g^{-1}$, where $g^{-1}$ is the inverse of the given  automorphism $g:\CA\rt \CA.$. Given an $\CA$-module homomorphism $f:M\rt N, \ f=(f_x)_{x \in \text{ob}(\CA)}$, we denote by ${}^gf:{}^gM\rt {}^gN,\ {}^gf=(f_{g^{-1}x})_{x \in \text{ob}(\CA)}.$ Any $g \in G$ induces an automorphism $g:\mmod \CA\rt \mmod \CA$, still denoted by $g$,  by sending $M \in \mmod \CA$ to ${}^gM=M\circ g^{-1}$. The set of all the induced  automorphisms $g:\mmod\CA\rt \mmod \CA$ makes a group which is isomorphic to $G$. Hence we can  identify the obtained subgroup of the group of all automorphisms of $\mmod \CA$ with the group $G$. Such an identification allows us to define the promised action of $\mmod \CA.$ Hence, in  this way we can  define an action of $G$ on $\mmod \CA.$ Write $\text{ind}\mbox{-}\CA$ the subcategory of $\mmod \CA$ consisting of all indecomposable objects. We say that $G$ acts freely on $\rm{ind} \mbox{-}\CA$ if and only if ${}^gM\simeq M$ implies that $g=e.$ The action of $G$ on $\CA$ is called {\it admissible} if $G$ acts freely on $\rm{ind}\mbox{-}\CA$. The reader is cautioned that our notation of admissibility may be  in conflict with that of other references. If $G$ is torsion free, then $G$ acts freely on $\text{ind}\mbox{-}\CA$, see e.g.  \cite[Lemma 2.2]{BL}.  We recall from \cite{BG} that $F_{\bullet}\circ F_{\la}(M)\simeq \oplus_{g \in G}{}^gM$. In fact, the natural transformation (or $\CA$-homomorphism) $\chi_M:F_{\bullet}\circ F_{\la}(M)\rt \oplus_{g \in G}{}^gM$, $\chi_{M, x}:F_{\bullet}\circ F_{\la}(M)(x)=\oplus_{g \in G}M(gx_{F(x)})\rt \oplus_{g \in G}M(g^{-1}x)$, where $\chi_{M, x}$ sends the summand $M(gx_{F(x)})$ at the position $g \in G$ by the identity to the summand $M((h_xg^{-1})^{-1}x)$ at the position $h_xg^{-1} \in G$, here $h_{x}$ is a unique element in $G$ such that $x=h_xx_{F(x)}$, gives the functorial isomorphism between the relevant functors.  \\
 For $M \in \mmod \CA$, the {\it support} of $M$ is the full subcategory of $\CA$ formed by all objects $x$ in $\CA$ such that $M(x)\neq 0.$ The category $\CA$ is {\it locally support-finite} if and only if for any object $x$ in $\rm{ob}(\CA)$, the union $$\bigcup_{M \in \text{ind}\mbox{-}\CA, M(x)\neq 0 } \text{supp}(M)$$
is finite. We recall that $\CA$ is {\it locally representation finite} if for any object $x$ of $\CA$ there are finitely many indecomposable $M$ in $\mmod \CA$, up to isomorphisms, with $x \in \rm{supp}(M)$.
 Hence  every  locally  representation-finite category is  locally support-finite.  We should mention that there are locally representation-infinite  categories which are locally support-finite; the reader may consult \cite{DS}.  Let $\rm{Ind}\mbox{-}\CA$ denote  a complete set of non-isomorphic indecomposable objects in $\mmod \CA$. By definition, one can see easily that $\rm{Ind}\mbox{-}\CA$ is locally finite dimensional, and moreover $\CA$ is locally representation finite if and only if $\rm{Ind}\mbox{-}\CA$ is locally bounded. Also, if $\CA$ is finite, $\CA$ is locally representation-finite if and only if $\mmod \CA$ is of finite representation type (there are only  finitely many indecomposable $\CA$-modules up to isomorphism). In this case we sometimes say $\CA$ is representation-finite.  \\

The notion of Galois functor was generalized in \cite{BL} for general linear categories, see also \cite{A}. Such a general notion of Galois functor was defined in a way such that it satisfies the similar properties of  the push-down functor $F_{\la}$. Let us recall from  \cite{A, BL} the definition of Galois (pre)covering where we do not need the involved categories to be locally finite dimensional. Because of our setting in this paper we will introduce their general notion of Galois coverings for $k$-categories.\\
The following definition is due to Asashiba originally under name of {\it invariance adjuster}, see \cite[Definition 1.1]{A}.

\begin{definition}\label{Definition 2.2}
Let $\mathcal{C}$ and $\mathcal{D}$ be categories with a group $H$ of $k$-linear automorphisms of $\mathcal{C}$. A functor $K:\mathcal{C}\rt \mathcal{D}$ is called {\it $H$-stable} provided there exist functorial isomorphisms $\delta_h:K\circ h\rt K, \ h \in H,$ such that 
$$\delta_{g, x}\circ \delta_{h, gx}=\delta_{hg, x},$$
for any $h, g \in H$ and $x \in \text{ob}(\mathcal{C})$. Here, for any $h \in H$, $\delta_h= (\delta_{h, y})_{y \in \text{ob}(\mathcal{C})}$. In this case, we call $\delta=(\delta_h)_{h \in H}$ a {\it $H$-stabilizer} for $K$.
\end{definition}
 The following definition is also due to Asashiba, see \cite[Definition 1.7]{A}.
 \begin{definition}\label{definition 2.4}
 	Let $\mathcal{C}$ and $\mathcal{D}$ be categories with a group $H$ of $k$-linear automorphisms of $\mathcal{C}$. A functor $K:\mathcal{C}\rt \mathcal{D}$ is called {\it $H$-precovering} provided that $K$ has a $H$-stabilizer $\delta$ such that, for any $x ,y \in \text{ob}(\mathcal{C})$, one of the following map is isomorphic:
 	$$K_{x, y}:\bigoplus_{h \in H}\mathcal{C}(x, hy)\rt \mathcal{D}(K(x), K(y)): \ (u_h)_{h \in H}\mapsto \Sigma_{h \in H}\delta_{h, y}\circ K(u_h).$$
 	$$K^{x, y}:\bigoplus_{g \in H}\mathcal{C}(hx, y)\rt \mathcal{D}(K(x), K(y)): \ (v_h)_{h \in H}\mapsto \Sigma_{h \in H}K(v_h)\circ \delta^{-1}_{h, x}.$$
 \end{definition}
For every $X, Y \in \mmod \CA$
$$\nu_{X, Y}:\bigoplus_{g \in G}\mmod \CA(X, {}^gY)\rt \mmod \CB(F_{\la}(X), F_{\la}(Y))$$
is given by $$\nu_{X,Y}=\eta^{-1}_{X, F_{\la}(Y)}\circ \Mod \CA(X, \chi^{-1}_Y) \circ \Upsilon_{X,Y},$$ where $\Upsilon_{X,Y}:\oplus_{g \in G}\mmod \CA(X, {}^gY)\rt \Mod \CA(X, \oplus_{g \in G}{}^gY)$ is the natural isomorphism. \\
The push-down functor $F_{\la}:\mmod \CA\rt \mmod \CB$ has a $G$-stabilizer. Indeed,  for each $h \in G$, we define natural isomorphism  $\delta^F_h:F_{\la}\circ h\rt F_{\la}$, $\delta^F_h=(\delta^F_{h, X})_{X \in \mmod \CA}$, where for each  $a \in \rm{ob}(\CB)$ and $X \in \mmod \CA $, $\delta^F_{h, X, a}:F_{\la}\circ h(X)(a)=\oplus_{g \in G}X(h^{-1}gx_a)\rt F_{\la}(X)(a)=\oplus_{g \in G}X(gx_a)$ is defined by mapping the summand $X(h^{-1}gx_a)$
at the position $g$ by the identity to the $X(h^{-1}gx_a)$ at the position $h^{-1}g.$ A direct computation shows that $\delta^F=(\delta^F_h)$ is  a $G$-stabilizer and moreover $(F_{\la})_{X, Y}=\nu_{X, Y}.$ Therefore, $F_{\la}$ is a $G$-precovering, see also \cite[Theorem 6.5]{BL}.\\
Through the paper when we say that  a functor is Galois covering, we mean in the sense of the above definition unless stated in the sense of Definition \ref*{Def 2.1}. \\

\begin{definition}
		Let $\mathcal{C}$ and $\mathcal{D}$ be categories with a group $H$ of $k$-linear automorphisms of $\mathcal{C}$.  A $H$-precovering $K:\mathcal{C}\rt \mathcal{D}$ is called a {\it Galois $H$-covering} provided that it is dense.
\end{definition}
The below theorem summarize important properties of the push-down functor $F_{\la}$.

\begin{theorem}\label{Theorem 2.2}
	Assume that $\CA$ is a locally bounded $k$-catgery  and $G$
	an admissible  group of $k$-linear automorphisms of $\CA$ acting freely on the objects of $\CA,$ and  $F:\CA\rt \CB$ is a Galois covering functor as in Definition \ref{Def 2.1}. Then the functor $F_{\la}:\mmod \CA\rt \mmod \CB$ satisfies the following assertions. 
	\begin{itemize}
		\item [$(1)$] The functor $F_{\la}$ induces the following isomorphisms of vector spaces
		$$(F_{\la})^{X, Y}:\bigoplus_{g \in G}\mmod \CA({}^gX, Y)\st{\sim}\rt \mmod \CB(F_{\la}(X), F_{\la}(Y))\st{\sim}\leftarrow  \bigoplus_{g \in G}\mmod \CA(X, {}^gY):(F_{\la})_{X, Y}$$
		for any $X, Y \in \mmod \CA.$ In particular, the functor $F_{\la}$ is a $G$-Galois precovering.
		\item [$(2)$] If $\CA$ is locally support-finite, then  $F_{\la}$ is a $G$-Galois covering.
		\item [$(3)$] Assume that $X \in \mmod \CA$. Then $F_{\la}(X)\simeq F_{\la}({}^gX)$, for any $g \in G.$
		\item [$(4)$] Assume that $X, Y \in \text{ind}\mbox{-}\CA$. Then $F_{\la}(X) \in \text{ind}\mbox{-}\CB$ and $F_{\la}(X)\simeq F_{\la}(Y)$ yields $Y\simeq {}^gX$ for some $g \in G.$
	\end{itemize}
\end{theorem}
\begin{proof}
$(1)$ We discussed it in the above. \cite[Subsection 2.5]{DS} implies $(2)$. $(3)$ and $(4)$ follow from \cite[Lemma 2.9]{BL} and using this fact that Hom-finiteness of $\CB$ implies (\cite[Proposition 5.4]{K}): for any $A \in \mmod \CB$, the endomorphism ring of  $A $ is local if and only if $A$ is indecomposable.
\end{proof}
 We discussed in  the above how the group $G$ defines an action on $\mmod \CA.$  For any $g \in G$, since $g:\mmod \CA\rt \mmod \CA$ is   an automorphism, it preserves the projective objects in $\mmod \CA$. Hence we get the induced  automorphism $g:\underline{\rm{mod}}\mbox{-}\CA\rt \underline{\rm{mod}}\mbox{-}\CA$, still denoted by $g$,  of the stable category $\underline{\rm{mod}}\mbox{-}\CA.$ Here and throughout, when we want to consider a morphism $f$ as a morphism in the stable category we use the notation $\underline{f}$. In the same way for $\mmod \CA,$ the group $G$ defines an action on $\underline{\rm{mod}}\mbox{-}\CA.$ Moreover, for any $g \in G$, the following diagram commutes
\[\xymatrix{ \mmod \CA\ar[rr]^{\pi} \ar[d]^{g} && \underline{\rm{mod}}\mbox{-}\CA \ar[d]^{g} \\ \rm{mod}\mbox{-}\CA \ar[rr]^{\pi} && \underline{\rm{mod}}\mbox{-}\CA }\]	
where $\pi$ is the quotient functor. 
Since the push-down functor $F_{\la}$ preserves the projective objects in $\mmod \CA$, it induces the functor $\underline{F_{\la}}:\underline{\rm{mod}}\mbox{-}\CA\rt \underline{\rm{mod}}\mbox{-}\CB$ between the stable categories. In the next result we show that the (stable) functor $\underline{F_{\la}}$ is a $
G$-precovering, and even $G$-Galois covering if we assume $\CA$ is locally support-finite. First of all, the $G$-stabilizer $(\delta^F_h)_{h \in G}$ induces the $G$-stabilizer $(\underline{\delta^F_h})_{h \in G}$, where $\underline{\delta^F_h}=(\underline{\delta^F_{h, X}})_{X \in \mmod \CA}$.
 The above observation is visualized in the following diagram: 
\[ 
\xymatrix@C=7em@R=7em{
	\mmod \CA \rtwocell^{F_{\la}\circ h}_{F_{\la}}{\ \ \delta^F_h}& \mmod \CB\\
\underline{\rm{mod}}\mbox{-}\CA \rtwocell^{\underline{F_{\la}}\circ \underline{h}}_{\underline{F_{\la}}}{\ \ \ \underline{\delta^F_h}} & \underline{\rm{mod}}\mbox{-}\CB.
	\ar_{\pi} "1,1";"2,1" 
	\ar^{\pi} "1,2";"2,2"  }\] 
Here  we use the same notation $\pi$ to denote  the corresponding quotient functor.
\begin{proposition}\label{Prop 2.6}
	Keep in mind the above notation. Then the following hold.
	\begin{itemize}
		\item [$(i)$] The functor $\underline{F_{\la}}:\underline{\rm{mod}}\mbox{-}\CA\rt \underline{\rm{mod}}\mbox{-}\CB$ is a $G$-precovering.
		\item [$(ii)$] If $\CA$ is locally support-finite, then  $\underline{F_{\la}}$ is a $G$-Galois covering.
	\end{itemize}
\end{proposition}
\begin{proof}
	We only need to prove $(i)$. Let $M$ and $N$ be in $\mmod \CA.$ We must show that the $k$-linear map $\underline{F_{\la}}_{M, N}$  is an isomorphism (See Definition \ref{Definition 2.2}). Since $M$ is finitely generated, there is the following canonical isomorphism
	$$ \Upsilon_{M, N}:\oplus_{g \in G}\mmod \CA(M, {}^gN)\st{\sim}\rt \Mod \CA(M, \oplus_{g \in G}{}^gN).$$
	 The above isomorphism gives rise to the following isomorphism between the Hom-spaces in the stable categories
	 	$$\underline{\Upsilon}_{M, N}:\oplus_{g \in G}\underline{\rm{mod}}\mbox{-} \CA(M, {}^gN) \st{\sim}\rt \underline{\rm{Mod}}\mbox{-} \CA(M, \oplus_{g \in G}{}^gN).$$
	 As discussed in the above, the adjoint pair $(F_{\la}, F_{\bullet})$ yields the following
	  isomorphism 
	  $$\eta_{M, F_{\la}(N)}:\mmod \CB(F_{\la}(M), F_{\la}(N))\rt \Mod \CA (M, F_{\bullet}(F_{\la}(N))).$$
	  It can be check easily for any $f \in \mmod \CB(F_{\la}(M), F_{\la}(N))$, $\eta_{M, F_{\la}(N)}(f)$ is equal to the following composition 
	  $$M \st{i}\rt \oplus_{g \in G}{}^gM\st{\chi^{-1}_{M}}\rt F_{\bullet}(F_{\la}(M))\st{F_{\bullet}(f)}\rt F_{\bullet}(F_{\la}(N)), $$
	  where $i$ is the canonical injection. If $f$ factors through a projective object $Q$ in $\mmod \CB$, then by the above description  we observe that  $\eta_{M, F_{\la}(N)}(f)$ factors through the projective object $F_{\bullet}(Q)$. The fact gives rise to the following  surjective map
	  $$\underline{\eta_{M, F_{\la}(N)}}:\underline{\rm{mod}}\mbox{-} \CB(F_{\la}(M), F_{\la}(N))\rt \underline{\rm{Mod}}\mbox{-} \CA (M, F_{\bullet}(F_{\la}(N))).$$
Combining  all together the above introduced $k$-linear maps  yields the following commutative diagram  	  
	  \[\xymatrix{ \oplus_{g \in G}\underline{\rm{mod}}\mbox{-} \CA(M, {}^gN)\ar[rr]^{\underline{\Upsilon}_{M, N}} \ar[d]^{(\underline{F_{\la}})_{M, N}} && \underline{\rm{Mod}}\mbox{-} \CA(M, \oplus_{g \in G}{}^gN) \ar[d]^{\underline{\rm{Mod}}\mbox{-} \CA(M, \underline{\chi^{-1}_N})} \\ \underline{\rm{mod}}\mbox{-} \CB( \underline{F_{\la}}(M), \underline{F_{\la}}(N)) \ar[rr]^{\underline{\eta_{M, F_{\la}(N)}}} && \underline{\rm{Mod}}\mbox{-} \CA(M, F_{\bullet}(F_{\la}(N))}\]
	  The map $\chi_{N}$, so $\chi^{-1}_N$, is defined in the above (before the definition \ref{Definition 2.2}).  By the above diagram, $\underline{F_{\la}}_{M, N}$ is an isomorphism if and only if so is  $\underline{\eta_{M, F_{\la}(N)}}$. Therefore, it suffices  to show that $\underline{\eta_{M, F_{\la}(N)}}$ is injective.  Assume that $\underline{\eta_{M, F_{\la}(N)}}(\underline{f})=0$ for some  $f \in \mmod \CB(F_{\la}(M), F_{\la}(N))$. To this end, it is equivalent to show that $f$ factors through a projective object in $\mmod \CB.$ Similarly we have the following commutative diagram
	 \[\xymatrix{ \oplus_{g \in G}\rm{mod}\mbox{-} \CA(M, {}^gN)\ar[rr]^{\Upsilon_{M, N}} \ar[d]^{(F_{\la})_{M, N}} && \rm{Mod}\mbox{-} \CA(M, \oplus_{g \in G}{}^gN) \ar[d]^{\rm{Mod}\mbox{-} \CA(M, \chi^{-1}_N)} \\ \rm{mod}\mbox{-} \CB( F_{\la}(M), F_{\la}(N)) \ar[rr]^{\eta_{M, F_{\la}(N)}} && \rm{Mod}\mbox{-} \CA(M, F_{\bullet}(F_{\la}(N)).}\]
	  In fact, the first diagram in the above is induced by the second one. The commutativity of the above diagrams and using this fact that $(F_{\la})_{M, N}$ is an isomorphism  imply the existence of $u=(u_g)_{g\in G}$, $u_g:M \rt {}^gN$ such that $f=(F_{\la})_{M, N}(u)=\Sigma_{g \in G}\delta^F_g\circ F_{\la}(u_g)$ and each of $u_g$ factors through some projective object in $\mmod \CA$, namely, $u_g=p_g\circ q_g$, where $q_g:M \rt P_g, \ p_g:P_g\rt {}^gN$ and $P_g$ a projective object in $\mmod \CA.$ Now the  data provided in the above implies the following factorization of $f$
	
	  \[	  \xymatrix{
	  	F_{\la}(M) \ar[rr]^{f}
	  	\ar[rd]_{[F_{\la}(q_g)]^{\rm{t}}_{g \in G}}&&F_{\la}(N)\\
	  	& \oplus_{g \in G}F_{\la}(P_g)\ar[ru]_{\ \ \ \ [\delta^F_g\circ F_{\la}(p_g)]_{g \in G}}
	  }\]
	 
	 Since $F_{\la}(M)$ belongs to $\mmod \CB$, by the above factorization we get a factorization of $f$ through a direct sum of finitely many projective objects $F_{\la}(P_g)$, so desired claim.    We are done. 	  
\end{proof}

\section{Galois precovering for (Mono)morphism categories}\label{Subsection 3.1}
In this section, we will introduce  $G$-precoverings of the morphism categories and monomorphism categories which are  induced by the push-down functor $F_{\la}$. 

  Let $\mathbb{X}=(X\st{f}\rt Y)$ be an object in $\text{H}(\mmod \CA)$ and $g \in G.$ We denote by  ${}^g\mathbb{X}=({}^gX\st{{}^gf}\rt {}^gY)$ the object in $\text{H}(\mmod\CA)$ by acting $g$ competently-wise on $\mathbb{X}$. Assume a morphism $\phi=(\phi_1, \phi_2):\mathbb{X}\rt \mathbb{Y}$ in $\text{H}(\mmod \CA)$ is given. Define ${}^g\phi=({}^g\phi_1, {}^g\phi_2)$ a morphism in $\text{H}(\mmod \CA)$ from ${}^g\mathbb{X}$ to ${}^g\mathbb{Y}$ obtained by acting $g$ competently-wise on $\phi.$ We observe that $g$ induces an automorphism on $\text{H}(\mmod \CA)$. Denote again by $g$ the  automorphism on $\text{H}(\mmod \CA)$ induced by the automorphism $g:\CA \rt \CA.$ The set of all the induced automorphisms $g:\text{H}(\mmod \CA)\rt \text{H}(\mmod \CA)$ makes a group of automorphisms on $\rm{H}(\mmod \CA)$, which is isomorphic to $G$. We identify the new obtained group by the group $G$. Hence in this way one can define an action  of $G$ on $\text{H}(\mmod \CA)$. Since $g:\CA\rt \CA$ is an automorphism, $\phi$ is a monomorphism if and only if so is ${}^g\phi.$ Therefore, the action of $G$ on $\text{H}(\mmod \CA)$, already defined, can be restricted to  an action of $G$ on $\CS(\mmod \CA)$.\\
  
   Using the same idea (applying component-wisely a functor), we can extend the functor $F_{\la}:\mmod A\rt \mmod \CB$ to the functor $\text{H}\rm{F}_{\la}:\text{H}(\mmod \CA)\rt \text{H}(\mmod \CB)$, sending $ (X\st{f}\rt Y) \in \text{H}(\mmod \CA) $ to $(F_{\la}(X)\st{F_{\la}(f)}\rt F_{\la}(Y))$ in $\text{H}(\mmod \CB)$. Since $F_{\la}$ is exact, we can restrict $\text{H}\rm{F}_{\la}$
 to the corresponding monomorphism categories. Let us denote the restricted functor by $\CS \rm{F}_{\la}:\CS(\mmod \CA)\rt \CS(\mmod \CB)$.
Analog with  the above observation,  denote by $\text{H}\rm{F}_{\bullet}:\text{H}(\mmod\CB)\rt \text{H}(\Mod \CA)$ the functor  which is obtained by applying component-wisely the functor $F_{\bullet}:\mmod \CB\rt \Mod \CA$. Since the functor $F_{\bullet}$ is exact we get also the restricted functor $\CS \rm{F}_{\bullet}:\CS(\mmod \CB) \rt \CS(\Mod \CA)$.\\

\subsection*{Setup} From now on to the end of the paper, we assume that $G$ is torsion-free. Hence the action of $G$ on $\CA$ becomes admissible.
\begin{lemma}\label{Lemma 3.1}
	Assume that $\mathbb{X}\neq 0$ in $\rm{H}(\mmod\CA)$. If ${}^g\mathbb{X}\simeq \mathbb{X}$ for some $g \in G$, then $g=e$. 
	\end{lemma}
\begin{proof}
Assume $\mathbb{X}=(X\st{f}\rt Y)$. Since $\mathbb{X}\neq 0$, so either $X\neq 0$ or $Y\neq 0.$	Assume $X\neq 0$. The proof of the case $Y\neq 0$ is similar. Assume to the contrary, $g \neq e.$ The isomorphism ${}^g\mathbb{X}\simeq \mathbb{X}$ yields ${}^{g^{-n}}X\simeq X$ for any $n >0.$ Since $X\neq 0$, there exists $x \in \CA$ such that $X(x)\neq 0$, By the isomorphisms ${}^{g^{-n}}X\simeq X$, we get  $X(g^nx)\neq 0$. According to  our assumption of $G$ being torsion free, we infer $X(a)\neq 0$ for infinitely many $a \in \CA$. It contradicts the assumption of $X$ being finite dimensional. So we are done.
\end{proof}
Thus the above lemma says that $G$ acts freely on $\rm{H}(\mmod\CA)$ and $\CS(\mmod \CA)$.\\

The following theorem shows that the functors $\rm{H}F_{\la}, \rm{H}F_{\bullet}$, resp. $\CS \rm{F}_{\la}, \CS \rm{F}_{\bullet}$, behave the same as the pull-up functor $F_{\bullet}$ and the push-down functor $F_{\la}$.
\begin{proposition}\label{Prop 3.2}
		Assume that $\mathbb{X}, \mathbb{X}_1, \mathbb{X}_2   \in \rm{H}(\mmod \CA)$.  The following assertions hold.
		\begin{itemize}
			\item [$(1)$] $\rm{HF}_{\la}({}^g\mathbb{X})\simeq \rm{H}\rm{F}_{\la}(\mathbb{X})$, for any $g \in G.$
			\item [$(2)$] There exists an isomorphism $$\rm{H}F_{\bullet}\circ \rm{H}F_{\la} (\mathbb{X})\simeq \oplus_{g \in G} {}^g\mathbb{X}$$
			\item [$(3)$] If $\mathbb{X}$ is indecomposable, then  so is $\rm{H}F_{\la}(\mathbb{X})$.
			\item [$(4)$] If objects $\mathbb{X}_1$ and $\mathbb{X}_2$ are indecomposable, then $\rm{H}F_{\la}(\mathbb{X}_1)\simeq \rm{H}F_{\la}(\mathbb{X}_2)$	implies $\mathbb{X}_1\simeq \mathbb{X}^g_2$,  for some $g \in G.$
			\end{itemize}
	In particular, the same statements $(1)-(4)$ hold for when the related objects and functors concern the corresponding monomorphism categories.	
\end{proposition}
\begin{proof}
 	$(1)$ Assume $\mathbb{X}=(X\st{f}\rt Y)$. Consider the pair $(\delta^F_{g, X}, \delta^F_{g, Y})$, see Section \ref{Perliminary} for the definition. It is a morphism in $\rm{H}(\mmod \CB)$. For this we need to show   the commutativity of the following diagram  for any $a \in \rm{ob}(\CB)$
 	\[\xymatrix{F_{\la}({}^gX)(a)\ar[rr]^{F_{\la}({}^gf)(a)} \ar[d]^{\delta^F_{g, X, a}} && F_{\la}({}^gY)(a) \ar[d]^{\delta^F_{g, Y, a}} \\ F_{\la}(X)(a) \ar[rr]^{F_{\la}(f)(a)} && F_{\la}(Y)(a) }\]
 The both composition functors $\delta^F_{g, Y, a}\circ F_{\la}({}^gf)(a)$ and $F_{\la}(f)(a) \circ \delta^F_{g, X, a}$, by following the definitions of the involved functors, send,  for any $h \in G$,  the summand $X(g^{-1}hx_a)$	at position $h$ of $F_{\la}({}^gX)(a)=\oplus_{h \in G}X(g^{-1}hx_a)$ by the morphism $f_{g^{-1}hx_a}$ to the summand $Y(g^{-1}hx_a)$ of $F_{\la}(Y)(a)=\oplus_{h \in G}Y(hx_a)$ at the position $g^{-1}h$. Hence we get  the required commutativity. 	
 	$(2)$ To prove it, consider the following isomorphisms
\begin{align*}
\rm{H}F_{\bullet}\circ \rm{H}F_{\la} (\mathbb{X})&\simeq (F_{\bullet}\circ F_{\la}(A)\st{F_{\bullet}\circ F_{\la}(f)}\lrt F_{\bullet}\circ F_{\la}(B)) \\
&\simeq ( \oplus_{g \in G} {}^gA\st{\oplus_{g \in G}{}^gf}\lrt \oplus_{g \in G} {}^gB) \\
& \simeq \oplus_{g \in G}({}^gA \st{{}^gf}\rt {}^gB)\\& \simeq \oplus_{g \in G}{}^g\mathbb{X}.
\end{align*}
The second isomorphism follows from the canonical isomorphism $F_{\bullet}\circ F_{\la}(C)\simeq \oplus_{g \in G} {}^gC$, for any $C \in \CA.$ In fact, $\chi_{\mathbb{X}}=(\chi_{X}, \chi_Y)$ gives the natural isomorphism $
\text{H}\text{F}_{\bullet}\circ \text{H}\text{F}_{\la}(\mathbb{X})\simeq \oplus_{g \in G}{}^g\mathbb{X}$. The definitions of $\chi_{X}, \chi_{Y}$ are given in Section \ref{Perliminary}, and a direct computation shows that $\chi_{\mathbb{X}}$ is a morphism in $\text{H}(\mmod \CA)$.   $(3)$ Assume that $\text{H}\rm{F}_{\la}(\mathbb{X})=C\oplus C'$ and $C\neq 0.$ 
By the part $(2)$ we have $\text{H}\rm{F}_{\bullet}\circ \rm{H}F_{\la} (\mathbb{X})\simeq \oplus_{g \in G} {}^g\mathbb{X}\simeq  \text{H}F_{\bullet}(C)\oplus \text{H}F_{\bullet}(C')$. Since by Lemma \ref{Lemma 3.1}, ${}^g\mathbb{X}$ are pairwise not isomorphic, we observe that $\text{H}\rm{F}_{\bullet}(C)=\oplus_{h \in V}{}^h\mathbb{X} \ $ for some subset  $V \subseteq G.$ By the property $F \circ g=F$, for any $g \in G$, of the Galois covering functor $F$, we infer that ${}^gF_{\bullet}(M)=F_{\bullet}(M)$ for each $M$ in $\mmod \CB$. This fact follows that  $\text{H}\rm{F}_{\bullet}(C)={}^g\text{H}\rm{F}_{\bullet}(C)\simeq \oplus_{h \in V}{}^{gh}C$ for each $g \in G$. By the Krull-Schmidt property, for each $g \in G$ and $h \in V$ there is $h' \in V$ such that ${}^{gh}\mathbb{X}\simeq {}^{h'}\mathbb{X}$.
Thanks to  Lemma \ref{Lemma 3.1}, $h'=gh$. Hence $V=gV$ for each $g \in G$, so $V=G$ and $C'=0.$ $(4)$ Assume $\text{H}\text{F}_{\la}(\mathbb{X}_1)\simeq \text{H}\rm{F}_{\la}(\mathbb{X}_2)$. Applying the functor $\rm{H}\rm{F}_{\bullet}$ on the the isomorphism and
in view of the part $(2)$ we get $$\oplus_{g \in G}{}^g\mathbb{X}_1\simeq \rm{H}\rm{F}_{\bullet}\circ \rm{H}\rm{F}_{\la}(\mathbb{X}_1)\simeq HF_{\bullet}\circ HF_{\la}(\mathbb{X}_2)\simeq \oplus_{g \in G}{}^g\mathbb{X}_2.$$ Hence $\mathbb{X}_1$ is a  summand of $\oplus_{g \in G}{}^g\mathbb{X}_2.$
The Krull-Schmidt property of $\rm{H}(\mmod \CA)$ implies  the existence of $g \in G$ such that $\mathbb{X}_1\simeq {}^g\mathbb{X}_2$, as desired. The  parts concerning monomorphism categories  follow from the morphism case since the  related functors $\CS \rm{F}_{\la}$ and $\CS \rm{F}_{\bullet}$ are restrictions of the functors $\text{H}\rm{F}_{\la}$ and $\text{H}\rm{F}_{\bullet}$.
\end{proof}
\begin{lemma}\label{Lemma 3.3}
	Let $\mathbb{X}$ be in $\rm{H}(\mmod \CA)$ and $\mathbb{M}$ in $\rm{H}(\mmod \CB).$ Then there exists the following natural isomorphisms  of $k$-vector spaces	
	$$\xi_{\mathbb{X}, \mathbb{M}}:\rm{H}(\mmod \CB)(\rm{H}F_{\la}(\mathbb{X}), \mathbb{M}) \rt \rm{H}(\Mod \CA)(\mathbb{X}, \rm{H}F_{\bullet}(\mathbb{M})).$$
In particular, there is the same isomorphism as above for the functors $\CS \rm{F}_{\la}$ and $\CS \rm{F}_{\bullet}$ when $\mathbb{X} \in \CS(\mmod \CA)$ and $\mathbb{M} \in \CS(\mmod \CB).$
\end{lemma}
\begin{proof}
Assume $\mathbb{X}=(X\st{f}\rt Y)$  and $\mathbb{M}=(M\st{h}\rt N)$.  Based on the adjoint isomorphism of the adjoint pair $(F_{\la}, F_{\bullet})$, introduced in Section \ref{Perliminary}, we define the following $k$-linear map 
$$\xi_{\mathbb{X}, \mathbb{M}}:\rm{H}(\mmod \CB)(\rm{H}F_{\la}(\mathbb{X}), \mathbb{M}) \rt \rm{H}(\Mod \CA)(\mathbb{X}, \rm{H}F_{\bullet}(\mathbb{M}))$$
by sending $\xi=(\xi_1, \xi_2)$ in  $\rm{H}(\mmod \CB)(HF_{\la}(\mathbb{X}), \mathbb{M})$ to $$\eta=(\eta_{X, M}(\xi_1), \eta_{Y, N}(\xi_2))$$ in $\rm{H}(\Mod \CA)(\mathbb{X}, HF_{\bullet}(\mathbb{M})).$ By a direct computation and following the definitions of $\eta_{-,-}, F_{\bullet}$ and $F_{\la}$, one can show that $\eta$ is a morphism in $\rm{H}(\mmod \CB)$. Similarly, one can define the $k$-linear map
$$\xi'_{\mathbb{X}, \mathbb{M}}:\rm{H}(\Mod \CA)(\mathbb{X}, HF_{\bullet}(\mathbb{M}))\rt \rm{H}(\mmod \CB)(HF_{\la}(\mathbb{X}), \mathbb{M})$$
by sending $\eta=(\eta_1, \eta_2)$ in  $\rm{H}(\Mod \CA)(\mathbb{X}, HF_{\bullet}(\mathbb{M}))$  to

$$\xi=(\eta^{-1}_{X, M}(\eta_1), \eta^{-1}_{Y, N}(\eta_1))$$ in $\rm{H}(\mmod \CB)(HF_{\la}(\mathbb{X}), \mathbb{M})$. It is not difficult to see that $\xi_{\mathbb{X}, \mathbb{M}}$ and $\xi'_{\mathbb{X}, \mathbb{M}}$ are mutually inverse. Hence $\xi_{\mathbb{X}, \mathbb{M}}$ is an  isomorphism, so the  desired result.
\end{proof}
According to  the proof of the above lemma, we have the adjoint pair $(\rm{H}F_{\la}, \rm{H}F_{\bullet})$ if from the beginning we define $\rm{H}F_{\la}$ and $\rm{H}F_{\bullet}$ on the morphism categories of $\Mod \CA$ and $\Mod \CB.$

\begin{theorem}\label{Theorem 3.2}
Assume that  $\CA$ is a locally support-finite $k$-category, $G$  a torsion-free group of $k$-linear automorphisms of $\CA$, and further $F:\CA \rt \CB$ a Galois  covering (in the sense of Definition \ref{Def 2.1}). Then
\begin{itemize}
	\item [$(1)$] The functor $\rm{H}F_{\la}:\rm{H}(\mmod \CA)\rt \rm{H}(\mmod \CB)$ is a $G$-precovering.
	\item [$(2)$] The functor $\CS\rm{F}_{\la}:\CS(\mmod \CA)\rt \CS(\mmod \CB)$ is a  $G$-precovering.
\end{itemize}
\end{theorem}
\begin{proof}
With help of $G$-stabilizer $\delta^F=(\delta^F_h)_{h \in G}$, explained in Section \ref{Perliminary}, for the $G$-precovering $F_{\la}$, we construct a $G$-stabilizer $\delta^{HF}=(\delta^{HF}_{h})_{h \in G}$ for the functor $\rm{H}F_{\la}$. For each $\mathbb{X}=(X\st{f}\rt Y)$ in $\rm{H}(\mmod \CA)$ and $h \in G$, we define the natural transformation $\delta^{HF}_h=(\delta^{HF}_{h, \mathbb{X}})$, where $\delta^{HF}_{h, \mathbb{X}}:\rm{H}F_{\la}\circ h (\mathbb{X})\rt HF_{\la}(\mathbb{X})$ is the morphism  $(\delta^F_{h, X}, \delta^F_{h, Y})$ in $\rm{H}(\mmod \CB)$. Indeed, the naturality of $\delta^F_{h}$ shows the pair $(\delta^F_{h, X}, \delta^F_{h, Y})$ is a morphism in $\rm{H}(\mmod \CB)$, so our definition is well-defined. It is not really difficult by following the definitions to show that $\delta^{HF}$ is a $G$-stabilizer for the functor $\rm{H}F_{\la}$. Consider the following composite $\omega_{\mathbb{X}, \mathbb{Y}}$ for every $\mathbb{X}=(X\st{f}\rt Y)$ and $\mathbb{Y}=(V\st{s}\rt W)$ in $\rm{H}(\mmod \CA)$

\[\xymatrix{ \oplus_{g \in G}\rm{H}(\CA)(\mathbb{X}, {}^g\mathbb{Y}) \ar[r]^>>>>>>{\mathbf{\Upsilon}_{\mathbb{X}, \mathbb{Y}}}  & \rm{H}(\Mod \CA)(\mathbb{X}, \oplus_{g \in G}{}^g\mathbb{Y}) \ar[rr]^>>>>>>>>>{\rm{H}(\CA)(\mathbb{X}, \varepsilon_{\mathbb{X}, \mathbb{Y}})}  &   & \rm{H}(\CA)(\mathbb{X}, \rm{H}F_{\bullet}\circ \rm{H}F_{\la}(\mathbb{Y}))\ar[d]^{\xi_{\mathbb{X}, \rm{H}F_{\la}(\mathbb{Y})}} \\
	 &  &   & \rm{H}(\CB)(\rm{H}F_{\la}(\mathbb{X}), \rm{H}F_{\la}(\mathbb{Y})).}\]
For saving the space  in the above diagram, we denote $\rm{H}(\mmod \CA)$ and $\rm{H}(\mmod \CB)$ by $\rm{H}(\CA)$
and $\rm{H}(\CB)$, respectively. Let us introduce the $k$-linear maps appeared in the above diagram based on the  $k$-linear maps established in Section \ref{Perliminary}. Firstly, $$\mathbf{\Upsilon}_{\mathbb{X}, \mathbb{Y}}:\oplus_{g \in G}\rm{H}(\mmod \CA)(\mathbb{X}, {}^g\mathbb{Y})\rt \rm{H}(\mmod \CA)(\mathbb{X}, \oplus_{g \in G}{}^g\mathbb{Y})$$ is the natural isomorphism. More precisely, for $u=(u_g)_{g \in G}$ in $\oplus_{g \in G}\rm{H}(\mmod\CA)(\mathbb{X}, {}^g\mathbb{Y})$, $u_g=(v_g, w_g): \mathbb{X}\rt {}^g\mathbb{Y}$, $$\mathbf{\Upsilon}_{\mathbb{X},\mathbb{Y}}(u):=(\Upsilon_{X,V}(v), \Upsilon_{Y, W}(w))$$
where $v=(v_g)_{g \in G}$ and $w=(w_g)_{g \in G}$. Secondly, $\varepsilon_{\mathbb{X}, \mathbb{Y}}$ is the isomorphism $(\chi^{-1}_X, \chi^{-1}_Y)$ in $\rm{H}(\mmod \CA)$. The definition of $\chi^{-1}$ is given in Section \ref{Perliminary}, and moreover the naturality of the natural transformation $\chi^{-1}$ implies that the pair $(\chi^{-1}_X, \chi^{-1}_Y)$ is a morphism  in $\rm{H}(\mmod \CA)$. In view of Lemma \ref{Lemma 3.3}, one can deduce that $\omega_{\mathbb{X}, \mathbb{Y}}$ is an isomorphism. Since our definition of $\omega_{\mathbb{X}, \mathbb{Y}}$ is a component-wise extension of $k$-linear map $\nu_{-,-}$, see Section \ref{Perliminary}, we easily observe the associated $k$-linear map $(\rm{H}F_{\la})_{\mathbb{X}, \mathbb{Y}}$, see Definition \ref{definition 2.4}, with respect to $G$-stabilizer $\delta^{HF}$ is equal to $\omega_{\mathbb{X}, \mathbb{Y}}$. So we are done.
\end{proof}

\section{Galois precovering for functor categories}\label{Subsection 3.2}
The main purpose of this section is to explain that how the `push-down' functor $\rm{H}F_{\la}:\rm{H}(\mmod \CA)\rt \rm{H}(\mmod \CB)$ defined in the previous section is related to the functor  $\Phi:\CF(\mmod \CA)\rt \CF(\mmod \CB)$ defined in Section 5 of \cite{P}.  This functor plays an essential role in \cite{P} to prove stability of the Krull-Gabriel dimension under Galois coverings.  Also the same observation will be carried out for the functor $\CS \rm{F}_{\la}$ and an restriction of $\Phi.$
\subsection*{Setup} From now on to the end of the paper, we assume that $\CA$ is locally support-finite. So Theorem \ref{Theorem 2.2} guarantees that the push-down $F_{\la}$ is dense.\\

	 Let us briefly recall the definition of $\Phi.$ Denote by $\rm{Add}(\mmod \CA)$ the full subcategory of $\Mod \CA$ whose objects are arbitrary direct sums of  objects of $\mmod \CA$. Assume that a natural transformation  $f:\oplus_{j \in J}M_j\rt \oplus_{i \in I}N_i$ is defined  by morphisms $f_{ij}:M_j\rt N_i$ in $\mmod \CA$ for $i \in I, j \in J.$ Let $T \in \CF(\mmod \CA)$. The functor $\widehat{T}:(\rm{Add}(\mmod \CA))^{\rm{op}}\rt \Mod k$ is defined by the following way. Take an object $\oplus_{j \in J}M_j$ and an $\CA$-homomorphism  $f:\oplus_{j \in J}M_j\rt \oplus_{i \in I}N_i$ of  $\rm{Add}(\mmod \CA)$ is defined  by morphisms $f_{ij}:M_j\rt N_i$ in $\mmod \CA$, for $i \in I$ and $j \in J.$ Then define $\widehat{T}(\oplus_{j \in J}M_j)=\oplus_{j \in J}T(M_j)$ and $\widehat{T}(f)=[\widehat{T}(f)_{ij}]_{i\in I, j \in J}:\oplus_{i \in I}T(N_i)\rt \oplus_{j \in J}T(M_j)$ where $\widehat{T}(f)_{ij}=T(f_{ij}):T(N_i)\rt T(M_j)$. Observe that the functor $\widehat{T}$ satisfies the following communicative diagrams
	
	\[ \xymatrix{  (\rm{Add}(\mmod \CA))^{\rm{op}} \ar[rr]^<<<<<<<<<<{\widehat{T}} && \Mod k  \\
		(\mmod \CA)^{\rm{op}} \ar@{^(->}[u] \ar[rr]^<<<<<<<<<<<<<{T} && \mmod k  \ar@{^(->}[u] }\]
Define $\Phi(T)=\widehat{T}\circ F_{\bullet}$.	More precisely, for $X \in \mmod \CB$, choose (it exists since $F_{\la}$ is dense)  and fix some $M \in \mmod \CB$ such that $X\simeq F_{\la}(M)$. On objects, define   $\Phi(T)(X)=\oplus_{g \in G}T({}^gM)$. Assume $\iota:T_1\rt T_2$ is a morphism in $\CF(\mmod \CA)$. Then $\Phi(\iota)=(\Phi(\iota)_{X})_{X \in \mmod \CB}:\Phi(T_1)\rt \Phi(T_2)$ is defined as follows. For each $X \in \mmod \CB$, 
	$$\Phi(\iota)_{X}=[\Phi(\iota)_{X, g, h}]_{(g, h) \in G \times G}:\bigoplus_{g \in G}T_1({}^gM)\rt \bigoplus_{g \in G}T_2({}^gM),$$ 
	where $M $ is the fixed object in $\mmod \CA$  such that $F_{\la}(M)\simeq X$, and for each $g, h \in G$, if $g=h$, then $\Phi(\iota)_{X, g, h}$ is the following composite $$T_1({}^gM)\st{i_g}\rt \oplus_{g \in G} T_1({}^gM)\st{\Phi(\iota)_{X}}\lrt  \oplus_{g \in G} T_2({}^gM)\st{\pi_g}\rt T_2({}^gM) $$
	where $i_g$ and $\pi_g$ are the canonical injection and projection, respectively, and if $g\neq h$, then $\Phi(\iota)_{X, g, h}=0.$ As it is shown in \cite[Theorem 5.5]{P} the functor  $\Phi$ is  a $G$-precovering.\\
	To state the desired connection between the functors $\Phi$ and $\rm{H}F_{\la}$, we need to establish  the following construction.  In fact, this functor is a functorial version of  the functor studied in \cite{Ar2}.
	\begin{construction}\label{Construction 4.1}
	Let $\CC$ be an abelian category. Applying the Yoneda functor on objects of $\rm{H}(\CC)$ gives the following mapping 
	$$\Theta_{\CC}:\text{H}(\CC) \rt \CF(\CC) \ \ \ \ \ (X_1\st{f}\rt X_2)\mapsto \text{Coker}(\CC(-, X_1)\st{\CC(-, f)}\lrt \CC(-, X_2)).$$
	
	For morphism: Let $(h_1, h_2)=\mathbb{X}\rt \mathbb{X}'$, where $\mathbb{X}=(X_1\st{f}\rt X_2)$ and $\mathbb{X}'=(X'_1\st{f'}\rt X'_2)$,  be a morphism in $\rm{H}(\CC)$. Define $\Theta_{\CC}((h_1, h_2))$ the unique morphism $\sigma$ satisfies the following commutative diagram 
	
	\[\xymatrix{ \CC(-, X_1) \ar[r]^{\CC(-, f)} \ar[d]^{\CC(-, h_1)} & \CC(-, X_2) \ar[r] \ar[d]^{\CC(-, h_2)} & \Theta_{\CC}(\mathbb{X})  \ar[r] \ar[d]^{\sigma} & 0 \\
		\CC(-, X'_1)  \ar[r]^{\CC(-,f')} & \CC(-, X'_2) \ar[r] & \Theta_{\CC}(\mathbb{X}')  \ar[r]  & 0.}\]
\end{construction}

\subsection*{Notation} For simplifying the notations  in 
	the rest of our paper,  for any object  $X$ in $\mmod \CA$, resp. $\underline{\rm{mod}}\mbox{-}\CA$, and morphism $X\st{f}\rt Y$, resp. $\underline{f}:X\rt Y$ in $\mmod \CA$, resp. $\underline{\rm{mod}}\mbox{-}\CA$,  the induced representable functor $\mmod \CA(-, X)$, resp. $\underline{\rm{mod}}\mbox{-}\CA(-, X)$, and the induced natural transformation  $\mmod \CA(-, f)$, resp. $\underline{\rm{mod}}\mbox{-}\CA(-, \underline{f})$, in $\CF(\mmod \CA)$, resp. $\CF(\underline{\rm{mod}}\mbox{-}\CA)$, will be denoted by $\widehat{\CA}(-, X)$ and $\widehat{\CA}(-, f)$, resp.  $\underline{\widehat{\CA}}(-, X)$ and $\underline{\widehat{\CA}}(-, \underline{f})$,  respectively. The same replacement for the corresponding notations related to $\mmod \CB$, resp. $\underline{\rm{mod}}\mbox{-}\CB$,  will also be considered.  Further, we also follow the same convention for the covariant case,  for instance we write $\widehat{\CA}(X, -)$ instead of  $(\mmod \CA)^{\rm{op}}(X, -)$.\\

	The following result from \cite{P}  provides a nice connection between the functor $\Phi$ and $\rm{H}F_{\la}$.
	 
\begin{proposition}\cite[Proposition 5.3]{P}\label{Prop 4.2}
	Assume that $X, Y \in \mmod \CA$ and $f:X \rt Y$ is an $\CA$-homomorphism. Let $T=\text{Coker}(\widehat{\CA}(-, X)\st{\widehat{\CA}(-, f)}\rt \widehat{\CA}(-, Y))$. Then there is a natural isomorphism $\Phi(T)\simeq \text{Coker}(\widehat{\CA}(-, X)\st{\widehat{\CA}(-, f)}\rt \widehat{\CA}(-, Y)).$ Consequently, there is the following commutative diagram (up to isomorphism of functors) 	\[\xymatrix{ \rm{H}(\mmod \CA)\ar[rr]^{\Theta_{\mmod \CA}} \ar[d]^{\rm{H}F_{\la}} && \CF(\mmod \CA) \ar[d]^{\Phi} \\ \rm{H}(\mmod \CB) \ar[rr]^{\Theta_{\mmod \CB}} && \CF(\mmod \CB). }\]	
\end{proposition}

 Observe that the group $G$ acts on $\CF(\mmod \CA)$. Indeed, given a functor $T \in \CF(\mmod \CA)$ and $g \in G$,  we define ${}^gT \in \CF(\mmod \CA)$ (see \cite[Section 5]{P} for an argument 
 to show that  it belongs to $\CF(\mmod \CA)$)  such that ${}^gT(X)=T({}^{g^{-1}}X)$ and ${}^gT(f)=T({}^{g^{-1}}f)$, for any object $X \in \mmod \CA$  and $\CA$-homomorphism $f$. Note that ${}^{g^{-1}}X$ and ${}^{g^{-1}}f$ are defined by the action $G$ on $\mmod \CA$, which  introduced earlier.

 In the sequel, a similar connection between $\CS\rm{F}_{\la}$ and an restriction of the functor $\Phi$ will be considered. Let us first give the following characterization.
 
 let $\CC$ be an abelian category with enough projectives. Due to \cite[Proposition 3.2]{Au1}  we have the  following categorical characterization  of the subcategory $\CF(\underline{\CC})$:  let $T$ be in $\CF(\CC)$. Then 
 $$T \in \CF(\underline{\CC}) \Leftrightarrow \CF(\CC)(T, \CC(-, X))=0 \ \text{for any} \ X \in \CC.$$
 For any $g \in G$, $P \in \CF(\underline{\rm{mod}}\mbox{-}\CA)$ and $X \in \mmod \CA$, by the above characterization we get
 $$\CF(\mmod \CA)({}^gP, \widehat{\CA}(-, X)) \simeq \CF(\mmod \CA)(P, \widehat{\CA}(-, {}^gX))=0. $$ 
 Hence the characterization implies that ${}^gP$ lies in $\CF(\underline{\rm{mod}}\mbox{-}\CA)$. This means that the action $G$ restricts to an action on $\CF(\underline{\rm{mod}}\mbox{-}\CA)$. Therefore, the $G$-precovering $\Phi$ restricts to a $G$-precovering $\Phi\mid:\CF(\underline{\rm{mod}}\mbox{-}\CA)\rt  (\underline{\rm{mod}}\mbox{-}\CB)$. \\
 We need to recall the following  construction from \cite[Construction 3.1]{H}.
 \begin{construction}\label{FirstCoonstr}
 	
 	Let $\CC$ be an abelian category with enough projectives. 	Taking an object $(A \st{f} \rt B)$ of $\CS(\CC)$, by applying the Yoneda functor on the following short exact sequence 
 	$$0 \rt A \st{f} \rt  B \rt \text{Coker}( f) \rt 0$$ we obtain   
 	the following short exact sequence
 	$$ (*) \ \ \ 0 \lrt \CC(-, A) \st{\CC(-, f )} \rt \CC(-, B) \rt \CC(-, \text{Coker}(f)) \rt F \rt 0$$
 	in $\CF(\mathcal{C})$. In fact, $(*)$  corresponds to a projective resolution of $F$ in $\mmod \mathcal{C}.$ We define $\Psi_{\CC}(A \st{f} \rt B):= F$. 
 	
 	For morphism: Let $\sigma=(\sigma_1, \sigma_2)$ be a morphism from $(A \st{f} \rt B)$ to $(A' \st{f'} \rt B')$, it gives the following commutative  diagram
 	\[\xymatrix{0 \ar[r] & A \ar[r]^{f} \ar[d]^{\sigma_1} & B \ar[r] \ar[d]^{\sigma_2} & \text{Coker}(f)  \ar[r] \ar[d]^{\sigma_3} & 0 \\
 		0 \ar[r] & A'  \ar[r]^{f'} & B' \ar[r] & \text{Coker}(f')  \ar[r]  & 0.}\]
 	By applying the Yoneda functor, the above diagram gives  the following commutative  diagram
 	
 	\[\xymatrix{0 \ar[r] & \CC(-, A) \ar[d]_{\CC(-, \sigma_1)} \ar[r]^{\CC(-, f)} & \CC(-, B) \ar[d]^{\CC(-, \sigma_2)} \ar[r] & \CC(-, \text{Coker}(f)) \ar[d]^{\CC(-, \sigma_3)} \ar[r] & F \ar[d]^{\overline{\CC(-, \sigma_3)}} \ar[r] & 0 \\ 0  \ar[r] & \CC(-, A')\ar[r]^{\CC(-, f')} & \CC(-, B')\ar[r] & \CC(-,\text{Coker}(f') )\ar[r] & F' \ar[r] & 0.}\]
 	
 	in $\mmod \mathcal{C}.$ Define $\Psi_{\CC}(\sigma):= \overline{\CC(-,\sigma_3)}$.	
 \end{construction}
 \begin{proposition}\label{Prop 4.4}
 	There is  the following commutative diagram (up to isomorphism of functors)	\[\xymatrix{ \CS(\mmod \CA)\ar[rr]^{\Psi_{\mmod \CA}} \ar[d]^{\rm{H}F_{\la}} && \CF(\underline{\rm{mod}}\mbox{-}\CA) \ar[d]^{\Phi\mid} \\ \CS(\rm{mod}\mbox{-}\CB) \ar[rr]^{\Psi_{\mmod \CB}} && \CF(\underline{\rm{mod}}\mbox{-}\CB). }\]	
 \end{proposition}
 \begin{proof}
 	The proof is rather a  direct consequence of \cite[Proposition 5.3]{P}. Let $\mathbb{X}=(X_1\st{f}\rt X_2)$ be in $\CS(\mmod \CA)$.
 	
 	\begin{align*}
 	\Phi\mid\circ \Psi_{\mmod \CA} (\mathbb{X})&\simeq \Phi \circ \Theta_{\mmod \CA}((X_2\rt \text{Coker} (f)))\\
 	&\simeq \Theta_{\mmod \CB}\circ \text{H}F_{\la}((X_2\rt \text{Coker}(f)) )\\
 	& \simeq \Psi_{\mmod \CB}\circ \CS F_{\la}((X_1\st{f}\rt X_2)).
 	\end{align*}
 	The  above natural isomorphisms follow from \cite[Proposition 5.3]{P} (or Proposition \ref{Prop 4.2}) and the exactness of the functor $F_{\la}$.  	
 \end{proof}
For our use in the next section  we only provide the above  relationships between the appeared  functors in Propositions \ref{Prop 4.2} and \ref{Prop 4.4}.   As in the remark below we shall discuss the relationship between those functors is more complete than the ones of the above results.
\begin{remark}
	Let $\CC$ be an abelian category with enough projectives. Let $\CU_{\CC}$ denote the smallest additive subcategory of  $\rm{H}(\mmod \CA)$ containing all  objects of the form $(X\rt 0)$ and $(X\st{1}\rt X)$, where $X$ runs through $\CA$. Denote by $\rm{H}(\CC)/ \CU_{\CC}$ the quotient category of $\CC$ by $\CA$. The quotient category  $\rm{H}(\CC)/ \CU_{\CC}$ has  the same objects as $\CC$ while the morphisms for object $X$, $Y$ in $\CC$  defined by the quotient
	$$\Hom_{\rm{H}(\CC)/ \CU_{\CC}}(X, Y)=\Hom_{\CC}(X, Y)/\{\phi\mid  \phi \ \text{ factors through an object in }\ \CU_{\CC}\}.$$  One can prove that the functor $\Theta_{\CC}:\rm{H}(\CC)\rt \mmod \CC$ induces  an equivalence $\hat{\Theta}_{\CC}:\rm{H}(\CC)/\CU_{\CC}\rt \mmod \CC.$ In fact, with the same argument of \cite[Theorem 3.1.3]{AHS} we can show that the functor $\Theta_{\CC}$ is full, dense and objective. Further, its kernel object is $\CU_{\CC}$. As a direct consequence of \cite[Appendix]{RZ} we get the claim. Now we specialize our  above observation to the  case $\CC= \mmod \CA$ or $\mmod \CB.$ Since the subcategory $\CU_{\mmod \CA}$ is preserved by any $g \in G$, hence the group $G$ defines an action on $\rm{H}(\CC)/\CU_{\CC}.$ In view of Theorem  \ref{Theorem 3.2} one may show that the $G$-precovering $\rm{H}F_{\la}$ induces a $G$-precovering $\widehat{\rm{H}F_{\la}}:\rm{H}(\mmod \CA)/\CU_{\mmod \CA}\rt \rm{H}(\mmod \CB)/\CU_{\mmod \CB}$  between the corresponding quotient categories. Let us summarize the above observation in the following commutative diagram
	
	\[\xymatrix@C-1.pc@R-1pc{ & \rm{H}(\mmod \CA) \ar[rr]^<<<<<{\pi_{\mmod \CA}} \ar[dl]_{\Theta_{\mmod \CA}} \ar[dd]^{\rm{H}F_{\la}} && \text{H}(\mmod \CA)/\CU_{\mmod \CA}  \ar[dd]^{\widehat{\text{H}F_{\la}}}\ar[dlll]_{\simeq }\\
	\CF(\mmod \CA)\ar[dd]^{\Phi}  &&  &\\
		& \text{H}(\mmod \CB) \ar[rr]^<<<<<{\pi_{\mmod \CB} } \ar[dl]_{\Theta_{\mmod \CB}}  && \rm{H}(\mmod \CB)/\CU_{\mmod \CB} \ar[dlll]^{\simeq} \\
	\CF(\mmod \CB)  && &
	}\]
where $\pi_{\mmod \CA}$ and $\pi_{\mmod \CB}$ denote the corresponding quotient functors, and the equivalences are the induced functors $\widehat{\Theta}_{\mmod \CA}$ and $\widehat{\Theta}_{\mmod \CB}$. Hence the above diagram gives a complete feature of the diagram in Proposition \ref{Prop 4.2}. We have the same observation for the diagram included in Proposition \ref{Prop 4.4} as follows. Let $\CV_{\CC}$ denote the smallest additive subcategory of  $\CS(\mmod \CA)$ containing all  objects of the form $(0\rt X)$ and $(X\st{1}\rt X)$, where $X$ runs through $\CA$.  In \cite[Theorem 3.2]{H} is proved that the functor $\Psi_{\CC}$, see Construction \ref{FirstCoonstr}, induces an equivalence $\widehat{\Psi}_{\CC}:\CS(\CC)/\CV_{\CC}\rt \CF(\underline{\CC})$. 
 In the same way as the morphism categories  and in conjunction with the equivalence $\widehat{\Psi}_{\CC}$ for the cases $\CC=\mmod \CA$ and $\mmod \CB$, we obtain the following commutative diagram up to isomorphism of functors
	\[\xymatrix@C-1.pc@R-1pc{ & \CS(\mmod \CA) \ar[rr]^<<<<<{\pi'_{\mmod \CA}} \ar[dl]_{\Psi_{\mmod \CA}} \ar[dd]^{\CS \text{F}_{\la}} && \CS(\mmod \CA)/\CV_{\mmod \CA}  \ar[dd]^{\widehat{\CS \text{F}_{\la}}}\ar[dlll]_{\simeq }\\
	\CF(\underline{\rm{mod}}\mbox{-}\CA)\ar[dd]^{\Phi\mid}  &&  &\\
	& \CS(\mmod \CB) \ar[rr]^<<<<<{\pi'_{\mmod \CB} } \ar[dl]_{\Psi_{\mmod \CB}}  && \CS(\mmod \CB)/\CV_{\mmod \CB} \ar[dlll]^{\simeq} \\
	\CF(\underline{\rm{mod}}\mbox{-}\CB)  && &
}\]
where the equivalences are the induced equivalences $\widehat{\Psi}_{\mmod \CA}$ and $\widehat{\Psi}_{\mmod \CB}$ by $\Psi_{\mmod \CA}$ and $\Psi_{\mmod \CB}$, respectively.
\end{remark}

\section{ The induced $G$-precovering $\rm{H}F_{\la}$ and $\CS \rm{F}_{\la}$ preserve the almost split sequences}\label{Section 5}
The main objective of this section is to show that the functors $\rm{H}\rm{F}_{\la}$ and $\CS \rm{F}_{\la}$ preserve the almost split sequences. This section is divided into three subsections. In the first subsection, we will provide the general structure of almost split sequence ending at  some kinds of objects in the morphism and monomorphism categories over a  Krull-Schmidt abelian category $\CC$ with enough projectives and injectives. Moreover, we will discuss a nice connection between the almost split sequences in $\rm{H}(\CC)$ and $\CS(\CC)$ and the ones in $\CF(\CC)$ and $\CF(\underline{C})$, respectively.  In the last two subsections, we specialize the general results for our special  cases ($(\CS)\rm{H}(\mmod \CA)$ and $(\CS)\rm{H}(\mmod \CB)$) to achieve  our main results. 

\subsection*{Setup} From now on to the end of the paper, we assume that the action of $G$ on $\CA$ have only finitely many $G$-orbits. In this case, the category $\CB$ (or the orbit catgery $\CA/G$) is finite and we treat it as an algebra. Further, $\mmod \CB$ is of finite representation type. Hence we can identify $\CF(\mmod \CB)$ and $\CF(\underline{\rm{mod}}\mbox{-}\CB)$ by the category of finitely generated modules over the  Auslander  algebra  and stable Auslander algebra of $\CB$, respectively.

\subsection{Almost split sequences in (Mono)morphism category over an abelian category }
Let $\CC$ be a Krull-Schmidt category (throughout this section) and $f:X\rt Y$ a morphism in $\CC$. Recall that $f$ is {\it left minimal} if every factorization $f = hf$ implies that $h$ is an automorphism of X; and left almost split if $f$ is not a section and every non-section morphism $g : X \rt Z$ factors through $f$. The concepts of {\it right minimal}  and {\it right almost split } are defined dually. Observe that $X $ or  $Y$ is indecomposable in case $f$ is left almost or right almost, respectively; see \cite{Ar1}.\\ Assume $C$ ia an abelian category and $\mathcal{D}$ an extension-closed subcategory of $\CC$. A short exact sequence $\delta: 0\rt X  \st{f}\rt Y\st{g}\rt Z\rt 0$ in $\mathcal{D}$ is said to be {\it almost split } if $f$ and $g$  are  a minimal left almost split morphism and   a minimal right almost split morphism in $\mathcal{D}$, respectively. Since $\delta $ is unique up to isomorphism for $X$ and $Z$, we may write $X=\tau_{\mathcal{D}}Z$ and $Z=\tau^{-1}_{\mathcal{D}}X$. We call $X$ the Auslander-Reiten translation of $Z$ in $\mathcal{D}$. 
We say that an object $Z$ in $\mathcal{D}$ is {\it relative-projective} if any exact sequence $0 \rt X\rt Y \rt Z\rt 0$ in $\mathcal{D}$ splits. Dually, we define {\it relative-injective} objects.  We shall say that $\mathcal{D}$ has {\it right  almost  split}   sequences
if every indecomposable object is either  relative projective or the ending term of an almost split sequence; dually, $\mathcal{D}$ has {\it left almost split} sequences
if every indecomposable object is either  relative injective or the starting term of an almost split sequence. We call  $\mathcal{C}$ has {\it almost  split}  sequences if it has both left and right almost split sequences.\\
Recall  that a Hom-finite $k$-linear  category $\CC$ is called a {\it dualizing $k$-variety} \cite{AR3} if the functors $D:\Mod \CC \rt \Mod \CC^{\rm{op}}$ and $D:\Mod \CC^{\rm{op}}\rt \Mod \CC$, induced by the $k$-dual standard $D=\Hom_k(-, k)$, restricts to the following dualities

$$D:\CF(\CC)\rt  \CF(\CC^{\rm{op}})\  \text{and} \  D:\CF(\CC^{\rm{op}})\rt \CF(\CC).$$

By abuse of notation all the aforementioned functors are denoted by the same letter $D$. We sometimes denote by $D_{\CC}$ if we want to make distinguish between the dualities for different abelian categories.  In this case, $\CF(\CC)$ is an abelian subcategory of $\Mod \CC$ which is closed under kernels, cokernels and extensions, and has enough projective objects and injective objects (more precisely, any object has projective cover, and equivalently has injective envelop).  Moreover, all simple objects in $\Mod \CC$ and $\Mod \CC^{\rm{op}}$ are finitely presented.  This implies that for every indecomposable object $X \in \CC$ there exists a right almost split morphism $f:Y\rt X $, and a left almost split morphism $g:X\rt Z,$ see \cite{Au3} for more details. Also, in the same way of the module category over a finite dimensional algebra \cite{AuslanreitenSmalo}, the Auslander-Reiten translation of a non-projective  indecomposable functor $M$ in $\CF(\CC)$  is computed as $\tau_{\CF(C)}\simeq D \circ \CN_{\CC}( \Ker(f_M))$, where $f_M:P_1\rt P_0$ a minimal projective resolution of $M$, and $\CN_{\CC}:\Mod \CC\rt \Mod \CC$ is the {\it Nakayama functor}. The Nakayama functor $\CN_{\CC}:\Mod \CC\rt \Mod \CC$ is the composite $D\circ (-)^*$, where $(-)^*:\Mod \CC\rt \Mod \CC^{\rm{op}} $ is a  contravariant functor and defined as follows: for $M \in \Mod \CC$ and $X \in \CC $, let
$$(M)^*(X):=\Mod\CC(M, \CC(-, X)).$$ 
By definition, one can see $\CN_{\CC}(\CC(-, X))\simeq D(\CC(X, -))$, for any $X \in \CC.$
For example, any locally bonded category $\mathcal{A}$ is a dualizing $k$-variety. To see this, apply the fact that $\CF(\CA)=\mmod \CA.$ If $\CC$ is a  dualizing $k$-variety, then  so is $\CF(\CC)$
 \cite[Proposition 2.6]{AR3}.
 
\subsection*{Setup} From now on to the end of this section, let $\CC$ be a semiperfect and Hom-finite abelian category with enough projectivs and injectivs and $\rm{H(\CC)}$ has almost split sequences. Hence $\CC$ is a Krull-Schmidt category.

 In \cite{RS2}, the notation $\CF(\CC)$ is used to denote the subcategory of $\rm{H}(\CC)$ consisting of all epimorphisms. Here since we use the notation  $\CF(\CC)$ for the category of finitely presented functors over $\CC$, hence in place of  $\CF(\CC)$   the notation $\rm{Epi}(\CC)$ is used. The kernel and cokernel functors  are defined as follows:
 $$\Ker:\rm{Epi}(\CC)\rt \CS(\CC), \ \ \ (A\st{f}\rt B) \mapsto (B \rt \rm{Coker}(f)),$$
 $$\rm{Cok}:\CS(\CC)\rt \rm{Epi}(\CC), \ \ \  (B\st{g}\rt C)\mapsto (\Ker(g)\rt B), $$
 induce a pair of inverse equivalences. To distinguish the above functors for a specific abelian category $\CC$, in case of need, we denote $\Ker_{\CC}$ and $\rm{Cok}_{\CC}$.
 
 Later we  will need the following preparatory  lemmas. For an object  $C \in \CC$, $\Omega_{\CC}(C)$ denotes the kernel of a projective cover of $C$.
\begin{lemma}\cite[Lemma 6.3]{H}\label{AlmostSplittrivialmonomorphisms}
	  Assume  $0 \rt  A\st{f} \rt B\st{g} \rt C \rt 0$ is  an almost split sequence in $\CC.$ Then
	\begin{itemize}
		\item[$(1)$] The almost split sequence in $\CS(\CC)$ ending at $(0\rt C)$ has the form 
		{\footnotesize  \[ \xymatrix@R-2pc {  &  ~ A\ar[dd]^{1}~   & A\ar[dd]^{f}~  & 0\ar[dd] \\ 0 \ar[r] &  _{ \ \ \ \ } \ar[r]^{1}_{f}  &_{\ \ \ \ \ } \ar[r]^{0}_{g} _{\ \ \ \ \ }&  _{\ \ \ \ \ }\ar[r] & 0. \\ & A & B &C}\]}
		\item [$(2)$] Let $e:A \rt I$ be a left minimal morphism with injective object $I$ (or injective envelop of $A$). Then the almost split sequence in $\CS(\CC)$ ending at $(C\st{1}\rt C)$ has the form
		{\footnotesize  \[ \xymatrix@R-2pc {  &  ~ A\ar[dd]^{e}~~   & B\ar[dd]^{h}~  & C\ar[dd]^1 \\ 0 \ar[r] &  _{ \ \ \ \ } \ar[r]^{f}_{[1~~0]^t}  &_{\ \ \ \ \ } \ar[r]^{g}_{[0~~1]} _{\ \ \ \ \ }&  _{\ \ \ \ \ }\ar[r] & 0, \\ & I & I \oplus C &C}\]} 	  	 	
		where $h$ is the map $[e'~~g]^t$ with $e':B \rt I$ is an extension of $e.$
		\item [$(3)$] Let $b:P\rt C$ be the projective cover of $C$ in $\CC$, i.e., a right minimal morphism with projective object $P$ in $\CC$. Then the almost split sequence in $\CS(\CC)$ starting  at $(0 \rt A)$ has the form 
		{\footnotesize  \[ \xymatrix@R-2pc {  &  ~ 0\ar[dd]   & \Omega_{\CC}(C)\ar[dd]^{h}~  & \Omega_{\CC}(C)\ar[dd]^i \\ 0 \ar[r] &  _{ \ \ \ \ } \ar[r]^{0}_{[1~~0]^t}  &_{\ \ \ \ \ } \ar[r]^{1}_{[0~~1]} _{\ \ \ \ \ }&  _{\ \ \ \ \ }\ar[r] & 0, \\ & A & A \oplus P &P}\]}
		where $h$ is  the kernel of morphism $[f~~b']:A\oplus P\rt B$, here $b'$ is a lifting of $b$ to $g$. 
		
	\end{itemize}
\end{lemma}
 The condition in our set up which $\rm{H}(\CC)$ has almost split sequences is essential in the following lemma.
\begin{lemma}\label{Lemma 5.9}
	The monomorphism category $\CS(\CC)$ has almost split sequences.
\end{lemma}
\begin{proof}
	The same argument in \cite[Section 2]{RS2} works to prove that $\CS(C)$ is a functorially finite subcategory in $\rm{H}(\CC)$. Now by our assumption we know that $\rm{H}(\CC)$ has almost split sequences. Then by use of \cite[Theorem 2.4]{ASm} we get the result.
\end{proof}
\begin{lemma}\label{Lemma 5.2}
	  Assume  $0 \rt  A\st{f} \rt B\st{g} \rt C \rt 0$ is  an almost split sequence in $\CC.$ Then
	\begin{itemize}
		\item[$(1)$] The almost split sequence in $\rm{H}(\CC)$ ending at $(0\rt C)$ has the form 
		{\footnotesize  \[ \xymatrix@R-2pc {  &  ~ A\ar[dd]^{1}~   & A\ar[dd]^{f}~  & 0\ar[dd] \\ 0 \ar[r] &  _{ \ \ \ \ } \ar[r]^{1}_{f}  &_{\ \ \ \ \ } \ar[r]^{0}_{g} _{\ \ \ \ \ }&  _{\ \ \ \ \ }\ar[r] & 0. \\ & A & B &C}\]}
		\item [$(2)$] The almost split sequence in $\rm{H}(\CC)$ ending at $(C\st{1}\rt C)$ has the form 
		{\footnotesize  \[ \xymatrix@R-2pc {  &  ~ A\ar[dd]^{}~   & B\ar[dd]^{g}~  & C\ar[dd]^{1} \\ 0 \ar[r] &  _{ \ \ \ \ } \ar[r]^{f}_{0}  &_{\ \ \ \ \ } \ar[r]^{g}_{1} _{\ \ \ \ \ }&  _{\ \ \ \ \ }\ar[r] & 0. \\ & 0 & C &C}\]}
		\end{itemize}
\end{lemma}
\begin{proof}
$(1)$ Let $(0, d):(X\st{v}\rt Y)\rt (0\rt C)$ be a non-retraction. Then the morphism $(0, d):(0 \rt Y)\rt (0\rt C)$ in $\CS(\CC)$ is a non-retraction as well. Because of the first part of Lemma \ref{AlmostSplittrivialmonomorphisms}
  we know that the short exact sequence in the statement is almost split in $\CS(\CC)$, hence there exists $(0, t):(0\rt Y)\rt (A\st{f}\rt B)$ such that $(0, g)\circ (0, t)=(0, d)$. Since $gtv=0$, there is a morphism $s:X\rt A$ with $fs=tv.$ Now we see that the morphism $(0, d)$ factors
	through $(0, g)$ via $(s, t)$. So we are done this case.\\
	$(2)$ By applying the cokernel functor $\rm{Cok}$ on the almost sequence in $\CS(\CC)$ ending at $(0\rt C)$, as stated in Lemma \ref{AlmostSplittrivialmonomorphisms}(1) we get the short exact sequence in $(2)$ of this lemma, that is a short exact sequence in $\rm{Epi}(\CC)$. Since the cokernel functor is an equivalence, the  short exact sequence in $(2)$  is an almost split in $\rm{Epi}(\CC)$. Using the same argument of the first part of this lemma we can prove that it is also an almost split in $\rm{H}(\CC)$, as desired.
\end{proof}
\begin{definition} An indecomposable object  $(A\st{f}\rt B)$ in $\rm{H}(\CC)$ is called {\it Mor-proper} if it is not isomorphic to an object of the form either $(C\rt 0)$, $(C \st{1}\rt C)$ or $(0 \rt C)$.
\end{definition}
We have the following easy lemma.
\begin{lemma}\label{Lemma 5.3}
Let $\mathbb{A}=(A\st{f}\rt B)$ be an indecomposable  Mor-proper object in $\rm{H}(\CC)$. Then $\Theta_{\CC}(\mathbb{A})$  is an  indecomposable non-projective object in $\CF(\CC)$. Moreover, $\CC(-, A)\st{\CC(-, f)}\rt \CC(-, B)\rt\Theta_{\CC}(\mathbb{A})\rt0 $ is a minimal projective presentation in $\CF(\CC)$.
\end{lemma}

\begin{construction}\label{Construction 5.4}
Let $\mathbb{A}$ be a  Mor-proper object in $\rm{H}(\CC)$. Our assumption on $\rm{H}(\CC)$ implies that the existence of  an almost split sequence 
$$\delta: \ \ 0 \rt \mathbb{D}\st{\varrho}\rt \mathbb{C}\st{\xi}\rt \mathbb{A}\rt 0$$
in $\rm{H}(\CC)$. Set $\mathbb{D}=(D\st{w}\rt E)$, $\mathbb{C}=(C\st{v}\rt L)$, $\varrho=(\varrho_1, \varrho_2)$ and $\xi=(\xi_1, \xi_2)$. We  claim that the short exact sequences
$$0 \rt D\st{\varrho_1}\rt C\st{\xi_1}\rt A\rt 0$$
$$0 \rt E\st{\varrho_2}\rt L\st{\xi_2}\rt B\rt 0$$
are split. To prove the claim, consider the test maps: $(1, f):(A\st{1}\rt A)\rt (A\st{f}\rt B)$ and $(f, 1):(B\st{1}\rt B)\rt (A\st{f}\rt B).$ One can verify that those test maps are not split epimorphisms by use this fact that $\mathbb{A}$ is Mor-proper. By the property of being almost split of $\delta$, those test maps factor through $\xi$. The obtained factorizations imply our claim.	Now, by the Snake lemme we have the following commutative diagram in $\CC$ with the splitting lower two rows
$$\xymatrix{& 0 \ar[d] & 0 \ar[d] & 0 \ar[d]&&\\
	0 \ar[r] & \text{Ker}(w) \ar[d] \ar[r] &  \text{Ker} (v) \ar[d]
	\ar[r]^j & \text{Ker}(f) \ar[d]  &\\
	0 \ar[r] & D\ar[d]^w \ar[r] & C
	\ar[d]^v \ar[r] & A \ar[d]^f \ar[r] & 0\\
	0 \ar[r] & E  \ar[r] & L
 \ar[r] & B \ \ar[r] & 0 }
$$
Take the morphism $(i, 0):(\text{Ker}(f)\st{i}\rt 0)\rt \mathbb{A}$. Since again $\mathbb{A}$ is Mor-proper, one can see that the morphism $(i, 0)$ is not split epimorphism. Hence it factors through $\xi$. The factorization yields that $j$ is a split epimorphism. Hence the all rows in the above diagram are split. This observation gives us the following commutative diagram in $\CF(\CC)$ after applying the functor $\Theta_{\CC}$ on the above diagram
$$\xymatrix{& 0 \ar[d] & 0 \ar[d] & 0 \ar[d]&\dagger&\\
	0 \ar[r] & \CC(-, \text{Ker}(w)) \ar[d] \ar[r] &  \CC(-, \text{Ker} (v)) \ar[d]
	\ar[r]^j & \CC(-, \text{Ker}(f)) \ar[d]\ar[r] &0\\
	0 \ar[r] &\CC(-, D)\ar[d] \ar[r] & \CC(-, C)
	\ar[d]\ar[r] & \CC(-, A) \ar[d] \ar[r] & 0\\
	0 \ar[r] &\CC(-, E) \ar[d] \ar[r] & \CC(-, L)
	\ar[d] \ar[r] & \CC(-, B) \ar[d] \ar[r] & 0\\
	0\ar[r]	& \Theta_{\CC}(\mathbb{D}) \ar[r]\ar[d] & \Theta_{\CC}(\mathbb{C}) \ar[r]\ar[d]  & \Theta_{\CC}(\mathbb{A}) \ar[r]\ar[d] &0\\ & 0  & 0  & 0 &  }
$$
The lower sequence is indeed the image, $\Theta_{\CC}(\delta)$, of $\delta$ under the functor $\Theta_{\CC}$. This means that $\Theta_{\CC}$ keeps the exactness of $\delta.$
\end{construction}

\begin{proposition}\label{Proposition 5.6}
Keep in mind  the notation in Construction \ref{Construction 5.4}. The short exact sequence $\Theta_{\CC}(\delta)$ is an almost split sequence in $\CF(\CC)$.
\end{proposition}
\begin{proof}
By Lemma \ref{Lemma 5.3},  we first note 
that $\Theta_{\CC}(\mathbb{A})$ and $\Theta_{\CC}(\mathbb{D})$ are indecomposable. $\Theta_{\CC}(\mathbb{A})$ also is not split; otherwise, $\delta$ becomes split that is a contradiction. Indeed, if it was split, then $\CC(-, f)\oplus \CC(-, w)$ would be a minimal projective presentation for $\Theta_{\CC}(\mathbb{C})$. On the other hand, we observe by the diagram $(\dagger)$ in the construction another projective preservation $\CC(-, C)\st{\CC(-, v)}\rt \CC(-, L)\rt \Theta_{\CC}(\mathbb{C})\rt 0.$ Comparing these two projective presentations for $\Theta_{\CC}(\mathbb{C})$ and applying the Yoneda lemma provide us the following isomorphism in $\rm{H}(\CC)$
\[(C\st{v}\rt L)\simeq (D\st{w}\rt E)\oplus(A\st{f}\rt B)\oplus (X\st{1}\rt X)\oplus (Y\rt 0) \]
for some $X$ and $Y$ in $\CC.$ We know that all rows in the diagram $(\dagger)$ except the lower one are split. Hence $L\simeq E\oplus B$ and $C\simeq D\oplus A.$ Because of Krull-Schmidt property of $\CC$ and minimality of the morphisms $f$ and $w$, we deduce that $X=Y=0$. So the middle term in $\delta$ is a direct sum of the ending terms, that is a contradiction as we claimed.
Now, invoking \cite[Theorem 2.14]{Ar1},   it is enough to show that $\Theta_{\CC}(\varrho)$  and $\Theta_{\CC}(\xi)$ respectively are left and right almost split. We will do it only for $\Theta_{\CC}(\varrho)$. Let $q:H\rt \Theta_{\CC}(\mathbb{A})$ be a non-retraction and a projective presentation $\CC(-, T)\st{\CC(-, y)}\rt \CC(-, M)\rt H\rt 0$. Hence $\Theta_{\CC}((T\st{y}\rt M))=H.$ The morphism $q$ can be lifted to  a morphism between the   projective presentations. The lifted morphism induces,  via the Yoneda lemma,     the following morphism  
{\footnotesize \[ \xymatrix@R-2pc {  & T \ar[dd]^{y}  & A\ar[dd]~  \\   &  _{\ \ \ \ \ \  \ \ \  }\ar[r]^{w_1}_{w_2}  \ar[r]&_{\ \ \ \ \ }{f}   \\ &M & B}\]}
in $\rm{H}(\CC)$ such that $\Theta_{\CC}((w_1, w_2))=q$. The morphism $(w_1, w_2)$ is not retraction. Otherwise, it follows $q$ so is, a contradiction. Since $\delta$ is an almost split sequence,  $(w_1, w_2)$ factors thorough $\xi$ via
{\footnotesize \[ \xymatrix@R-2pc {  & T \ar[dd]^{y}  & C\ar[dd]~  \\   &  _{\ \ \ \ \ \  \ \ \  }\ar[r]^{l_1}_{l_2}  \ar[r]&_{\ \ \ \ \ }{v}   \\ &M & L}\]}
Now by applying the functor $\Theta_{\CC}$ on such a factorization, we see that the morphism $q$ factors through $\Theta_{\CC}(\xi)$ via $\Theta_{\CC}((l_1, l_2))$, as required. So we proved the claim.
\end{proof}	

\begin{definition}
	An indecomposable object $(A\st{f}\rt B)$ in $\CS(\CC)$ is said to be {\it Mono-proper} if it is not isomorphic to $(0 \rt C)$, $(C\st{1}\rt C)$ and $(\Omega_{\CC}(C)\hookrightarrow P)$, where $C$ is an indecomposable object in $\CC$ and $P$ a projective cover of $C$ in $\CC.$
\end{definition}
Hence, by definition, the objects of the forms $(0 \rt C)$ and  $(C\st{1}\rt C)$ are not both Mor-
and Mono-proper. \\
In the rest of this subsection we shall provide a similar way to construct the almost split sequences in $\CF(\underline{\CC})$ via the ones in $\CS(\CC)$ ending at an indecomposable Mono-proper object. 
\begin{construction}\label{ConstructionMonoSplit}
	Let $\mathbb{A}=(A\st{f}\rt B)$ be an indecomposable  Mono-proper object in $\CS(\CC)$.	
	 By Lemma \ref{Lemma 5.9} there is an almost split sequence in $\CS(\CC)$ ending at $\mathbb{A}$, namely, 
	{\footnotesize  \[ \xymatrix@R-2pc {  &  ~ X_1\ar[dd]^{{d}}~   & Z_1\ar[dd]^{h}~  & A\ar[dd]^{f} \\ \epsilon: \ \ 0 \ar[r] &  _{ \ \ \ \ } \ar[r]^{\phi_1}_{\phi_2}  &_{\ \ \ \ \ } \ar[r]^{\psi_1}_{\psi_2} _{\ \ \ \ \ }& _{\ \ \ \ \ \ }\ar[r] & 0 \\ & X_2 & Z_2 &B }\]}
Set $\mathbb{X}=(X_1\st{d}\rt X_2), \mathbb{Z}=(Z_1\st{h}\rt Z_2)$ and $\phi=(\phi_1, \phi_2)$, $\psi=(\psi_1, \psi_2)$.	By expanding the above diagram in $\CC$ we get the following commutative diagram
		$$\xymatrix{& 0 \ar[d] & 0 \ar[d] & 0 \ar[d]&  &\\
		0 \ar[r] & X_1 \ar[d]^{d} \ar[r]^{\phi_1} & Z_1 \ar[d]^{h}
		\ar[r]^{\psi_1} & A \ar[d]^{f} \ar[r] & 0\\
		0 \ar[r] & X_2\ar[d] \ar[r]^{\phi_2} & Z_2
		\ar[d] \ar[r]^{\psi_2} & B \ar[d] \ar[r] & 0\\
		0 \ar[r] & \text{Coker}(d) \ar[d] \ar[r]^{\mu_1} & \text{Coker}(h)
		\ar[d] \ar[r]^{\mu_2} & \text{Coker}(f) \ar[d] \ar[r] & 0\\
		& 0  & 0  & 0 & }
	$$
As discussed in \cite[Construction 5.4]{H}, one can see that all rows in the above diagram are split. The splitting rows allow us to obtain the following commutative diagram in $\CF(\CC)$ after applying the Yoneda functor on the above diagram  

	$$\xymatrix{& 0 \ar[d] & 0 \ar[d] & 0 \ar[d]&  &\\
	0 \ar[r] &\CC(-, X_1) \ar[d]^{\CC(-, d)} \ar[r]^{\CC(-, \phi_1)} & \CC(-, Z_1) \ar[d]^{\CC(-, h)}
	\ar[r]^{\CC(-, \psi_1)} & \CC(-, A) \ar[d]^{\CC(-, f)} \ar[r] & 0\\
	0 \ar[r] & \CC(-, X_2)\ar[d] \ar[r]^{\CC(-, \phi_2)} & \CC(-, Z_2)
	\ar[d] \ar[r]^{\CC(-, \psi_2)} & \CC(-, B) \ar[d] \ar[r] & 0\\
	0 \ar[r] & \CC(-, \text{Coker}(d)) \ar[d] \ar[r]^{\CC(-, \mu_1)} &\CC(-,  \text{Coker}(h))
	\ar[d] \ar[r]^{\CC(-, \mu_2)} &\CC(-,  \text{Coker}(f)) \ar[d] \ar[r] & 0\\
	0\ar[r]& \Psi_{\CC}(\mathbb{X})\ar[r]\ar[d]  & \Psi_{\CC}(\mathbb{Z})\ar[r] \ar[d]& \Psi_{\CC}(\mathbb{A})\ar[r]\ar[d] & 0\\ &0&0&0&}
$$
The above diagram shows that the image, $\Psi_{\CC}(\epsilon)$, of the short exact sequence $\epsilon$ under the functor $\Psi_{\CC}$ remains exact.
\end{construction}

\begin{proposition}\label{Proposition 5.10}
	Keep in mind  the notation in Construction \ref{ConstructionMonoSplit}. The short exact sequence $\Psi_{\CC}(\epsilon)$ is an almost split sequence in $\CF(\underline{C})$.	
\end{proposition}
\begin{proof}
	 See  \cite[Proposition 5.6]{H}.
\end{proof}
\begin{remark}\label{Rem5.11}
	Assume further $\CC$ is a $k$-dualizing variety. By definition, one can see easily the following commutative
	diagram
	\[ \xymatrix{  \CF(\CC) \ar[r]^{D} & \CF(\CC^{\rm{op}}) \\
		\CF(\underline{\CC}) \ar@{^(->}[u] \ar[r]^D & \CF((\underline{\CC})^{\rm{op}})  \ar@{^(->}[u] }\]
	Since we have the full dense functor $\mathcal{\CC}\rt \underline{\CC}$, we identify $\CF((\underline{\CC})^{\rm{op}})$	with the full subcategory of $\CF(\CC^{\rm{op}})$ consisting of those functors vanish on the projective objects in $\CC.$ The above diagram implies that $\underline{\CC}$ is a $k$-dualizing variety. Hence $\CF(\underline{\CC})$ has almost split sequences (see \cite[Propsition 2.5]{AR3}).
\end{remark}

\subsection{The case of $\rm{H}F_{\la}$} In this subsection we shall prove that the functor $\rm{H}F_{\la}$ preserves almost split sequences in $\rm{H}(\mmod \CA)$. Our proof is divided into two cases that the ending term is either  a Mor-proper or non-Mor-proper object in $\rm{H}(\mmod \CA)$.\\

\begin{proposition}\label{Proposition 5.11}
	Assume that $\eta:0 \rt \mathbb{A}\st{p}\rt \mathbb{B}\st{q}\rt \mathbb{C}\rt 0$ is 
	 an almost split sequence in $\rm{H}(\mmod \CA)$ ending at a 
	 non-Mor-proper object in $\rm{H}(\mmod \CA)$. Then, the short exact sequence $\rm{H}F_{\la}(\eta): 0 \rt \rm{H}F_{\la}(\mathbb{A})\st{\rm{H}F_{\la}(p)}\lrt \rm{H}F_{\la}(\mathbb{B})\st{\rm{H}F_{\la}(q)}\rt \rm{H}F_{\la}(\mathbb{C})\rt 0$ is an almost split sequence in $\rm{H}(\mmod \CB).$
\end{proposition}
\begin{proof}
We prove case by case  for when   that $\mathbb{C}$ is isomorphic to either proper object $(0 \rt C)$, $(C\st{1}\rt C)$ or proper object $(C\rt 0)$,  for some indecomposable object $C$ in $\mmod \CA$.  	The case $\mathbb{C}\simeq (0\rt C):$ we observe by Lemma \ref{Lemma 5.2} the almost split $\eta$ is isomorphic to a short exact sequence of  the following form
	{\footnotesize  \[ \xymatrix@R-2pc {  &  ~ A\ar[dd]^{1}~   & A\ar[dd]^{f}~  & 0\ar[dd] \\ 0 \ar[r] &  _{ \ \ \ \ } \ar[r]^{1}_{f}  &_{\ \ \ \ \ } \ar[r]^{0}_{g} _{\ \ \ \ \ }&  _{\ \ \ \ \ }\ar[r] & 0. \\ & A & B &C}\]}
where $0 \rt A\st{f}\rt B \st{g}\rt C\rt 0$ is an almost split sequence in $\mmod \CA.$ Applying the functor $\rm{H}F_{\la}$ on the above sequence  we reach to the following short exact sequence
	{\footnotesize  \[ \xymatrix@R-2pc {  &  ~ F_{\la}(A)\ar[dd]^{1}~   & F_{\la}(A)\ar[dd]^>>>>{F_{\la}(f)}~  & 0\ar[dd] \\ 0 \ar[r] &  _{ \ \ \ \ } \ar[r]^{1}_{F_{\la}(f)}  &_{\ \ \ \ \ } \ar[r]^{0}_{F_{\la}(g)} _{\ \ \ \ \ }&  _{\ \ \ \ \ }\ar[r] & 0. \\ & F_{\la}(A) & F_{\la}(B) &F_{\la}(C)}\]}

 We know by \cite[Theorem 3.6]{G} the short exact sequence $0 \rt F_{\la}(A)\st{F_{\la}(f)}\rt F_{\la}(B)\st{F_{\la}(g)}\rt F_{\la}(C)\rt 0$ is an almost split sequence in $\mmod \CB.$ Hence, again by Lemma \ref{Lemma 5.2} we infer that the latter short exact sequence in the above is an almost split sequence in $\rm{H}(\mmod \CB)$ ending at $(0 \rt F_{\la}(C))\simeq F_{\la}(\mathbb{C})$. By the uniqueness of almost split sequences we deduce the result. The case $\mathbb{C}\simeq (C\st{1}\rt C)$: This case is proved similarly, in this case  we use the structure  of almost split sequence ending at $(X\st{1}\rt X)$ given in Lemma \ref{Lemma 5.2}. \\
 For the last case $\mathbb{C }\simeq (C\rt 0)$  we need the following fact (see \cite[Proposition 3.1]{MO}) : For $(M\rt 0)$ in $\rm{H}(\mmod \CA)$, $\tau_{\rm{H}(\mmod \CA)}(M\rt 0)\simeq (\CN_{ \CA}(P_2)\st{\CN_{\CA}(q)}\rt \CN_{ \CA}(P_1))$, where $P_2\st{q}\rt P_1\rt M\rt 0$ is a minimal projective resolution in $\mmod \CA.$ The same fact is also true for $(M\rt 0)$ given in $\rm{H}(\mmod \CB)$. \\
 Take a minimal projective resolution $Q_2\st{p}\rt Q_1\rt C$ in $\mmod \CA$. The fact follows that $\tau_{\rm{H}(\mmod \CA)}(C\rt 0)\simeq (\CN_{ \CA}(Q_2)\st{\CN_{\CA}(p)}\rt \CN_{ \CA}(Q_1))$. Here $\CN_{\CA}$ denotes the Nakayama functor of $\mmod \CA$. Hence we have 
 $$\text{H}\text{F}_{\la}(\mathbb{A})\simeq \text{H}\text{F}_{\la}(\tau_{\text{H}(\mmod \CA)}(C\rt 0))\simeq (F_{\la}\CN_{ \CA}(Q_2)\st{F_{\la}\CN_{\CA}(p)}\rt F_{\la}\CN_{ \CA}(Q_1)) .$$
 Using these points that the push-down functor $F_{\la}$ preserves the minimality and the  Nakayama functor ($F_{\la}\CN_{ \CA}\simeq \CN_{\CB}F_{\la}$) and   use again of the aforementioned fact and of course the  above isomorphisms we get 
 $$\text{H}\text{F}_{\la}(\mathbb{A})\simeq (\CN_{ \CB}F_{\la}(Q_2)\st{\CN_{\CB}F_{\la}(p)}\rt \CN_{ \CB}F_{\la}(Q_1))\simeq \tau_{\CF(\mmod \CB)}(F_{\la}(C)\rt 0)\simeq \tau_{\CF(\mmod \CB)}(\text{H}\text{F}_{\la}(\mathbb{C}))$$
 Now by following the same argument given in pages 90-91 of \cite{G}, we can prove $\text{H}\text{F}_{\la}(\eta)$ is an almost split sequence. Indeed, to see $\rm{H}F_{\la}(\eta)$ is an almost split sequence, now since we know  by the above $\tau_{\CF(\mmod \CB)}(\text{H}\text{F}_{\la}(\mathbb{C}))\simeq \text{H}\text{F}_{\la}(\mathbb{A})$ and using this fact $\rm{H}F_{\la}(\eta)$ does not split due to \cite[Lemma 2.7]{BL},   it suffices by \cite[Proposition 2.2, page 148]{AuslanreitenSmalo} to prove that any non-automorphism over $\text{H}\text{F}_{\la}(\mathbb{A})$ factors through $\text{H}\text{F}_{\la}(p)$.  To do this,    the adjunction isomorphism 
 $$\mmod \CB(\text{H}\text{F}_{\la}(\mathbb{A}), \text{H}\text{F}_{\la}(\mathbb{A}) )\simeq \mmod \CA(\mathbb{A},\text{H}\text{F}_{\bullet}\circ \text{H}\text{F}_{\la}(\mathbb{A}) $$
 sends any endomorphism of $\text{H}\text{F}_{\la}(\mathbb{A})$ onto a family $(f_g)_{g \in G}:\mathbb{A}\rt \oplus_{g \in G}{}^g\mathbb{A}$. It can be seen easily $f$ is a non-authomorphism if and only if so is  $f_e$. In this case the family  $(f_g)_{g \in G}$ factors trough $p$ since $\eta$ is an almost split sequence. Thanks to the adjunction we observe that $f$ factors through $\text{H}\text{F}_{\la}(p)$, as desired. So we are done. 
\end{proof}

In the Mor-proper case, Construction \ref{ConstructionMonoSplit},  making a connection between the almost split sequences in  $\rm{H}(\CC)$ and $\CF(\CC)$, is  helpful. For the  use of the construction we need to prove that the morphism categories $\rm{H}(\mmod \CA)$ and $\rm{H}(\mmod \CB)$ have almost split sequences. The case $\rm{H}(\mmod \CB)$ is clear since it equivalent to the module category of triangular matrix algebra of some finite dimensional  algebra. For the case $\rm{H}(\mmod \CA)$ we prove in the following. Before proving we need the next lemma.\\
Let $\La$ be a finite dimensional algebra and $\CP(\mmod\La)$ denotes the subcategory of $\rm{H}(\mmod \La)$ consisting of all objects $(P\st{f}\rt Q)$ with $P$ and $Q$ projective modules in $\mmod \La.$
\begin{lemma}(\cite[Theorem 5.1]{Ra})\label{Lemma 5.14}
	The exact category $\CP(\mmod \La)$ has almost split sequences. 
\end{lemma}
\begin{remark}There is an alternative proof of the above lemma by showing that  $\CP(\mmod \La)$ is functorially finite subcategory in $\rm{H}(\mmod \La)$. Then one can get the result by \cite[Theorem 2.4]{ASm}. To show the first claim, one may use \cite[Theorem A]{EHHS} which implies that $\CP(\La)$ is contravariantly finite. Due to \cite[Corollary 0.3]{KS} since $\CP(\La)$ is resolving, it is also covariantly finite, so we get the claim.
	\end{remark}
\begin{lemma}
	The morphism category $\rm{H}(\mmod \CA)$ has almost split sequences.
\end{lemma}
\begin{proof}
Let $\mathbb{A}=(A\st{f}\rt B)$ be a non-projective object. Since $\mmod \CA$ is locally bounded, there are only finitely many indecomposable objects in $\mmod \CA$, say $\{X_1, \cdots X_n\}$, such that there is  a non-zero map either  from $A$ or $B$ to $X_i$ or vice versa. Denote $X=X_1\oplus \cdots\oplus X_n$ and $\mathcal{D}=\rm{add}\mbox{-}X$. Set $\La=\rm{End}_{\mmod \CA}(X)$. By the evaluation functor $e_X:\CF(\mathcal{D})\rt \mmod \La, L\mapsto L(X)$,  we can identify the functor category $\CF(\mathcal{D})$ with the module category $\mmod \La$. According to the Lemma \ref{Lemma 5.14} there exists the following almost split sequence in $\CP(\CF(\mathcal{D}))$
	{\footnotesize  \[ \xymatrix@R-2pc {  &  ~ \mathcal{D}(-, G)\ar[dd]  & \mathcal{D}(-, E)\ar[dd] & \mathcal{D}(-, A)\ar[dd] \\ \delta: \ \ 0 \ar[r] &  _{ \ \ \ \ } \ar[r]  &_{\ \ \ \ \ } \ar[r]&  _{\ \ \ \ \ }\ar[r] & 0. \\ & \mathcal{D}(-, K) & \mathcal{D}(-, D) &\mathcal{D}(-, B)}\]}
We ignore labeling the terms in the above short exact sequence. Using the Yoneda lemma gives us the following short exact sequence in $\rm{H}(\mmod \CA)$
	{\footnotesize  \[ \xymatrix@R-2pc {  &  ~ G\ar[dd]  & E\ar[dd] & A\ar[dd] \\ 0 \ar[r] &  _{ \ \ \ \ } \ar[r]  &_{\ \ \ \ \ } \ar[r]&  _{\ \ \ \ \ }\ar[r] & 0. \\ & K &  D &B}\]}
Considering this fact that  for any arbitrary object $(M\st{g}\rt N)$ with a non-zero map to $\mathbb{A}$, we may assume that $M$ and $N$ have indecomposable direct summands isomorphic to the $X_i,$ we can prove that the  almost split sequence  $\delta$ in  $\CP(\CF(\mathcal{D}))$ implies that  so is the above  induced short exact sequence in $\rm{H}(\mmod \CA)$. So we are done.
\end{proof}
Therefore, the above lemma allows us to apply our results in the preceding subsection.
We need also  the following lemma which says that the the functor $\Phi$ preserves the Nakayama functor.
\begin{lemma}\label{Lemma 5.11}
	For any $M \in \mmod \CA,$ there is a natural isomorphism $\Phi(D_{\mmod \CA}(\widehat{\CA}(M, -))\simeq D_{\mmod \CB}(\widehat{\CB}(F_{\la}(M), -))$. In particular, $\Phi(\CN_{\mmod \CA}(E))\simeq \CN_{\mmod \CB}(\Phi(E))$, for any $E \in \CF(\mmod \CA)$.
	\end{lemma}
\begin{proof}
		Let $M$ be in $\mmod \CA$ and $X\in \mmod \CB$. Since $F_{\la}$ is dense, there is $N \in \mmod \CA$ such that $F_{\la}(N)\simeq X.$ To prove  the lemma, consider the following sequence  of natural isomorphisms
\begin{align*}
\Phi(D_{\mmod \CA}(\widehat{\CA}(M, -))(X)&\simeq \Phi(D_{\mmod \CA}(\widehat{\CA}(M, -))(F_{\la}(N)) \\
&\simeq \reallywidehat{D_{\mmod \CA}(\widehat{\CA}(M, -))}\circ F_{\bullet}(F_{\la}(N)) \\
& \simeq \reallywidehat{D_{\mmod \CA}(\widehat{\CA}(M, -))}(F_{\bullet}(F_{\la}(N)))\\ & \simeq \reallywidehat{D_{\mmod \CA}(\widehat{\CA}(M, -))}(\oplus_{g \in G} {}^gN)\\ &\simeq \oplus_{g \in G} D(\widehat{\CA}(M, N))\\ & \simeq D(\widehat{\CA}(M , \oplus_{g \in G}{}^gN))\\ &\simeq D(\widehat{\CB}(F_{\la}(M), F_{\la}(N))) \\ &\simeq D_{\mmod \CB}(\widehat{\CB}(F_{\la}(M), -)) (F_{\la}(N)) \\ &\simeq D_{\mmod \CB}(\widehat{\CB}(F_{\la}(M), -)) (X)
\end{align*}
Since $G$ acts freely on the objects of $\CA$ and supports of $M$ and $N$ are finite, we get the commutativity of direct sums in the 6-th isomorphism. For the 7-th isomorphism we apply the fact that $F_{\la}$ is a $G$-precovering, see Theorem \ref{Theorem 2.2}. 	 
\end{proof}
\begin{proposition}\label{Proposition 5.16}
	Assume that $\eta:0 \rt M \st{f}\rt N\st{g}\rt K\rt 0$ is an almost split sequence in $\CF(\mmod \CA)$.  Then, the short exact sequence $\Phi(\eta): 0 \rt \Phi(M)\st{\Phi(f)}\lrt \Phi(N)\st{\Phi(g)}\rt \Phi(K)\rt 0$ is an almost split sequence in $\CF(\mmod \CB).$
\end{proposition}
\begin{proof}
	 Take a minimal projective presentation $\widehat{\CA}(-, X_1)\st{\widehat{\CA}(-, w)}\rt \widehat{\CA}(-, X_0)\rt K\rt 0$ of $K$ in $\CF(\mmod \CA).$   As mentioned in the beginning of this section, we have $0 \rt \tau_{\CF(\mmod \CA)}(K)\rt D_{\mmod \CA}(\widehat{\CA}(X_1, -))\st{D_{\mmod \CA}(\widehat{\CA}(w, -))}\rt D_{\mmod \CA}(\widehat{\CA}(X_0, -))$. By applying $\Phi$ on the minimal projective presentation and in view of Lemma \ref{Lemma 5.11}, we obtain $$\Phi(\tau_{\CF(\mmod \CA)}(K))\simeq \Ker(D_{\mmod \CB	}(\widehat{\CB}(F_{\la}(X_1), -)\st{D_{\mmod \CB}(\widehat{\CB}(F_{\la}(w), -))}\lrt D_{\mmod \CB}(\widehat{\CB}(F_{\la}(X_0), -)).$$ Since $F_{\la}$ preserves a minimal morphism, we deduce that $\Phi(M)\simeq \Phi(\tau_{\CF(\mmod \CA)}(K))\simeq \tau_{\CF(\mmod \CB)}(\Phi(K)).$ Now by following the same argument given in pages 90-91 of \cite{G}, we can prove $\Phi(\eta)$ is an almost split sequence (see also the end of proof of  Proposition \ref{Proposition 5.11}). 
\end{proof}
\begin{proposition}\label{Proposition 5.13}
	Assume that $\delta: \ \ 0 \rt \mathbb{D}\st{\varrho}\rt \mathbb{C}\st{\xi}\rt \mathbb{A}\rt 0$ is an almost split sequence in $\rm{H}(\mmod \CA)$ ending at a Mor-proper object in $\rm{H}(\mmod \CA)$. Then, the short exact sequence $\rm{H}F_{\la}(\delta): 0 \rt \rm{H}F_{\la}(\mathbb{D})\st{\rm{H}F_{\la}(\varrho)}\lrt \rm{H}F_{\la}(\mathbb{C})\st{\rm{H}F_{\la}(\xi)}\rt \rm{H}F_{\la}(\mathbb{A})\rt 0$ is an almost split sequence in $\rm{H}(\mmod \CB).$
\end{proposition}
\begin{proof}
	 We observe by the commutative diagram in Proposition \ref{Prop 4.2}, the isomorphism  $\Phi\circ \Theta_{\mmod \CA}(\delta)\simeq \Theta_{\mmod \CB}\circ \rm{H}F_{\la}(\delta)$ of the short exact sequences in $\CF(\mmod \CB)$. Here $\Phi\circ \Theta_{\mmod \CA}(\delta)$ and $ \Theta_{\mmod \CB}\circ \rm{H}F_{\la}(\delta)$ mean the short exact sequences in $\CF(\mmod \CB)$ obtained by applying component-wisely the functors $\Phi\circ \Theta_{\mmod \CA}$ or  $ \Theta_{\mmod \CB}\circ \rm{H}F_{\la}$ on $\delta$, respectively. We observe by Proposition \ref{Proposition 5.6}, $\Theta_{\mmod \CA}(\delta)$ is an almost split sequence  in $\CF(\mmod \CA)$, and also by Proposition \ref{Proposition 5.16}, $\Phi\circ \Theta_{\mmod \CA}(\delta)$ is an almost split sequence in $\CF(\mmod \CB)$. Hence by the aforementioned isomorphism $\Theta_{\mmod \CB}\circ \rm{H}F_{\la}(\delta)$  is an almost split sequence in $\CF(\mmod \CB)$ as well.
	According to Construction \ref{FirstCoonstr} the short exact sequence $\delta$ induces the following commutative diagram with the splitting rows  and exact columns in $\mmod \CA$
	 
	 $$\xymatrix{& 0 \ar[d] & 0 \ar[d] & 0 \ar[d]&&\\
	 	0 \ar[r] & \text{Ker}(w) \ar[d] \ar[r] &  \text{Ker} (v) \ar[d]
	 	\ar[r]^j & \text{Ker}(f) \ar[d] \ar[r] &0\\
	 	0 \ar[r] & D\ar[d]^w \ar[r] & C
	 	\ar[d]^v \ar[r] & A \ar[d]^f \ar[r] & 0\\
	 	0 \ar[r] & E  \ar[r] & L
	 	\ar[r] & B \ \ar[r] & 0 }
	 $$

Here we follow the same notation as Construction \ref{Construction 5.4}. By applying the push-down functor $F_{\la}$ on the above diagram  we get the following commutative diagram again with the splitting rows and exact columns in $\mmod \CB$

 $$\xymatrix{& 0 \ar[d] & 0 \ar[d] & 0 \ar[d]&&\\
	0 \ar[r] & F_{\la}(\text{Ker}(w)) \ar[d] \ar[r] &  F_{\la}(\text{Ker} (v)) \ar[d]
	\ar[r] & F_{\la}(\text{Ker}(f)) \ar[d] \ar[r] &0\\
	0 \ar[r] & F_{\la}(D)\ar[d]^{F_{\la(w)}} \ar[r] & F_{\la}(C)
	\ar[d]^{F_{\la}(v)} \ar[r] & F_{\la}(A) \ar[d]^{F_{\la}(f)} \ar[r] & 0\\
	0 \ar[r] & F_{\la}(E)  \ar[r] & F_{\la}(L)
	\ar[r] & F_{\la}(B) \ \ar[r] & 0 }$$
The lower last two rows are the extended version of $\rm{H}F_{\la}(\delta)$ in $\mmod \CB$. Next, we apply the Yoneda functor to obtain the following diagram in $\CF(\mmod \CB)$ with exact rows and columns
	$$\xymatrix{& 0 \ar[d] & 0 \ar[d] & 0 \ar[d]&\dagger&\\
	\mathbf{X}_1: \ \ \	0 \ar[r] & \widehat{\CB}(-, F_{\la}(\text{Ker}(w))) \ar[d] \ar[r] &  \widehat{\CB}(-, F_{\la}(\text{Ker} (v))) \ar[d]
		\ar[r] & \widehat{\CB}(-, F_{\la}(\text{Ker}(f))) \ar[d]\ar[r] &0\\
\mathbf{X}_2: \ \ \		0 \ar[r] &\widehat{\CB}(-, F_{\la}(D))\ar[d] \ar[r] & \widehat{\CB}(-, F_{\la}(C))
		\ar[d]\ar[r] & \widehat{\CB}(-, F_{\la}(A)) \ar[d] \ar[r] & 0\\
\mathbf{X}_3: \ \ \		0 \ar[r] &\widehat{\CB}(-, F_{\la}(E)) \ar[d] \ar[r] & \widehat{\CB}(-, F_{\la}(L))
		\ar[d] \ar[r] & \widehat{\CB}(-, F_{\la}(B)) \ar[d] \ar[r] & 0\\
	\mathbf{Y}: \ \ \	0\ar[r]	& \Theta_{\mmod \CB}(\text{H}\text{F}_{\la}(\mathbb{D})) \ar[r]\ar[d] & \Theta_{\mmod \CB}(\text{H}\text{F}_{\la}(\mathbb{C })) \ar[r]\ar[d]  & \Theta_{\mmod \CB}(\text{H}\text{F}_{\la}(\mathbb{A})) \ar[r]\ar[d] &0\\ & 0  & 0  & 0 &  }
	$$ 
By the observation given in the beginning we see that  the lower short exact sequence is    an almost split sequence in $\CF(\mmod \CB)$.	On the other hand, since $\rm{H}F_{\la}(\mathbb{A})$ is a Mor-proper object in $\rm{H}(\mmod \CB)$, by applying  this turn Construction \ref{Construction 5.4} for $\CC=\mmod \CB$, there exists the following almost split sequence in $\rm{H}(\mmod \CB)$
	 	$$\delta': \ \ 0 \rt \mathbb{D}'\st{\varrho'}\rt \mathbb{C}'\st{\xi'}\rt \rm{H}F_{\la}(\mathbb{A})\rt 0,$$
which induces the following commutative diagram  with the splitting rows and exact columns 	in $\mmod \CB$  $$\xymatrix{& 0 \ar[d] & 0 \ar[d] & 0 \ar[d]&&\\
		0 \ar[r] & \text{Ker}(w') \ar[d] \ar[r] &  \text{Ker} (v') \ar[d]
		\ar[r] & F_{\la}(\text{Ker}(F_{\la}(f))) \ar[d] \ar[r] &0\\
		0 \ar[r] & D'\ar[d]^{w'} \ar[r] & C'
		\ar[d]^{v'} \ar[r] & F_{\la}(A) \ar[d]^{F_{\la}(f)} \ar[r] & 0\\
		0 \ar[r] & E'  \ar[r] & L'
		\ar[r] & F_{\la}(B) \ \ar[r] & 0 }$$ 	
\end{proof}
By applying the Yoneda functor on the above diagram we get  the following commutative  diagram in $\CF(\mmod \CB)$ with exact rows and columns
$$\xymatrix{& 0 \ar[d] & 0 \ar[d] & 0 \ar[d]&\dagger\dagger&\\
\mathbf{X}'_1: \ \ \	0 \ar[r] & \widehat{\CB}(-, \text{Ker}(w')) \ar[d] \ar[r] &  \widehat{\CB}(-,\text{Ker}(v')) \ar[d]
	\ar[r] & \widehat{\CB}(-, F_{\la}(\text{Ker}(f))) \ar[d]\ar[r] &0\\
\mathbf{X}'_2: \ \ \	0 \ar[r] &\widehat{\CB}(-, D')\ar[d] \ar[r] & \widehat{\CB}(-, C')
	\ar[d]\ar[r] & \widehat{\CB}(-, F_{\la}(A)) \ar[d] \ar[r] & 0\\
\mathbf{X}'_3: \ \ \	0 \ar[r] &\widehat{\CB}(-, E') \ar[d] \ar[r] & \widehat{\CB}(-, L')
	\ar[d] \ar[r] & \widehat{\CB}(-, F_{\la}(B)) \ar[d] \ar[r] & 0\\
\mathbf{Y}': \ \ \	0\ar[r]	& \Theta_{\mmod \CB}(\mathbb{D}') \ar[r]\ar[d] & \Theta_{\mmod \CB}(\mathbb{C}') \ar[r]\ar[d]  & \Theta_{\mmod \CB}(\text{H}\text{F}_{\la}(\mathbb{A})) \ar[r]\ar[d] &0\\ & 0  & 0  & 0 &  }
$$ 
Proposition \ref{Proposition 5.6} implies that the lower row is an almost split sequence in $\CF(\mmod \CB)$. Since both sequences in the lower rows of the diagram $(\dagger)$ and $(\dagger\dagger)$ are almost split,  we deduce that  there are isomorphic. To complete the proof we make a trick as follows. Viewing the rows of $(\dagger)$ and $(\dagger\dagger)$ as complexes over $\CF(\mmod \CB)$ (or as representations of the quiver $\bullet\rt \bullet \rt \bullet$ over $\mmod \CB$) we obtain the following diagram in the category of complexes over  $\CF(\mmod \CB)$ 
	\[\xymatrix{0 \ar[r] & \mathbf{X}_1  \ar[r] & \mathbf{X}_2  \ar[r] & \mathbf{X}_3  \ar[r] & \mathbf{Y}\ar[d]^{\mathbf{a}} \ar[r] & 0 \\ 0  \ar[r] & \mathbf{X}'_1\ar[r] &\mathbf{X}'_2\ar[r] & \mathbf{X}'_3\ar[r] & \mathbf{Y}' \ar[r] & 0}\]

with $\mathbf{a}$ an isomorphism of complexes. Since all $\mathbf{X}_i$ and $\mathbf{X}'_i$, $1 \leqslant i \leqslant 3,$ are split complexes of the projective objects, hence they are projective objects in the category of complexes. In fact, the induced projective resolutions in the category of complexes are minimal because of the minimal projective resolutions in $\CF(\mmod \CB)$ lying in the ending columns of diagrams ($\dagger$) and ($\dagger\dagger$). So we can lift $\mathbf{a}$ to a chain map between the projective resolutions of $\mathbf{Y}$ and $\mathbf{Y}'$, as shown in below, 
	\[\xymatrix{0 \ar[r] & \mathbf{X}_1  \ar[r]\ar[d]^{\mathbf{b}_1} & \mathbf{X}_2  \ar[r]\ar[d]^{\mathbf{b}_2} & \mathbf{X}_3 \ar[d]^{\mathbf{b_3}} \ar[r] & \mathbf{Y}\ar[d]^{\mathbf{a}} \ar[r] & 0 \\ 0  \ar[r] & \mathbf{X}'_1\ar[r] &\mathbf{X}'_2\ar[r] & \mathbf{X}'_3\ar[r] & \mathbf{Y}' \ar[r] & 0.}\]
The piece  \[\xymatrix{ \mathbf{X}_2  \ar[r]\ar[d]^{\mathbf{b}_2} & \mathbf{X}_3 \ar[d]^{\mathbf{b}_3}  &  \\ \mathbf{X}'_2\ar[r] & \mathbf{X}'_3 & }\]
of the above diagram in conjunction with the Yoneda lemma give us the following commutative diagram with vertical isomorphisms

	$$\xymatrix@!0{
	0\ar[rr] &&F_{\la}(D)\ar[rr]\ar[dd]\ar[dr]^<<<<<{F_{\la}(w)}&&F_{\la}(C)\ar[dd]\ar[rr]\ar[dr]^<<<<<{F_{\la}(v)}&&F_{\la}(A)\ar[dd]\ar[rr]\ar[dr]^<<<<<{F_{\la}(f)}&&0\\
	& 0\ar[rr]&&F_{\la}(E)\ar[dd]\ar[rr]&&F_{\la}(L)\ar[dd]\ar[rr]&&F_{\la}(B)\ar[rr]\ar[dd]&&0\\
	0\ar[rr]&&D'\ar[rr]\ar[dr]&&C'\ar[rr]\ar[dr]&&F_{\la}(A)\ar[rr]\ar[dr]^<<<{F_{\la}(f)}&&0\\
	&0\ar[rr]&&E'\ar[rr]&&L'\ar[rr]&&F_{\la}(B)\ar[rr]&&0
}$$
Because of the minimality of the corresponding projective presentations and being isomorphism of $\mathbf{a}$ we get the all vertical maps in the above diagram are isomorphisms. Hence the above diagram implies that $\rm{H}F_{\la}(\delta)\simeq \delta'$. Hence $\rm{H}F_{\la}(\delta)$ is an almost split sequences, as desired. 
\subsection{The case of $\CS \rm{F}_{\la}$}
For the case of monomorphism category we follow the same strategy as the morphism category given in the preceding subsection. We divide our argument into two steps that the ending term of the given almost split in $\CS(\mmod \CA)$ is either  a Mono-proper or non-Mono-proper object in $\CS(\mmod \CA)$. Our proof for the non-Mono-proper case is based on Lemma \ref{AlmostSplittrivialmonomorphisms}, and for the Mono-proper case we need Construction \ref{ConstructionMonoSplit}. Since  the results in this subsection are proved the same  as their  counterpart  of the previous subsection,   we mainly skip the proofs in this subsection. We just shall provide some more explanations to help for adopting the proofs for the monomorphism case. 

\begin{proposition}\label{Proposition 5.14}
	Assume that $\eta:0 \rt \mathbb{A}\st{p}\rt \mathbb{B}\st{q}\rt \mathbb{C}\rt 0$ be an almost split sequence in $\CS(\mmod \CA)$ ending at a non-Mono-proper object in $\CS(\mmod \CA)$. Then, the short exact sequence $\CS \text{F}_{\la}(\eta): 0 \rt \CS \text{F}_{\la}(\mathbb{A})\st{\CS \text{F}_{\la}(p)}\lrt \CS \text{F}_{\la}(\mathbb{B})\st{\CS \text{F}_{\la}(q)}\rt \CS \text{F}_{\la}(\mathbb{C})\rt 0$ is an almost split sequence in $\CS (\mmod \CB).$
\end{proposition}
\begin{proof}
Analogous to Proposition \ref{Proposition 5.11}, we prove case by case  for when   that $\mathbb{C}$ is isomorphic to either proper objects $(0 \rt C)$, $(C\st{1}\rt C)$ and  $(\Omega_{\mmod \CA}(C)\hookrightarrow P)$,  for some indecomposable object $C$ in $\mmod \CA$. The case $\mathbb{C}\simeq (0\rt C)$: $\eta$ is isomorphic to the one given in Lemma \ref{AlmostSplittrivialmonomorphisms} (1) for when $\CC=\mmod \CA.$ We know from Proposition \ref{Proposition 5.11} that $\CS \rm{F}_{\la}(\eta)=\rm{H}F_{\la}(\eta)$ is an almost split in $\rm{H}(\mmod \CB)$. Hence it is an almost split in $\CS(\mmod \CA)$ too. So we are done this case. For two other cases, since the push-down functor $F_{\la}$ preserves the almost split sequences, projective and injective objects, and minimal morphisms, we observe that the structures of the almost split sequences ending at two other non-Mono-proper cases, given in Lemma \ref{AlmostSplittrivialmonomorphisms} for $\CC=\mmod \CA, \mmod \CB$, are preserved by the functor $\CS \rm{F}_{\la}$. We are done.
\end{proof}

For the  non-Mono-proper case, we need to say  an analog of Lemma \ref{Lemma 5.11}, i.e., the restricted functor $\Phi\mid:\CF(\underline{\rm{mod}}\mbox{-}\CA)\rt  \CF(\underline{\rm{mod}}\mbox{-}\CB)$, see the lines preceding Construction \ref{FirstCoonstr}, commutes with the Nakayama functor. Note that by Remark \ref{Rem5.11} as $\underline{\rm{mod}}\mbox{-}\CB$ is   a $k$-dualizing variety, so we can talk about the Nakayama functor for it.
\begin{lemma}\label{lemma 5.20}
	For any $M \in \mmod \CA,$ there is a natural isomorphism $\Phi\mid (D_{\underline{\rm{mod}}\mbox{-}\CA}(\underline{\widehat{\CA}}(M, -))\simeq D_{\underline{\rm{mod}}\mbox{-}\CB}(\underline{\widehat{\CB}}(F_{\la}(M), -))$. In particular, $\Phi\mid (\CN_{\underline{\rm{mod}}\mbox{-}\CA}(E))\simeq \CN_{\underline{\rm{mod}}\mbox{-}\CB}(\Phi\mid (E))$, for any $E \in \CF(\underline{\rm{mod}}\mbox{-}\CA)$.
\end{lemma}
\begin{proof}
	The same argument as Lemma \ref{Lemma 5.11}
 works here. In this case we also  need to use this fact that the functor $\underline{F_{\la}}:\underline{\rm{mod}}\mbox{-}\CA\rt \underline{\rm{mod}}\mbox{-}\CB$ is a $G$-precovering, see Proposition \ref{Prop 2.6}.
\end{proof}

  Next in view of the above lemma and Construction \ref{ConstructionMonoSplit}  we can follow  the same proof of Proposition \ref{Proposition 5.13} to prove the following:
  \begin{proposition}\label{Prop 5.16}
  	Assume that $\eta:0 \rt \mathbb{A}\st{f}\rt \mathbb{B}\st{g}\rt \mathbb{C}\rt 0$ is an almost split sequence in $\CS(\mmod \CA)$ ending at a Mono-proper object in $\CS(\mmod \CA)$. Then, the short exact sequence $\CS \text{F}_{\la}(\eta): 0 \rt \CS \text{F}_{\la}(\mathbb{A})\st{\CS \text{F}_{\la}(f)}\lrt \CS \text{F}_{\la}(\mathbb{B})\st{\CS \text{F}_{\la}(g)}\rt \CS \text{F}_{\la}(\mathbb{C})\rt 0$ is an almost split sequence in $\CS (\mmod \CB).$
  \end{proposition}
  \begin{proof}
  A brief sketch of the reasoning is given below. Thanks to Proposition \ref{Proposition 5.10} we deduce that $\Psi_{\mmod \CA}(\eta)$ is an almost split sequence in $\CF(\underline{\rm{mod}}\mbox{-}\CA).$ Following the same lines in the Proof of Proposition \ref{Proposition 5.16} together with applying  Lemma \ref{lemma 5.20} we infer that $\Phi \mid(\Psi_{\mmod \CA}(\eta) ) $ is an almost split sequence  in $\CF(\underline{\rm{mod}}\mbox{-}\CB)$.  We get the diagram $(\dagger')$ with  the lower row $\mathbf{Z}':=\Psi_{\mmod \CB}(\CS \text{F}_{\la}(\eta))=\Phi \mid(\Psi_{\mmod \CA}(\eta) )$ the same as the diagram $(\dagger)$ in the proof of Proposition \ref{Proposition 5.13}. Accordingly, using Construction \ref{ConstructionMonoSplit}, we have the diagram $(\dagger'\dagger')$ with   the lower row $\mathbf{Z}'$ the same as the diagram $(\dagger\dagger)$ in the proof of Proposition \ref{Proposition 5.13}. Viewing the rows $\mathbf{Z}$ and $\mathbf{Z}'$ as complexes over $\CF(\mmod \CB)$, and considering the projective resolutions of $\mathbf{Z}$ and $\mathbf{Z}'$ (in the category of complexes over $\CF(\mmod \CB)$) inducing by the diagrams $(\dagger')$ and $(\dagger'\dagger')$ respectively, we can deduce that $\CS \text{F}_{\la}(\eta)$ is an almost split in $\CS (\mmod \CB)$.
  \end{proof}
  
  Combining the all above results (Propositions \ref{Proposition 5.11}, \ref{Proposition 5.13}, \ref{Proposition 5.14} and \ref{Prop 5.16}) leads to  our main result in this section.
\begin{theorem}\label{Theorem 5.22}
	Assume that $\CA$ is a locally bounded $k$-catgery  and $G$
	a torsion-free  group of $k$-linear automorphisms of $\CA$ acting freely on the objects of $\CA,$ and  $F:\CA\rt \CB$ is a Galois functor as in Definition \ref{Def 2.1}. Further, the action of $G$ on $\CA$ have only finitely many $G$-orbits, and $\mmod \CB$ is of finite representation type. 
	\begin{itemize}
		\item [$(i)$] The functor $\rm{H}F_{\la}$ preserves the almost split sequences in $\rm{H}(\mmod \CA)$. 
		\item [$(ii)$] The functor $\CS \rm{F}_{\la}$ preserves the almost split sequences in $\CS(\mmod \CA)$. 
	\end{itemize}
\end{theorem}

\section {Gabriel's theorem for (Mono)morphism category}\label{Section 6}
 The notion of the Auslander-Reiten quivers  is  mainly based on the one of  almost split sequences.  As a direct consequence of our result in the preceding section, we show that the functor $\rm{H}\rm{F}_{\la}$, resp. $\CS \rm{F}_{\la}$, induces a Galois covering between some components of the Auslander-Reiten quivers of $\rm{H}(\mmod \CA)$ and $\rm{H}(\mmod \CB)$, resp. $\CS(\mmod \CA)$ and $\CS(\mmod \CB)$. One of  the important result in the covering theory, proved by Gabriel\cite{G}, is to make a connection between the representation type of $\CA$ and $\CB$, namely in our setup, asserting   $\CA$ is locally representation finite if and only if $\CB$ is representation-finite.  Our  approach in this  paper  enables us to establish a  similar connection between the involved morphism, resp. monomorphism, categories.

Following \cite[Section 4]{BL}, we give a brief recall of the notion of Auslander-Reiten quivers.  Let $\CQ$ be a finite quiver $\CQ=(\CQ_0, \CQ_1, s, t)$, where $\CQ_0$ and $\CQ_1$ are the set of arrows and vertices of $\CQ$, respectively, and $s$ and $t$ are the starting and ending maps from $\CQ_1$ to $\CQ_0,$ respectively. For $x \in \CQ_0$, denote by $x^+$ the set of arrows $\alpha$ with $s(\alpha)=x$, and by $x^-$ the set of arrows $\beta$ with $t(\beta)=x$. A {\it valued quiver} is a pair $(\CQ, v)$, where $\CQ$ is a quiver without multiple arrows and $v$ is a valuation on arrows, that is, each arrow $x\rt y$ is endowed with a pair $(v_{xy}, v'_{xy})$ of positive integers. A {\it valued quiver morphism} $\phi:(\CQ, v)\rt (\CQ', v')$ is a quiver morphism, consisting of two maps $\phi_0:\CQ_0\rt \CQ'_0$ and $\phi_1:\CQ_1\rt \CQ'_1$ such that $\phi_1(\CQ_1(x, y))\subseteq \CQ'_1(\phi_0(x), \phi_0(y))$ for all $x, y \in \CQ_0$, such that $v_{xy}\leqslant v'_{\phi_0(x)\phi_0(y)}$, for any $x\rt y\in \CQ_1.$ The valued quiver $(\CQ, v)$ is called {\it symmetrically valued} if $v_{xy}=v'_{xy}$, for any $x, y \in \CQ_0$. Finally, a {\it valued translation quiver} is a triple $(\CQ, v, \tau)$ is a valued quiver and $\tau$  a translation, that is a bijection from one subset of $\CQ_0$ to another one such that, for any $x \in \CQ_0$ with $\tau x$ defined, we have  $x^+=(\tau x)^-\neq\emptyset$ and $(v_{\tau x, y}, v'_{\tau x, y})=(v'_{y, x}, v_{y, x})$ for every $y \in x^+$. In this case, $x \in \CQ_0$ is called {\it projective} or {\it injective} if $\tau x$ or $\tau ^{-1} x $ is not defined, respectively. A morphism of valued translation quivers $\phi:(\CQ, v, \tau)\rt (\CQ', u, \tau')$ is a valued quiver morphism $\phi:(\CQ, v)\rt (\CQ', u)$ satisfying the condition: for any non-projective $x \in \CQ_0$, the image $\phi_0(x)$ is not projective with $\tau' \phi_0(x)=\phi_0(\tau x).$ Suppose that $H$ is a group, as the $k$-categories setting, we say that the group $H$ acts on $\CQ$, if  there exists a homomorphism from $H$ into the group of automorphisms of $\CQ$.
\begin{definition}
	Let $(\CQ, v, \tau)$ be a valued translation quiver with a  free action of a group $H$, in the usual sense. A  morphism of valued translation quivers $\phi:(\CQ, v, \tau) \rt (\CQ', u, \tau')$,  is called a Galois $H$-covering provided that the following conditions are satisfied.
	\begin{itemize}
		\item [$(1)$] The map $\phi_0$ is surjective.
		\item [$(2)$] If $g \in G$, then $\phi_0\circ g=\phi_0.$
		\item [$(3)$] If $x, y \in \CQ_0$ with $\phi_0(x)=\phi_0(y)$, then $y=gx$ for some $g \in G.$
		\item [$(4)$] If $x \in \CQ_0$, then $\phi_1$ induces two bijections $x^+\rt \phi_0(x)^+$ and $x^-\rt \phi_0(x)^-.$
		\item [$(5)$] If $x \in \CQ_0$ with $a \in \phi_0(x)^+$ and $b \in \phi_0(x)^-$, then 
		$$u_{\phi_0(x), a}=\Sigma_{y \in x^+\cap\phi^-(a)}v_{x,y} \  \text{and} \ u_{b, \phi_0(x)}=\Sigma_{z \in x^-\cap\phi^-(b)}v_{z, x}$$ 
		\item [$(6)$] For any projective vertex $x \in \CQ_0$, the image $\phi_0(x)$ is projective, or equivalently its dual holds.
	\end{itemize}
	\end{definition}
Let $\mathcal{D}$ be an extension-closed subcategory of a  Krull-Schmidt  abelian $k$-category $\CC$. For $A, B \in \mathcal{D}$, let $\rm{rad}_{\mathcal{D}}(A, B)$ denote the {\it radical } of $\mathcal{D}(A, B)$ defined as 
$$\{f \in \mathcal{D}(A, B)\mid hfg \ \text{is not isomorphism for any} \ g:X\rt A\ \text{and} \ h:B\rt X\ \text{with} \ X \ \text{in } \ \rm{ind}\mbox{-}\mathcal{D} \}.$$ Further, $\text{rad}^2_{\mathcal{D}}(A, B)=\{f \in \mathcal{D}(A, B)\mid f=g\circ h \ \text{for some } h \in \rm{rad}(A, X), g \in \rm{rad}(X, B) \ \text{and} \ X \in \rm{ob}(\mathcal{D}) \}$. Define $\rm{irr}_{\mathcal{D}}(A, B):=\rm{rad}_{\mathcal{D}}(A, B)/\rm{rad}^2_{\mathcal{D}}(A, B)$. Then the {\it Auslander-Reiten quiver } $\Gamma_{\mathcal{D}}$ is a valued symmetrically translation quiver defined as follows, see e.g. \cite{L}: the vertex set is $\rm{Ind}\mbox{-}\mathcal{D}$, a complete set of non-isomorphic indecomposable objects in $\mathcal{D}$, for any vertices $X, Y$, there exists a single arrow $X\rt Y$ with valuation $(d_{XY}, d'_{XY})$ where $d_{XY}=d'_{XY}=\rm{dim}_{k}\rm{irr}_{\mathcal{D}}(X, Y)>0,$ and the translation $\tau_{\mathcal{D}}$ is the Auslander-Reiten translation, that is, $X=\tau_{\mathcal{D}}Z$ if and  only if $\mathcal{D}$ has an almost split sequence $0\rt X\rt Y \rt Z\rt 0.$ For short, we sometimes use $\Gamma_{\La}$ instead of $\Gamma_{\mmod \La}$. \\
Suppose now that $H$ is a group acting freely on $\mathcal{D}$. Let $\Sigma$ be a complete set of representatives of the $H$-orbits in $\rm{ob}(\mathcal{D})$. Then, we can choose $\rm{Ind}\mbox{-}\mathcal{D}$ to be the set of objects ${}^gX$ with $X \in \Sigma$ and $g \in H$. In this way, $\rm{Ind}\mbox{-}\mathcal{D}$ becomes $H$-stable, that is, if $M \in \rm{Ind}\mbox{-}\mathcal{D}$, then ${}^gM \in \rm{Ind}\mbox{-}\mathcal{D}$, for every $g \in H.$ It is easy to see that the free $H$-action on $\mathcal{D}$ induces a free $H$-action on the valued translation quiver $\Gamma_{\mathcal{D}}.$\\

Now, by means of our main result (Theorem \ref{Theorem 5.22}) in the  preceding section and following the same  proof of \cite[Theorem 4.7]{BL} imply the following theorem.
\begin{theorem}\label{Theorem 6.2}
	Assume that $\CA$ is a locally bounded $k$-catgery  and $G$
	a torsion free  group of $k$-linear automorphisms of $\CA$ acting freely on the objects of $\CA,$ and  $F:\CA\rt \CB$ is a Galois functor as in Definition \ref{Def 2.1}. Further, the action of $G$ on $\CA$ have only finitely many $G$-orbits. Further, assume $\mmod \CB$ is of finite representation type and  connected.
	\begin{itemize}
		\item [$(i)$] The functor $\rm{H}F_{\la}$ induces a quiver morphism $\overline{\rm{H} F}_{\la}$ from  $\Gamma_{\rm{H}(\mmod \CA)}$ on to the  union of some connected components of  $\Gamma_{\rm{H}(\mmod \CB)}.$ The image $\overline{\rm{H} F}_{\la}(\Gamma_{\rm{H}(\mmod \CA)})$ is a symmetrically valued  translation subquiver of $\Gamma_{\rm{H}(\mmod \CB)}$, and moreover the induced quiver morphism $\overline{\rm{H} F}_{\la}: \Gamma_{\rm{H}(\mmod \CA)}\rt \overline{\rm{H} F}_{\la}(\Gamma_{\rm{H}(\mmod \CA}))$ is a Galois $G$-covering of symmetrically valued  translation quivers.
		\item [$(ii)$] A monomorphism version of $(i)$ holds. More precisely, the functor $\CS \rm{F}_{\la}$ induces a quiver morphism $\overline{\CS \rm{F}}_{\la}$ from  $\Gamma_{\CS(\mmod \CA)}$ on to the  union of some connected components of  $\Gamma_{\CS(\mmod \CB)}.$ The image $\overline{\CS \rm{F}}_{\la}(\Gamma_{\CS(\mmod \CA)})$ is a symmetrically valued translation subquiver of $\Gamma_{\CS(\mmod \CB)}$, and moreover the induced quiver morphism $\overline{\CS \rm{F}}_{\la}: \Gamma_{\CS(\mmod \CA)}\rt \overline{\CS \rm{F}}_{\la}(\Gamma_{\CS(\mmod \CA)})$ is a Galois $G$-covering of symmetrically valued translation quivers.
	\end{itemize}
\end{theorem}
\begin{remark}\label{remark 6.3}
	Assume $\CC$ is a Krull-Shmidt abelian  $k$-category. The functor $\Psi_{\CC}$, introduced in Construction \ref{FirstCoonstr}, since it is full, faithful and objective,  induces an equivalence $\CS(\CC)/\CV\simeq \CF(\underline{\CC})$, where $\CV$ is generated by the objects of the form $(0\rt X)$ and $(X\st{1}\rt X)$, where $X$ runs through $\CC,$ see \cite[Theorem 3.2]{H} for more details. In the case that $\CC$ is locally bounded, by the equivalence it is easy to see that: The monomorphism category $\CS(\CC)$ is locally bounded if and only if so is $\CF(\underline{\CC})$. In particular, if we assume $\CC$ is of finite representation type, then we get:  $\CS(\CC)$ is of finite representation type if and only if so is $\CF(\underline{\CC})$. Similarly, by use of Construction \ref{Construction 4.1} the similar relationship between $\rm{H}(\CC)$ and $\CF(\CC)$ holds.
\end{remark}

A category $C$ is called {\it connected}  if and only if there is no non-trivial partition $\rm{ob}(\CC)=V\cup U$  such that $\CC(X, Y)=0=\CC(Y, X)$ for any $X \in V$ and $Y \in Y$. We recall that an  algebra $\La$
is said to be {\it indecomposable} if and only if $\rm{prj}\mbox{-}\La$, the subcategory of finitely generated projective modules,  is a connected category \cite{AuslanreitenSmalo}. \\
We need the following  criterion for finite representation type of (Mono)morphism categories  in terms of Auslander-Reiten quivers.

\begin{lemma}\label{Lemma 6.4}
Assume $\La$ is an indecomposable  algebra of finite representation type. Then, the following hold.
\begin{itemize}
	\item [$(1)$] The morphism category $\rm{H}(\mmod \La)$ is of finite representation type if and only if the Auslander-Reiten quiver $\Gamma_{\rm{H}(\mmod \La)}$ has a finite connected component. 
		\item [$(2)$] The monomorphism category $\CS(\mmod \La)$ is of finite representation type if and only if the Auslander-Reiten quiver $\Gamma_{\CS(\mmod \La)}$ has a finite  component.	
\end{itemize} 
\end{lemma}
\begin{proof}
The lemma follows from \cite[Lemma 5.1]{L}. We just point out that since $\La$ is an indecomposable algebra, the categories $\rm{H}(\mmod \La)$ and $\CS(\mmod \La)$ are connected. To prove the connectness, one may use the classification of projective objects in $\rm{H}(\mmod \La)$ (e.g. \cite[Chapter III, Proposition 2.5]{AuslanreitenSmalo}) and $\CS(\mmod \La)$ (\cite[Proposition 1.4]{RS2}).
\end{proof}
Now we are ready to state our second main theorem in this section. The proof of the following theorem is inspired by  the proof of \cite[Lemma 3.3 and Theorem 3.6]{G}.
\begin{theorem}\label{Theorem 6.5}
	Assume that $\CA$ is a locally bounded $k$-catgery  and $G$
	a torsion-free group of $k$-linear automorphisms of $\CA$ acting freely on the objects of $\CA,$   $F:\CA\rt \CB$ is a Galois covering functor as in Definition \ref{Def 2.1}, and the action of $G$ on $\CA$ have only finitely many $G$-orbits. Further, assume $\mmod \CB$ is of finite representation type and  connected.
	\begin{itemize}
		\item [$(i)$] The following conditions are equivalent.
		\begin{itemize}		
			\item [$(1)$] $\rm{H}(\mmod \CA)$ is locally bounded.
			\item [$(2)$] $\CF(\mmod \CA)$ is locally bounded.
			\item [$(3)$] $\rm{H}(\mmod \CB)$ is of finite representation type.
			\item [$(4)$] $\CF(\mmod \CB)$ is of finite representation type.
		\end{itemize}
		\item [$(ii)$] The following conditions are equivalent.
		\begin{itemize}
						\item [$(1)$] $\CS(\mmod \CA)$ is locally bounded.
			\item [$(2)$] $\CF(\rm{\underline{mod}}\mbox{-} \CA)$ is locally bounded.
			\item [$(3)$] $\CS(\mmod \CB)$ is of finite representation type.
			\item [$(4)$] $\CF(\rm{\underline{mod}}\mbox{-} \CB)$ is of finite representation type.
		\end{itemize}
	\end{itemize}
\end{theorem}
\begin{proof}
 We give a proof  for the case $(i)$. The equivalences $(1)\Leftrightarrow (2)$
  and $(3)\Leftrightarrow(4)$ follow from Remark \ref{Remark 7.6}. Assume $(1)$ holds. Then by Theorem \ref{Theorem 6.2} there is a finite connected component in $\Gamma_{\rm{H}(\mmod \CB)}$. Lemma \ref{Lemma 6.4} implies $(3)$. Suppose $(3)$ is true. Lemma \ref{Lemma 6.4} follows that the quiver morphism $\overline{\rm{H} F}_{\la}:\Gamma_{\rm{H}(\mmod \CA)}\rt \Gamma_{\rm{H}(\mmod \CB)} $ is surjective. This fact implies that the functor $\rm{H}F_{\la}$ is dense.  Let $\{\mathbb{X}_1,\cdots,\mathbb{X}_n\}$ be representatives of the indecomposable objects in $\rm{H}(\mmod \CB)$. Since $\rm{H}F_{\la}$ is dense, for each $1\leqslant i \leqslant n$, there exists an indecomposable $\mathbb{Y}_i$ in $\rm{H}(\mmod \CA)$ such that $\rm{H}F_{\la}(\mathbb{X}_i)=\mathbb{Y}_i$. 
  We claim that each indecomposable object $\mathbb{Y}$ in $\rm{H}(\mmod \CA)$ is isomorphic to ${}^gY_i$ for some $g \in G$ and $1\leqslant i \leqslant n.$ By the krull-Schmidt property of $\rm{H}(\mmod \CB)$ we have the isomorphism $\rm{H}F_{\la}(\mathbb{Y})\simeq \mathbb{X}^{m_1}_1\oplus \cdots \mathbb{X}^{m_n}_n.$ By applying $\rm{H}F_{\bullet}$ on the isomorphism, Proposition \ref{Prop 3.2} yields 
  $$\oplus_{g \in G} {}^g\mathbb{Y}\simeq \text{H}\text{F}_{\bullet}\circ \text{H}\text{F}_{\la}(\mathbb{Y})\simeq \text{H}\text{F}_{\bullet}(\mathbb{X}_1)^{m_1}\oplus \cdots \oplus \text{H}\text{F}_{\bullet}(\mathbb{X}_n)^{m_n}\simeq (\oplus_{g \in G} {}^g\mathbb{Y}_1)^{m_1}\oplus \cdots \oplus (\oplus_{g \in G} {}^g\mathbb{Y}_n)^{m_n} $$
  
So in view of \cite[Theorem 1]{W}, $\mathbb{Y}$ is isomorphic to ${}^g\mathbb{Y}_i$ for some $g \in G$ and $1\leqslant i \leqslant n$. For any $1 \leqslant i \leqslant n$, set $\mathbb{Y}_i=(A_i\st{f_i}\rt B_i)$. Since $\rm{Supp}(A_i\oplus B_i)$ is finite, we infer there are finitely many indecomposable objects of the form ${}^gY_j$ such that $\text{H}(\mmod \CA)({}^gY_i, {}^hY_j)\neq 0$, so desired result. 
\end{proof}
Suppose $\CF(\mmod \CA)$ is locally bounded. In this case $\rm{H}(\mmod \CB)$ has a single component, by making use of Theorem \ref{Theorem 6.2} we get the functor $\rm{H}F_{\la}$ is dense. According to the diagram given in Proposition \ref{Prop 4.2}, we observe that the functor $\Phi$ is dense. See  \cite[Remark 5.6]{P} for a discussion about density of the functor $\Phi.$ We also add that Proposition \ref{Prop 4.4} implies that $\Phi$ is dense if and only if so is $\rm{H}F_{\la}$.
\section{Self-injective algebras with the stable Auslander algebras of finite representation type}\label{Section 7}
Throughout this section assume $k$ is an algebraically closed field of
characteristic different from 2. As an application of our results on covering theory over monomorphism categories, we  determine for which types of representation-finite selfinjective having  the monomorphism category of finite representation type (or equivalently having  the stable Auslander algebra of finite representation type).\\ 
We recall the definition of the repetition of a category (cf. \cite{BG, Ha}).\\
Let $\CC$ be a $k$-category. The {\it repetitive category} $\widehat{\CC}$ of $\CC$ is  a $k$-category  defined as follows. The class of objects is $\{ (X,i) \mid X \in A, i\in Z \}$ and the morphism space is given by 
\begin{align*}
\widehat{\CC} \bigl((X,i), (Y,j)\bigr)=\begin{cases}
\CC(X,Y) & i=j, \\
D(\CC(Y,X)) & j=i+1, \\
0 & \textnormal{else}.
\end{cases}
\end{align*}
For $f\in \widehat{\CC} \bigl((X,i), (Y,j)\bigr)$ and $g\in \widehat{\CC} \bigl((Y,j), (Z,k)\bigr)$, the composition is given by 
\begin{align*}
g\circ f=\begin{cases}
g\circ f & i=j=k, \\
D\bigl(\CC(Z,f)\bigr)(g) & i=j=k-1, \\
D\bigl(\CC(g,X)\bigr)(f) & i+1=j=k, \\
0 & \textnormal{else}.
\end{cases}
\end{align*}
Note that if $\CC$ is locally bonded, then so is $\widehat{\CC}$.\\
An auto-equivalence $\phi$ of a Krull-Schmidt category $\CC$ is said to be {\it admissible} if the cycle group $G=\{\phi^i\mid i \in \mathbb{Z}\}$ acts
freely on isomorphism classes of indecomposable objects in $\CC$. For ease of notation the orbit category $\CC/G$ is denoted by $\CC/\phi.$

Let $\La$ be a hereditary  finite dimensional $k$-algebra and let $M$ be a tilting $A$-module. Then $B=\rm{End}_{\La}(M)$ is called a tilted algebra. If $\La\simeq k\Delta$, where $k\Delta$ is the path algebra of $\Delta$ over $k$, we will say that $B$ is a tilted algebra of type $\Delta$. In particular, if the underlying graph of  $\Delta$ is a Dynking diagram, $B$ is then called a tilted algebra of Dynkin type, that is representation-finite \cite[Chapter III, Section 5]{Ha}.\\
There is a well-known classification of representation-finite selfinjective algebras, due chiefly to Riedtmann \cite{BLR, Ri1, Ri2} (see also \cite [Theorm 3.5]{skowronski06} and \cite[Theorem 8.1]{SY}). 
The classification says: let  $\La$ be  a tilted algebra of Dynkin type, $\phi$ is an admissible auto-equivalence of the repetitive category $\widehat{\La}$. Then  the orbit category $\widehat{\La}/\phi$, viewed as a finite dimensional algebra (see \cite[Lemma 3.4]{DI}), is a representation-finite selfinjctive algebra. Conversely,   every representation-finite selfinjctive algebra over an algebraically closed field of characteristic different from 2 (as our convention in this section) is obtained  in this way.\\
Another  important known result about  representation-finite selfinjective algebras is the description of the stable Auslander-Reiten quiver $\Gamma^s_{\La}$, which is obtained by removing projective vertices of $\Gamma_{\La}$,   established by C. Riedtmann \cite{Ri3} and G. Todorov \cite{T}: $\Gamma^s_{\La}$  is isomorphic to the orbit quiver $\mathbb{Z}\Delta/G$,  where $\Delta$ is a Dynkin quiver (the underlying diagram of  $\Delta$ is a Dynkin diagram), and $G$ is an admissible infinite cyclic group of automorphisms of the translation quiver $\mathbb{Z}\Delta$. Therefore, we may associate to any selfinjective algebra $\La$ of finite representation type a Dynkin graph  $\Delta(\La)$, called the Dynkin type of $\La$,  such that $\Gamma^s_{\La}=\mathbb{Z}\Delta/G$ for a quiver $\Delta$ having $\Delta(\La)$ as underlying diagram. We also mention that, for a Dynkin quiver $\Delta$ and a tilted algebra B of type $\Delta$, the orbit algebras $ \widehat{B}/G $ are selfinjective algebras of finite representation type whose Dynkin type is the underlying graph of $\Delta$.

We need the following lemma of \cite{IO} (see also \cite{Ki} for an extension).
\begin{lemma}\label{Lemma 7.1}
	Assume that  $\La$ is a representation finite hereditary algebra. Then, there exists a triangle equivalence
	$$\underline{\CF(\mathbb{D}^{\rm{b}}(\mmod\La) )}\simeq \mathbb{D}^{\rm{b}}(\mmod \Gamma),$$
	 where $\Gamma$ is the stable Auslander algebra of $\La$, and recall that $\underline{\CF(\mathbb{D}^{\rm{b}}(\mmod\La) )}$ is the stable category of the  the Frobenius category $\CF(\mathbb{D}^{\rm{b}}(\mmod\La)) $of finitely presented functors from $\mathbb{D}^{\rm{b}}(\mmod \La)$ to the category of abelian groups $\CA b$.
	\end{lemma}
Note that by \cite[Theorem 3.1]{F} we know that $\CF(\CC)$ is a  Frobenius category whenever $\CC$ is a triangulated category.

\begin{proposition}\label{proposition 7.2}
Let  $\La$ be  a tilted algebra of Dynkin type $\Delta$, $\phi$ is an admissible auto-equivalence of the repetitive category $\widehat{\La}$. Suppose $A$, resp. $\Gamma,$ is the stable Auslander algebra of the path algebra $k\Delta$, resp.  representation-finite selfinjective algebra $\widehat{\La}/\phi$. Then the following conditions are equivalent.
\begin{itemize}
	\item [$(1)$] The monomorphism category $\CS(\mmod \widehat{\La})$ is locally bounded.
	\item [$(2)$] The monomorphism category $\CS(\mmod \widehat{\La}/\phi)$ is of finite representation type.
	\item [$(3)$] The  algebra $\Gamma$ is representation-finite.	
		\item [$(4)$] The bounded derived category $\mathbb{D}^{b}(\mmod A)$ is locally bounded.	
\end{itemize}	
\end{proposition}
\begin{proof}
	The Galois functor $\widehat{\La}\rt \widehat{\La}/\phi$, with $G=<\phi>\simeq \mathbb{Z}$, satisfies  the required assumption in Theorem \ref{Theorem 6.5} by use of \cite{G} or \cite{DLS}.  Now, the equivalences $(1)\Leftrightarrow (2)\Leftrightarrow(3)$ follow from the theorem. Consider the following three assertions.
	$$ \CF(\rm{\underline{mod}}\mbox{-} \widehat{\La}) \ \text{is locally bounded} \ \Leftrightarrow \ \text{ so is}  \ \CF(\rm{\underline{mod}}\mbox{-} \widehat{k\Delta}). \ \ \  \ \ \ \ \dagger$$
	Proof of the claim $(\dagger)$: since $\La$ is a tilted algebra of type $\Delta$, we get $\La$ is derived equivalent to the path algebra $k\Delta$, i.e., $\mathbb{D}^{\rm{b}}(\mmod \La)\simeq \mathbb{D}^{\rm{b}}(\mmod k\Delta).$ Hence by \cite[Theorem 1.5]{As} we obtain their repetitive categories are derived equivalent, i.e., $\mathbb{D}^{\rm{b}}(\mmod \widehat{\La})\simeq \mathbb{D}^{\rm{b}}(\mmod \widehat{k\Delta})$. The latter derived equivalence implies an equivalence between the stable categories $\rm{\underline{mod}}\mbox{-} \widehat{\La}\simeq \rm{\underline{mod}}\mbox{-} \widehat{k\Delta}$. In fact, the derived equivalence yields a triangle equivalence between their singularity categories, i.e.,
	$$\mathbb{D}_{\rm{sg}}(\widehat{\La})=\frac{\mathbb{D}^{\rm{b}}(\mmod \widehat{\La})}{\mathbb{D}^{\rm{per}}(\widehat{\La})}\simeq \frac{\mathbb{D}^{\rm{b}}(\mmod \widehat{k\Delta})}{\mathbb{D}^{\rm{per}}(\widehat{k\Delta})}=\mathbb{D}_{\rm{sg}}(\widehat{k\Delta}),$$
	where $\mathbb{D}^{\rm{per}}(\widehat{\La})$ and $\mathbb{D}^{\rm{per}}(\widehat{k\Delta})$ denote the category of prefect complexes of $\mathbb{D}^{\rm{b}}(\mmod \widehat{\La})$ and $\mathbb{D}^{\rm{b}}(\mmod \widehat{k\Delta})$, respectively.
	 Since $\mmod \widehat{\La}$ and $\mmod \widehat{k\Delta}$ are Frobenius categories, by the Buchweitz-Happel-Rickard's theorem we get the equivalence $\mathbb{D}_{\rm{sg}}(\widehat{\La})\simeq \rm{\underline{mod}}\mbox{-} \widehat{\La}$ and  $ \mathbb{D}_{\rm{sg}}(\widehat{k\Delta})\simeq \rm{\underline{mod}}\mbox{-} \widehat{k\Delta} $. Now  the obtained equivalence  $\rm{\underline{mod}}\mbox{-} \widehat{\La}\simeq \rm{\underline{mod}}\mbox{-} \widehat{k\Delta}$ follows the claim.\\
	 From Theorem \ref{Theorem 6.5} we have the following equivalent condition 
	$$ \CS(\mmod  \widehat{\La}) \ \text{is locally bounded} \ \Leftrightarrow \ \text{ so is}  \ \CF(\rm{\underline{mod}}\mbox{-} \widehat{\La}). \ \ \  \ \ \ \ \dagger\dagger$$
	The last equivalent condition we need is:
	$$ \CF(\mmod  \widehat{k\Delta}) \ \text{is locally bounded} \ \Leftrightarrow \ \text{ so is}  \ \mathbb{D}^{\rm{b}}(\mmod A). \ \ \  \ \ \ \ \dagger\dagger\dagger$$
	Proof of the above claim: since the path algebra  $k\Delta$ has of finite global dimension, by \cite[Theorem 4.9]{Ha} we have the equivalence $\mathbb{D}^{\rm{b}}(\mmod k\Delta)\simeq \rm{\underline{mod}}\mbox{-} \widehat{k\Delta}$. Hence $\CF(\mathbb{D}^{\rm{b}}(\mmod k\Delta))\simeq \CF(\rm{\underline{mod}}\mbox{-} \widehat{k\Delta}).$ Using this property that $\mathbb{D}^{\rm{b}}(\mmod k\Delta)$ (because of the equivalence $\mathbb{D}^{\rm{b}}(\mmod k\Delta)\simeq \rm{\underline{mod}}\mbox{-} \widehat{k\Delta}$ ) is locally bounded, we observe that  $\CF(\mathbb{D}^{\rm{b}}(\mmod k\Delta))\simeq \CF(\rm{\underline{mod}}\mbox{-} \widehat{k\Delta})$  is locally bounded if and only if so is the stable category $\underline{\CF(\mathbb{D}^{\rm{b}}(\mmod k\Delta) )}$. On the other hand, Lemma \ref{Lemma 7.1} implies the equivalence $\underline{\CF(\mathbb{D}^{\rm{b}}(\mmod k\Delta) )}\simeq \mathbb{D}^{\rm{b}}(\mmod A)$. Hence we get the claim.\\
 All the assertions  $(\dagger), (\dagger\dagger)$ and $(\dagger\dagger\dagger)$ prove the equivalence $(1)\Leftrightarrow (4)$. We are done.
\end{proof}
By the above result, the task of finding those representation-finite selfinjective algebras with  the associated monomorphism category of finite representation type (or equivalently with the stable Auslander algebra of finite representation type) is to classify those  representation-finite hereditary algebras whose their bounded derived categories are locally bounded.\\

\begin{corollary}\label{Corollary 7.3}
	Let $\La$, resp. $\La'$, be a tilted algebra of Dynkin type $\Delta$, resp. $\Delta'$, with the same underlying diagram, $\phi$, resp. $\phi'$,  is an admissible auto-equivalence of the repetitive category $\widehat{\La}$, resp. $\widehat{\La'}$.  Then $\CS(\mmod \widehat{\La}/\phi)$ is of finite representation type if and only if so is $\CS(\mmod \widehat{\La'}/\phi')$.
	\end{corollary}
\begin{proof}
Let $A$, resp. $A'$, be the stable Auslander algebra of the path algebra $k\Delta$, resp. $k\Delta'$.	We know by \cite[Page 166]{La},  asserting that  the stable Auslander algebras of the  path algebras of two quivers $\CQ$ and $\CQ'$ with  two orientations of a Dynkin diagram  are derived equivalent, $A$ and $A'$ are derived equivalent. Hence $\mathbb{D}^{\rm{b}}(\mmod A)$ is locally bounded if and only if so is $\mathbb{D}^{\rm{b}}(\mmod A')$. Proposition \ref{proposition 7.2} implies the claim.
\end{proof}
\begin{proposition} \label{proposition 7.4}
	Let $\Delta$ be a quiver of Dynkin
	 type $\mathbb{D}_n (n\geqslant 4), \mathbb{E}_6, \mathbb{E}_7$ and $\mathbb{E}_8$. Then the monomorphism  category $\CS(k\Delta)$ is of infinite representation type. In particular, the bounded derived category of the stable Auslander algebra of
	  $k\Delta$ is not locally bounded.
\end{proposition}
\begin{proof}
Let $A$ be the stable Auslander algebra of the path algebra $k\Delta$. By Remark \ref{remark 6.3} it suffices to prove that the algebra $A$ is of infinite representation type. We know by a general fact that the bounded quiver of $A$ is obtained by removing the projective vertices of the Auslander-Reiten quiver $\Gamma_{\mmod k\Delta}$ with mesh relations. In case that $\Delta$ is one of the type in the statement, one can see that 	 the stable Auslander-Reiten quiver of $k\Delta$ contains a subquiver as the following 
$$
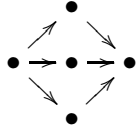
\xymatrix@-5mm{
	&&&&\bullet\ar[dr]\\
	&&&\bullet\ar[dr]\ar[r]\ar[ur]&\bullet\ar[r]&\bullet\\
	&&&&\bullet\ar[ur]&&	}
$$
In fact, it is induced by  an almost split sequence with no projective direct summand in the middle term, which exists, see e.g. \cite[Chapter VII]{ASS}. Therefore, the bounded quiver of $A$ contains the above bounded subquiver bt the fact mentioned in the beginning of the proof. Since the algebra defined by the above bound subquiver is clearly 
 of infinite representation type (e.g.  Happel-Vossieck list \cite{HV}), hence so  is $A$. The last assertion follows from the embedding $\mmod A\hookrightarrow \mathbb{D}^{b}(\mmod A)$.  We are done.
\end{proof}
Hence by the  above proposition we only need to focus on hereditary algebras with Dynkin type $\mathbb{A}_n.$\\
The below result  concerning the monomorphism categories over hereditary  algebras (of Dynkin type $\mathbb{A}$) is of independent interest, although we do not use it for our investigation.
\begin{proposition}\label{Proposition 7.5}
	Let $\Delta$ be a quiver of Dynkin
	type $\mathbb{A}_n$ with the linear orientation
	$$\Delta: 1\rt 2\rt \cdots\rt n. $$
	  Then the monomorphism  category $\CS(k\Delta)$ is of finite representation type if and only if $n\leqslant 5$.
\end{proposition}
\begin{proof}
	We know that the Auslander-Reiten quiver of $\Delta$ has the following shape
$$\xymatrixrowsep{16pt} \xymatrixcolsep{20pt}
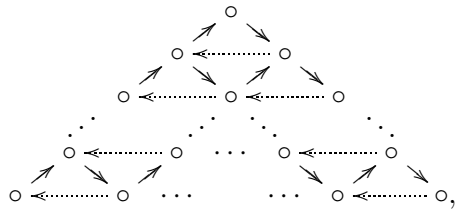
\xymatrix@!=0.1pt{&&&& \circ \ar[dr]&&&&\\
	&&&\circ \ar[ur]\ar[dr] \ar@{<.}[rr]&& \circ \ar[dr]&&&\\
	&& \circ\ar[ur] \ar@{<.}[rr] && \circ\ar[ur] \ar@{<.}[rr] && \circ
	&&\\
}$$ \vspace{-14pt}
$$\iddots \quad\;\; \quad \; \iddots \;\;\;  \ddots \;\quad \;\;\;\;\; \ddots$$
\vspace{-22pt}
$$\hspace{1pt}\xymatrixrowsep{16pt}
\xymatrixcolsep{20pt}\xymatrix@!=0.1pt{& \circ\ar[dr] \ar@{<.}[rr] && \circ   &\cdots& \circ\ar[dr]  \ar@{<.}[rr]  && \circ \ar[dr] & \\
	\circ \ar[ur]\ar@{<.}[rr] && \circ \ar[ur]  &\cdots&&\cdots&
	\circ\ar[ur] \ar@{<.}[rr]&& \circ,}\vspace{2pt}
$$
with $n$ vertices in the leftmost side. As mentioned in the proof of previous proposition, the  bounded quiver of the stable Auslander algebra $A$ of $k\Delta$ is obtained   by deleting  vertices in the leftmost side (projective vertices) of $\Gamma_{\mmod k\Delta}$, and with mesh relations. This fact in view of the structure of the Auslander-Reiten quivers of linear quivers implies that the algebra $A$ is isomorphic as algebras to the Auslander algebra of $k\Delta'$, where $\Delta'$  is obtained by $\Delta$ with deleting the sink vertex $n$. The characterization given in \cite{IPTZ} (see also \cite{LS}), for those representation-finite algebras  with the associated Auslander algebra of finite representation type,  implies that the stable Auslander algebra $A$ is of finite representation type (equivalently, $\CS(k\Delta)$ is of finite representation type)  if and only if  $n\leqslant 5$. So we are done.
\end{proof}
The above two results (Propositions \ref{proposition 7.4}, \ref{Proposition 7.5}) are taken from the first version of \cite{H} (available on arXiv) due to the first named author. These results are not included in the accepted  version of \cite{H}.
\begin{remark}\label{Remark 7.6}
Recall that a finite dimensional $k$-algebra $\La$ is a Nakayama algebra if any indecomposable is uniserial, i.e. it has a unique composition series (\cite{AuslanreitenSmalo}, p.197). In this case $\La$ is representation-finite. Every connected self-injective Nakayama algebra is Morita equivalent to $\La(n, t), n \geqslant 1, t \geqslant 2$ (\cite{GR}, p.243), which is defined  as follows. Let $C_n$ be the cyclic quiver with $n\geqslant 1 $ vertices  and $J$ the ideal generated by all arrows, and then for any $ t \geqslant 2$, set $\La(n, t):= kC_n/J^t$. By the structure of almost split sequences of $\mmod\La(n, t)$   given in \cite{AuslanreitenSmalo}[VI.2], we can deduce  $\Gamma^s_{\La(n, t)}=\mathbb{Z}A_{t-1}/<\tau^n>$, where $A_{t-1}$ is a linear quiver with $t-1$ vertices. 	
\end{remark}

\begin{proposition}
	Let $\Delta$ be a quiver of Dynkin
	type $\mathbb{A}_n$, and let $A$ be the stable Auslander algebra of $k\Delta$. Then the bounded derived category $\mathbb{D}^{\rm{b}}(\mmod A)$ is locally  bounded if and only if $n\leqslant 4$.
\end{proposition}
\begin{proof}
Thanks to \cite[Page 166]{La}, we may assume that $$\Delta=A_n: 1\rt 2\rt \cdots\rt n. $$ Set $B_n$ the stable Auslander algebra of $kA_n$. As mentioned in Introduction, it was proved in \cite{RS1} that the monomorphism category $\CS(\mmod k[x]/(x^n))$ is of finite representation type (or equivalently the stable Auslander algebra of $k[x]/(x^n)$ is representation-finite) if and only if $n\leqslant 5$. In view of Remark \ref{Remark 7.6}, $\Gamma^s_{k[x]/(x^n)}=\mathbb{Z}A_{n-1}/<\tau>$. Proposition  \ref{proposition 7.2} implies:
$$ \CS(\mmod k[x]/(x^n)) \ \text{is of finite represenation type}  \Leftrightarrow \mathbb{D}^{b}(\mmod B_{n-1}) \ \text{is locally bounded}. $$
Now the above-mentioned facts imply the assertion.
\end{proof}
Now the results established in the above imply our main result in this section.
\begin{theorem}
	Let  $\La$ be  a tilted algebra of Dynkin type $\Delta$, $\phi$ is an admissible auto-equivalence of the repetitive category $\widehat{\La}$. Suppose $A$, resp. $\Gamma,$ is the  stable Auslander algebra of the path algebra $k\Delta$, resp.  representation-finite selfinjective algebra $\widehat{\La}/\phi$.Then the following conditions are equivalent.
	\begin{itemize}
		\item [$(1)$] The monomorphism category $\CS(\mmod \widehat{\La}/\phi)$ is of finite representation type.
		\item [$(2)$] The monomorphism category $\CS(\mmod \widehat{\La})$ is of finite representation type.
		\item [$(3)$] The stable Auslander algebra $\Gamma$ is representation-finite.
		\item [$(4)$] The bounded derived category $\mathbb{D}^{b}(\mmod A)$ is locally bounded.
		\item [$(5)$] The quiver $\Delta$ is of Dynkin type $\mathbb{A}_n$ with $n \leqslant 4.$
	\end{itemize}
\end{theorem}
 
 \begin{corollary}\label{Corollary 7.8}
 Keep the notation used in Remark \ref{Remark 7.6}. Then the monomorphism category $\CS(\mmod \La(n, t))$ is of finite representation type if and 
 only if $t \leqslant 5.$
 \end{corollary}
\begin{proof}
	We observe by Remark \ref{Remark 7.6} the type of the selfinjective Nakayam algebra $\La(n, t)$ is $\mathbb{A}_{t-1}$. Now the statement is a direct consequence of our main theorem.
\end{proof}
\section*{acknowledgment}
The authors would like to thank Dr. Mohsen Asgharzadeh for carefully reading the manuscript.
This research was supported in part by a grant from School of Mathematics,Institute for Research in Fundamental Sciences (IPM).  The second author also thank the Center of Excellence for Mathematics (University of  Isfahan).


\begin{thebibliography}{9999}
\bibitem[AHS]{AHS} {\sc J. Asadollahi, R. Hafezi and S. Sadeghi,}	{\sl On the Monomorphism Category of $n$-Cluster Tilting Subcategories,} avaliable on arXiv:2008.04178.
 \bibitem[A1]{As} {\sc H. Asashiba,} {\sl A covering technique for derived equivalence.,}  J. Algebra {\bf 191} (1997), no. 1, 382-415.

\bibitem[A2]{A} {\sc H. Asashiba,} {\sl  A generalization of Gabriel's Galois covering functors and derived equivalences,}  J. Algebra {\bf 334} (2011), 109-149.
\bibitem[ASS]{ASS} {\sc I. Assem, D. Simson and  A. Skowro\'nski,} {\sl 
Elements of the representation theory of associative algebras,}  Vol. 1.
Techniques of representation theory. London Mathematical Society Student Texts, 65. Cambridge University Press, Cambridge, 2006. x+458 pp. ISBN: 978-0-521-58423-4; 978-0-521-58631-3; 0-521-58631-3. 
\bibitem[Au1]{Au1} {\sc M. Auslander,} {\sl Coherent functors,} in Proc. Conf. Categorical Algebra (La Jolla, Calif., 1965), 189-231, Springer, New York, 1966.

\bibitem[Au2]{Au} {\sc	M. Auslander,} {\sl Representation theory of artin algebras I,} Communication in Algebra {\bf 1} (1974), pp. 177-268.
\bibitem[Au3]{Au3} {\sc M. Auslander,} {\sl  A functorial approach to representation theory,} Representations of algebras (Puebla, 1980), pp. 105-179, Lecture Notes in Math., 944, Springer, Berlin-New York, 1982. 
\bibitem[Au4]{Au4} {\sc  M. Auslander,} {\sl The  representation  dimension  of  artin  algebras,} Queen Mary College Mathematics Notes (1971). Republished in Selected works of Maurice Auslander. Amer. Math. Soc., Providence 1999.
\bibitem[AR1]{Ar1}  {\sc M. Auslander and  I. Reiten,} {\sl Representation theory of artin algebras IV,} Comm. Algebra {\bf 5} (1977) 443-518.
\bibitem[AR2]{Ar2} {\sc M. Auslander and  I. Reiten,} {\sl On the representation type of triangular matrix rings,} J. Lond. Math. Soc. (2){\bf 12}(3) (1976), 371-382.	
\bibitem[AR3]{AR3} {\sc M. Auslander and I. Reiten,} {\sl Stable equivalence of dualizing $R$-varieties,} Adv. Math. {\bf 12} 306-366, 1974.
\bibitem[ARS]{AuslanreitenSmalo} {\sc M. Auslander, I. Reiten and S.O. Smal\o,} {\sl Representation theory of Artin algebras,} Cambridge Studies in Advanced Mathematics, 36. Cambridge University Press, Cambridge, 1995. xiv+423 pp. ISBN: 0-521-41134-3.
\bibitem [AS]{ASm} {\sc M. Auslander and S.O. Smal{\o},} {\sl Almost split sequences in subcategories,} J. Algebra {\bf 69}(2) (1981), 426-454.
\bibitem[BL]{BL} {\sc R. Bautista and S. Liu,} {\sl  Covering theory for linear categories with application to derived categories,} J. Algebra {\bf 406} (2014), 173-225.
\bibitem[Bi]{Bi} {\sc G. Birkhoff,} {\sl  Subgroups of abelian groups,}  Proc. Lond. Math. Soc. II, Ser. {\bf38}(1934), 385-401.
\bibitem[BLR]{BLR} {\sc O. Bretscher, C. L{\"a}ser and C. Riedtmann,}  {\sl Self-injective and simply connected algebras},  Manuscripta Math, {\bf36(3)}: 253-307, 1981/82.
\bibitem[BG]{BG} {\sc K. Bongartz and P. Gabriel,} {\sl  Covering spaces in representation-theory,} Invent. Math. {\bf 65} (1981/82) 331-378.
\bibitem[DI]{DI} {\sc E. Darp{\"o} and  O. Iyama,} {\sl  d-representation-finite self-injective algebras,} Adv. Math. {\bf 362} (2020), 106932.
\bibitem[DL]{DL} {\sc V. Dlab and C.M. Ringel,} {\sl The module theoretical approach to quasi-hereditary algebras,} In: Representations of Algebras and Related Topics (ed. H. Tachikawa and S. Brenner). London Math. Soc. Lecture Note Series168. Cambridge University Press (1992), 200-224.
\bibitem[DLS]{DLS} {\sc P. Dowbor, H. Lenzing and A. Skowro\'nski,} {\sl Galois coverings by algebras of locally support-finite categories,} (Ottawa, Ont., 1984), 91-93, Lecture Notes in Math., 1177, Springer, Berlin, 1986. 
	
\bibitem[DS]{DS} {\sc P. Dowbor and A. Skowro\'nski,} {\sl Galois coverings of representation-infinite algebras,} Comment. Math. Helv. {\bf 62}, no. 2 (1987), 311-337.
 
\bibitem[E]{E} {\sc \"{O}. Eir\'iksson,} {\sl From submodule categories to the stable Auslander algebra,} J. Algebra {\bf 486} (2017), 98-118.

\bibitem[EHHS]{EHHS} {\sc H. Eshraghi, R. Hafezi, E.  Hosseini and Sh. Shokrollah,} {\sl  Cotorsion theory in the category of quiver representations,}  J. Algebra Appl. {\bf 12} (2013), no. 6, 1350005, 16 pp. 

	\bibitem[F]{F} {\sc P. Freyd,} {\sl  Representations in abelian categories,}  In Proc. Conf. Categorical Algebra (La Jolla, Calif., 1965), pages 95-120. Springer, New York, 1966.	
\bibitem[G]{G} {\sc P. Gabriel,} {\sl  The universal cover of a representation-finite algebra,} in: Representations of Algebras (Puebla, 1980), Lecture Notes in Math., 903, Springer, 1981,
68-105.
\bibitem [GRi]{GR1}{\sc P. Gabriel, and C. Riedtmann,} {\sl  Group representations without groups,} Comment. Math. Helv. {\bf 54} (1979), 240-287.
	
\bibitem[GRo]{GR} {\sc P. Gabriel and A.V. Roiter,} {\sl Representations of Finite-Dimensional Algebras,} Encyclopedia of Mathematical Sciences, Vol. 73, Springer-Verlag, Berlin/New York, 1992.
\bibitem[Gr]{Gr} {\sc  E.L. Green,} {\sl  Graphs with relations, coverings and group-graded algebras,} Trans. Amer. Math. Soc.  {\bf 279} (1983), 297-310.

\bibitem[H]{H} {\sc R. Hafezi,} {\sl From subcategories to the entire module categories,}  Forum Math. {\bf 33} (2021), no. 1, 245-270.

\bibitem[Ha]{Ha} {\sc D. Happel,} {\sl Triangulated categories in the representation theory of finite-dimensional algebras,} London Mathematical Society Lecture Note Series, 119. Cambridge University Press, Cambridge, 1988.
\bibitem[HV]{HV}  {D. Happel and V. Dieter,} {\sc  Minimal algebras of infinite representation type with preprojective component'} Manuscripta Math. {\bf 42} (1983), no. 2-3, 221-243. 

	\bibitem[IPTZ]{IPTZ} {\sc K. Igusa, M. Platzeck, G. Todorov and D. Zacharia,} {\sl Auslander algebras of finite representation type,} Comm. Algebra {\bf 15} (1987), no. 1-2, 377-424.
\bibitem[IO]{IO} {\sc O. Iyama and  S. Oppermann,} {\sl  Stable categories of higher preprojective algebras,} Adv. Math. {\bf 244} (2013), 23-68.

\bibitem[Ki]{Ki} {\sc Y. Kimura,} {\sl Singularity categories of derived categories of hereditary algebras are derived categories,} 
J. Pure Appl. Algebra {\bf 224} (2020), no. 2, 836-859. 
\bibitem[K]{K} {\sc H. Krause,} {\sl  Krull-Schmidt categories and projective covers,} Expo. Math. {\bf 33} (2015), no. 4, 535-549.
\bibitem[KS]{KS}{\sc 	H. Krause and {\o}. Solberg,} {\sl Applications of cotorsion pairs,} J. London Math.Soc. (2) {\bf 68} (2003), 631-650.	

\bibitem[KLM]{KLM} {\sc D. Kussin, H. Lenzing and H. Meltzer,} {\sl Nilpotent  operators and weighted  projective lines,} J. Reine Angew. Math. {\bf685}(2013), 33-71
\bibitem[K]{K} {\sc H. Krause,} {\sl  Krull-Schmidt categories and projective covers,} Expo. Math. {\bf 33} (2015), no. 4, 535-549.
\bibitem[La]{La} {\sc S. Ladkani,} {\sl On derived equivalences of lines, rectangles and triangles,} 
J. Lond. Math. Soc. (2) {\bf 87} (2013), no. 1, 157-176.
\bibitem[LSi]{LS} {\sc Z. Leszczy\'nski and D. Simson,} {\sl  On triangular matrix rings of finite representation type,}  J. London Math. Soc. (2) {\bf 20} (1979), no. 3, 396-402.
\bibitem[LSk]{LSk} {\sc Z. Leszczy\'nski and A. Skowro{\'n}ski,} {\sl  Tame triangular matrix algebras,} Colloq. Math. {\bf 86} (2000), no. 2, 259-303.
\bibitem[Li]{L} {\sc S. Liu,} {\sl  Auslander-Reiten theory in a Krull–Schmidt category,} Sao Paulo J. Math. Sci. {\bf 4} (2010) 425-472.
\bibitem[LZ]{LZ} {\sc X-H. Luo and P. Zhang,} {\sl  Separated monic representations I: Gorenstein-projective modules,} J. Algebra {\bf 479} (2017), 1-34.
\bibitem[MD]{MD} {\sc R. Mart\'inez-Villa and J.A. de la Pe\~na,} {\sl  The universal cover of a quiver with relations,} J. Pure Appl. Algebra {\bf 30} (1983), no. 3, 277-292.
\bibitem[MO]{MO} {\sc R. Mart\'inez-Villa and M. Ortiz-Morales,} {\sl Tilting theory and functor categories III. The maps category,}  Int. J. Algebra {\bf 5} (2011), no. 9-12, 529-561.


\bibitem[P]{P} {\sc G. Pastuszak,} {\sl  On Krull-Gabriel dimension and Galois coverings,}  Adv. in Math. {\bf 349} (2019), 959-991.
\bibitem[Ra]{Ra} {\sc B. Raymundo,} {\sl  
The category of morphisms between projective modules,} Comm. Algebra {\bf 32} (2004), no. 11, 4303-4331.
\bibitem[R]{R} {\sc C.M. Ringel,} {\sl Finite-dimensional hereditary algebras of wild representation type,} Math. Z. {\bf161} (1978) 236-255.

\bibitem[Ri1]{Ri3} {\sc  C. Riedtmann,} {\sl  Algebren, Darstellungsk{\"o}cher, {\"U}berlagerungen and zurück,} Comment. Math. Helv. 55 (1980) 199-224.
\bibitem[Ri2]{Ri1} {\sc C. Riedtmann,} {\sl   Representation-finite self-injective algebras of class $A_{n}$,}  In  Representation theory, {II} ({P}roc. {S}econd {I}nternat.{C}onf., {C}arleton {U}niv., {O}ttawa, {O}nt., 1979), Lecture Notes in Math. 832, pages 449-520. Springer, Berlin, 1980.
\bibitem[Ri3]{Ri2} { \sc C. Riedtmann,} { \sl Representation-finite self-injective algebras of class $D_{n}$,}  Compos. Math. {\bf 49(2)}: 231-282, 1983.
\bibitem[RS1]{RS1} {\sc C.M. Ringel and M. Schmidmeier,} {\sl  Invariant subspaces of nilpotent linear operators,} I. J. Reine Angew. Math. {\bf 614} (2008), 1-52. 

\bibitem[RS2]{RS2} {\sc C.M. Ringel and M. Schmidmeier,} {\sl The Auslander-Reiten translation in submodule categories,} Trans. Amer. Math. Soc. {\bf 360} (2008), no. 2, 691-716.
\bibitem[RZ]{RZ} {\sc C. Ringel and P. Zhang,} {\sl From  submodule  categories  to  preprojective  algebras,} Math. Z, {\bf 278}(1-2)(2014), 55-73.

\bibitem[Sc]{Sc}{\sc M. Schmidmeier,} {\sl  Bounded submodules of modules,}  J. Pure Appl. Algebra {\bf 203} (2005), no. 1-3, 45-82.

\bibitem[Si]{Si} {\sc D. Simson,} {\sl Chain categories of modules and subprojective representations of posets over uniserial algebras,} Rocky Mountain J. Math. {\bf32}(2002), 1627-1650.

\bibitem[S]{skowronski06} {\sc A. Skowro\'nski,} {\sl  Selfinjective algebras: finite and tame type,}
 In  Trends in representation theory of algebras and related 	topics, volume 406 of  Contemp. Math., pages 169-238. Amer. Math. Soc., Providence, RI, 2006.

\bibitem [SY]{SY} {\sc A. Skowro{\'n}ski and K. Yamagata,} 
 {\sl Selfinjective algebras of quasitilted type,}  In  Trends in representation theory of algebras and related topics, EMS Ser. Congr. Rep., pages 639--708. Eur. Math. Soc., Z\"urich, 2008.
\bibitem[T]{T} {\sc  G. Todorov,} {\sl  Almost split sequences for TrD-periodic modules,} in: Representation Theory II, in: Lecture Notes in Math., vol. 832, Springer-Verlag, Berlin, Heidelberg, 1980, pp. 600-631.
\bibitem[W]{W}
{\sc R.B. Warfield,} {\sl  A Krull-Schmidt theorem for infinite sums of
modules,} Proc. Amer. Math. Soc. {\bf 22}, 1969, 460-465.


\bibitem[XZZ]{XZZ}{\sc B-L. Xiong, P. Zhang and  Y-H. Zhang,} {\sl Auslander-Reiten translations in monomorphism categories,} Forum Math. {\bf 26} (2014), no. 3, 863-912.
\bibitem[ZX]{XZ}{\sc P. Zhang  and B-L. Xiong,}  {\sl Separated monic representations II: Frobenius subcategories and RSS equivalences,} Trans. Amer. Math. Soc. {\bf 372} (2019), no. 2, 981-1021.

\end{thebibliography}
\end{document}